\documentclass{article}
\usepackage{prelude}

\newkeytheorem{definition}[name = Definition, numberwithin=section]
\newkeytheorem{example}[name = Example, style = example, numberwithin=section]
\newkeytheorem{theorem}[name = Theorem, style = theorem, numberwithin=section]
\newkeytheorem{conjecture}[name = Conjecture, style = theorem, numberwithin=section]
\newkeytheorem{proposition}[name = Proposition, style = theorem, numberwithin=section]
\newkeytheorem{corollary}[name = Corollary, style = theorem, numberwithin=section]
\newkeytheorem{lemma}[name = Lemma, style = theorem, numberwithin=section]
\newkeytheorem{proof}[name = Proof, style = proof, numbered=no]
\newkeytheorem{remark}[name = Remark, style = remark]

\usepackage{multicol}

\usetikzlibrary{hobby}
\usepackage[outline]{contour}
\contourlength{1pt}

\renewcommand{\L}{\mathcal L}
\newcommand{\K}{\mathcal K}
\newcommand{\W}{\mathcal W}
\newcommand{\E}{\mathcal E}

\DeclareMathOperator{\Kappa}{K}
\let\H\relax 
\DeclareMathOperator{\H}{H}
\DeclareMathOperator{\conv}{conv}
\DeclareMathOperator{\epi}{epi}
\newcommand{\dTheta}{\mathop{d_\Theta}\nolimits}

\renewcommand{\k}{\kappa}
\renewcommand{\t}{\theta}

\usepackage[x11names]{xcolor}

\newcommand{\InitializeExtendedComplexPlane}[1]{%
    \pgfmathtruncatemacro{\ComplexPlaneSides}{16}%
    \pgfmathtruncatemacro{\ComplexPlaneLastVertex}{\ComplexPlaneSides-1}%
    \pgfmathsetmacro{\ComplexPlaneRadius}{#1}%
    \pgfmathsetmacro{\ComplexPlaneRotation}{180/\ComplexPlaneSides}%
    \pgfmathsetmacro{\ComplexPlaneAxisHit}{\ComplexPlaneRadius*cos(\ComplexPlaneRotation)}%
    \colorlet{sectorcolor}{red}%
    \coordinate (CPOrigin) at (0,0);%
    \coordinate (CPEastInfinity) at (\ComplexPlaneAxisHit,0);%
    \coordinate (CPWestInfinity) at (-\ComplexPlaneAxisHit,0);%
    \coordinate (CPNorthInfinity) at (0,\ComplexPlaneAxisHit);%
    \coordinate (CPSouthInfinity) at (0,-\ComplexPlaneAxisHit);%
}

\newcommand{\SetExtendedComplexPlaneStandardEdges}{%
    \def\ComplexPlaneDrawPolyEdge##1##2##3{%
        \foreach \EdgeIndex in {0,...,5} {%
            \pgfmathsetmacro{\EdgeA}{0.34*(\EdgeIndex/6)^1.5}%
            \pgfmathsetmacro{\EdgeB}{\EdgeA + 0.06*(\EdgeIndex/6)}%
            \draw[##1] ($(##2)!\EdgeA!(##3)$) -- ($(##2)!\EdgeB!(##3)$);%
            \draw[##1] ($(##2)!{1-\EdgeB}!(##3)$) -- ($(##2)!{1-\EdgeA}!(##3)$);%
        }%
        \draw[##1] ($(##2)!0.34!(##3)$) -- ($(##2)!0.66!(##3)$);%
    }%
}

\newcommand{\DrawExtendedComplexPlaneRay}[3][]{%
    \pgfmathsetmacro{\DashGap}{4}%
    \pgfmathtruncatemacro{\DashCount}{7}%
    \pgfmathtruncatemacro{\LastDash}{\DashCount-1}%
    \coordinate (RayBreak) at ($(CPOrigin)+({#3*\ComplexPlaneAxisHit*cos(#2)},{#3*\ComplexPlaneAxisHit*sin(#2)})$);
    \coordinate (RayEnd) at ($(CPOrigin)+({\ComplexPlaneAxisHit*cos(#2)},{\ComplexPlaneAxisHit*sin(#2)})$);
    \draw[#1] (CPOrigin) -- (RayBreak);
    \path let \p1 = (RayBreak), \p2 = (RayEnd), \n1 = {veclen(\x2-\x1,\y2-\y1)}
        in \pgfextra{\xdef\DashSpan{\n1}};
    \pgfmathsetmacro{\DashMax}{2*(\DashSpan-\LastDash*\DashGap)/\DashCount}%
    \foreach \DashIndex in {0,...,\LastDash} {%
        \pgfmathsetmacro{\DashLength}{\DashMax*(\LastDash-\DashIndex)/max(1,\LastDash)}%
        \pgfmathsetmacro{\DashStart}{\DashIndex*\DashGap+\DashMax*\DashIndex-\DashMax*\DashIndex*(\DashIndex-1)/(2*max(1,\LastDash))}%
        \pgfmathsetmacro{\DashStop}{min(\DashStart+\DashLength,\DashSpan)}%
        \pgfmathparse{\DashLength>0 && \DashStart<\DashSpan ? 1 : 0}%
        \ifnum\pgfmathresult=1
            \pgfmathsetmacro{\DashStartFrac}{\DashStart/\DashSpan}%
            \pgfmathsetmacro{\DashStopFrac}{\DashStop/\DashSpan}%
            \draw[#1] ($(RayBreak)!\DashStartFrac!(RayEnd)$)
                -- ($(RayBreak)!\DashStopFrac!(RayEnd)$);
        \fi
    }%
}

\newcommand{\DrawExtendedComplexPlaneAxesStandard}{%
    \foreach \ComplexPlaneAxisAngle in {0,90,180,270} {%
        \DrawExtendedComplexPlaneRay{\ComplexPlaneAxisAngle}{0.46}%
    }%
}

\newcommand{\DrawExtendedComplexPlaneLabelsStandard}{%
    \node[above left] at (CPOrigin) {$0$};
    \node[right] at (CPEastInfinity) {$\infty$};
    \node[left] at (CPWestInfinity) {$-\infty$};
    \node[above] at (CPNorthInfinity) {$i\infty$};
    \node[below] at (CPSouthInfinity) {$-i\infty$};
}

\newcommand{\CreateExtendedComplexPlaneVertices}{%
    \foreach \VertexIndex in {0,...,\ComplexPlaneLastVertex} {
        \coordinate (CP\VertexIndex) at ({\ComplexPlaneRadius*cos(\ComplexPlaneRotation+360*\VertexIndex/\ComplexPlaneSides)},
            {\ComplexPlaneRadius*sin(\ComplexPlaneRotation+360*\VertexIndex/\ComplexPlaneSides)});
    }
}

\newcommand{\DrawExtendedComplexPlaneCirclesStandard}{%
    \foreach \CircleIndex in {1,...,20} {
        \pgfmathsetmacro{\CircleT}{\CircleIndex/20}
        \pgfmathsetmacro{\CircleRadius}{0.97*\ComplexPlaneRadius*(2*\CircleT - \CircleT*\CircleT)}
        \pgfmathsetmacro{\CircleOpacity}{0.115 - 0.11*\CircleT}
        \draw[gray,opacity=\CircleOpacity] (CPOrigin) circle (\CircleRadius);
    }
}

\newcommand{\DrawExtendedComplexPlaneBoundary}{%
    \foreach \VertexIndex in {0,...,\ComplexPlaneLastVertex} {
        \pgfmathtruncatemacro{\NextVertex}{mod(\VertexIndex+1,\ComplexPlaneSides)}
        \ComplexPlaneDrawPolyEdge{}{CP\VertexIndex}{CP\NextVertex}
    }
}

\newcommand{\DrawExtendedComplexPlaneInfinityDots}{%
    \fill (CPEastInfinity) circle (1.1pt);
    \fill (CPWestInfinity) circle (1.1pt);
    \fill (CPNorthInfinity) circle (1.1pt);
    \fill (CPSouthInfinity) circle (1.1pt);
}

\newcommand{\DrawExtendedComplexPlane}[2][]{%
    \InitializeExtendedComplexPlane{3.1}%
    \SetExtendedComplexPlaneStandardEdges
    #1%
    \DrawExtendedComplexPlaneAxesStandard
    \DrawExtendedComplexPlaneLabelsStandard
    \CreateExtendedComplexPlaneVertices
    \DrawExtendedComplexPlaneCirclesStandard
    \DrawExtendedComplexPlaneBoundary
    \DrawExtendedComplexPlaneInfinityDots
    #2%
}

\title{The indicator diagram theory for maximal type entire functions and the characterization of analytic continuability}
\author{Kei Beauduin}
\date{}

\begin{document}

\maketitle

\begin{abstract}
    We extend Pólya's indicator diagram theory to encompass entire functions of order at most $1$, allowing functions of maximal type. To do so, we introduce an extension of the complex plane in which indicator diagrams may be unbounded or even purely infinite. In this framework, we establish the Laplace and inverse Laplace transforms, and prove an analog of Pólya's theorem. As an application, we extend a characterization of the analytic continuability of power series. The resulting theory is illustrated by a range of classical examples of special functions.
\end{abstract}

\setcounter{tocdepth}{3}\tableofcontents

\section{Introduction}

\subsection{Overview}

Pólya's indicator diagram theory gives a geometric description of the growth of an entire function of exponential type in terms of the singularities of its \emph{Laplace transform}. If $\phi$ is entire of exponential type, then its \emph{indicator function}
\begin{equation}\label{e:hintro}
    h_\phi(\t) := \limsup_{r\to\infty} \frac{\log|\phi(re^{i\t})|}{r}
\end{equation}
describes the exponential growth of $\phi$ along rays in the complex plane $\C$. On the other hand, a closed convex subset $K$ of $\C$ is characterized by its \emph{support function}
\begin{equation}\label{e:kappaintro}
    \k_K(\t) := \sup_{p\in K}\Re(pe^{-i\t}).
\end{equation}
The \emph{conjugate diagram} is the convex hull of the singularities of the Laplace transform of $\phi$, whose complex conjugate is called the \emph{indicator diagram}, denoted by $\K_\phi$. The classical \emph{Pólya theorem} states the equality
\begin{equation}\label{e:polyaintro}
    h_\phi(\t) = \k_{\K_\phi}(\t),
\end{equation}
for all $\t$.

Entire functions of exponential type are precisely the entire functions of \emph{order} at most $1$ and finite \emph{type}, and their corresponding indicator diagrams are necessarily bounded. The aim of this paper is to extend this theory to entire functions of order at most $1$ that may have \emph{infinite} (or \emph{maximal}) type, and thus possibly unbounded indicator diagrams. Typical examples are
\begin{equation}\label{e:phi123}
    s \longmapsto \inv{\Gamma(s+1)},\quad \frac{-\zeta(-s)}{\Gamma(s+1)}, \quad s\zeta(1+s),
\end{equation}
where $\Gamma$ is the Gamma function and $\zeta$ is the Riemann zeta function. For $\phi(s)=1/\Gamma(s+1)$, Stirling's approximation yields
\begin{equation}\label{e:hRGammaintro}
    h_\phi(\t) =
    \begin{cases}
        -\infty &\com{if} \t \in \pa{-\frac\pi{2},\frac\pi{2}}, \\
        \frac\pi{2} &\com{if} \t = \pm \frac\pi{2}, \\
        \infty &\com{otherwise}.
    \end{cases}
\end{equation}
Thus, if \cref{e:polyaintro} is to remain valid, the indicator diagram $\K_\phi$ cannot merely be an unbounded subset of $\C$; it must contain points at infinity, formally represented here by "$-\infty + i[-\frac\pi2,\frac\pi2]$".

\medskip

To formalize this phenomenon, we introduce in \Cref{s:Cbar} an extension $\bar\C$ of the complex plane, distinct from the Riemann sphere, in which unbounded and purely infinite indicator diagrams can still be treated as closed convex sets. The corresponding convex geometry is developed in \Cref{s:cvxset,s:cvxanalysis}: in particular, we extend support functions to subsets of $\bar\C$, relate them to convex hulls, and characterize them by \emph{trigonometric convexity}.

After collecting the necessary growth estimates for entire functions of order at most $1$ in \Cref{s:growth1,s:growth2}, we develop the indicator diagram theory in \Cref{s:IDT}. The central results are the definition of the Laplace transform via \Cref{t:WN}, a Watson--Nevanlinna-type theorem, the inversion formula \Cref{t:invlapl}, and the extension of Pólya's theorem in \Cref{t:polya}.

We then apply the theory to analytic continuation problems. In \Cref{t:AC} we extend a theorem of Carlson, Dufresnoy, and Pisot characterizing the analytic continuability of power series. Finally, \Cref{s:ram1} derives a version of Ramanujan's master theorem within the same framework.

The exposition is intended to be as self-contained as possible, although some familiarity with the standard language of complex analysis and convexity will be helpful. The longer convex-analytic proofs, together with several supplementary results, are collected in the appendix.

\subsection{Historical background and survey}

In 1896, Borel introduced the resummation method now known as \emph{Borel summation} \cite{borel1896,borel1899}, together with the \emph{Borel transform} and the notion of \emph{summability polygon} \cite[Chap.~VIII]{hardy1949}. Soon afterward, Servant gave a description of Borel's polygon in terms of the growth of a function along rays in $\C$ \cite{servant1899,servant1899a}. This led Borel to formulate two questions about these growth quantities in 1901 \cite[\S 72]{borel1928}. In 1908, Phragmén and Lindelöf rediscovered the same quantities as the values of the indicator function \eqref{e:hintro}, and thereby resolved Borel's first problem.

A major conceptual advance came in Chapter 2 of Pólya's 1929 memoir \cite{polya1929}. There, Pólya developed a simple and elegant framework for the study of entire functions $\phi$ of exponential type, in which a central role was played by the Phragmén--Lindelöf indicator function. Pólya used it to define the Laplace transform $\Phi$ of $\phi$ outside a bounded set containing its singularities; he called the convex hull of the complex conjugates of these singularities the indicator diagram $\K_\phi$. By adopting this point of view, Pólya refined Borel's approach. His framework was not only geometric, but also \emph{convex}-geometric, which allowed him to bring in the support function \eqref{e:kappaintro} introduced by Minkowski in his foundational 1903 work on convex bodies \cite{minkowski1903}. Within this framework, Pólya resolved Borel's second problem and, most importantly, proved the Pólya theorem \eqref{e:polyaintro}.

Pólya's approach was rapidly adopted and became a standard reference in the study of entire functions: clear expositions of the foundational theory are given in the books \cite{bernstein1933,doetsch1937,chebotarev1949,doetsch1950,boas1954,bieberbach1955,cartwright1956,levin1964,evgrafov1961,henrici1991,leontev1983,levin1996,rubel1996,maergoiz2003}.

Subsequent work aimed to broaden the class of functions $\phi$ and to adapt both the underlying theory and Pólya's theorem to the wider setting. The first extensions concerned entire functions of order $\rho\ge0$ and finite type, that is, functions satisfying a bound of the form $A e^{B |s|^\rho}$ with $A,B\ge0$; the case $\rho=1$ is precisely exponential type. The notion of \emph{proximate order} further generalizes this framework by allowing $\rho$ to depend on $|s|$ in a suitably convergent manner (see, e.g., \cite{levin1996}).

In 1933 and 1936, Bernstein obtained two distinct extensions. The first \cite{bernstein1933a}, formulated in terms of proximate order, improves the earlier order-$\rho$ framework of \cite{subbotin1931}, but does not provide a complete analog of Pólya's theorem: for functions with exponential decay in some direction, \cref{e:polyaintro} applies only to the positive cutoff $\max(0,h_\phi)$ of the indicator. The same approach was later rediscovered in the 1950s in the order-$\rho$ setting \cite{lokhin1955,avetisyan1955,evgrafov1961} (see also \cite{dzhrbashyan1966,maergoiz1984,maergoiz2003}), and in the 1970s in the proximate-order setting \cite{evgrafov1976,maergoiz1978,maergoiz1980}.

Bernstein's second approach \cite{bernstein1936} recasts Pólya's theory for entire functions defined on suitable Riemann surfaces, allowing a function $\phi(s)$ of order $\rho$ to be studied via $\phi(s^{1/\rho})$, which has order $1$. This idea was independently discovered earlier by Pfluger \cite{pfluger1933,pfluger1935} and was later developed further by Maergoiz \cite{maergoiz1975,maergoiz1987,maergoiz2003,maergoiz2018,maergoiz2022}.

In the 1950s Ivanov and Stavskij extended Pólya's theorem to several complex variables, but only in a componentwise manner \cite{ivanov1957,ivanov1962,stavskij1963,stavskij1961}. The corresponding higher-order extensions followed in the 1960s: the two-dimensional and $n$-dimensional order-$\rho$ theories are due, respectively, to Kobeleva and Ronkin \cite{kobeleva1962,ronkin1963,ronkin1974}, and Geche treated the two-dimensional proximate-order case \cite{geche1965,geche1968}. Radically different ideas were required for the complete $n$-dimensional theorem, a case achieved by Martineau \cite{martineau1963} (see also \cite{martineau1967,lelong1968,maergoiz2003}).

Pólya's theory has since found numerous applications in the theory of entire functions; see, for example, \cite{boller1932,gelfond1938,matison1938,boas1954,bieberbach1955,komatsu1987,levin1996}. One application, especially relevant here, is the refinement by Dufresnoy and Pisot \cite{dufresnoy1951} of Carlson's characterization of the analytic continuability of power series outside a bounded region \cite{carlson1914,carlson1921}; see \cite[\S 1.3]{bieberbach1955}.

Carlson had also proved in his thesis \cite{carlson1914} a second analytic-continuation theorem, concerning continuation outside an unbounded region and refining earlier work. The sufficient part of that theorem was later presented in detail by Hille \cite[Sec.~11.3]{hille1962}, who used P\'olya's terminology and introduced \emph{unbounded} indicator diagrams. The indicator-diagram theory underlying this unbounded setting was subsequently developed by Henrici \cite[Sec.~10.9--10.10]{henrici1991}.

\subsection{Scope of the extension}

The present paper continues this latter direction. More precisely, our extension is guided by the following hierarchy which reflects successive settings in which the indicator diagram theory can be developed, each step bringing new phenomena and considerable difficulties.
\begin{itemize}
    \item $\E_0$: the singleton containing the zero function.
    \item $\E_1$: the set of \emph{exponential sums}, i.e., finite sums of terms $p(s)e^{as}$ where $p$ is a polynomial and $a$ is a complex number.\footnote{This is a significant subject in its own right, whose study was also initiated by P\'olya; see the recent survey \cite{heittokangas2023}.}
    \item $\E_2$: the set of entire functions of exponential type.
    \item $\E_3$: the set of entire functions of order $\le 1$ and of exponential type only in some sector.
\end{itemize}
These classes satisfy the strict inclusions
\begin{equation}\label{e:hierarchy1}
    \E_0 \subset \E_1 \subset \E_2 \subset \E_3.
\end{equation}

Historically, Pólya first developed the theory for exponential sums $\E_1$ in 1923 \cite{polya1923} before extending it to entire functions of exponential type $\E_2$ in 1929 \cite{polya1929}. Henrici later considered functions defined in a sector and of exponential type there \cite{henrici1991}. Here we focus on the class $\E_3$, a globally constrained subclass of this sectorial setting, and show that it supports a precise extension of indicator diagram theory. The extended complex plane $\bar\C$ provides the natural framework for the "virtual singularities" at infinity that arise in Henrici's study \cite{henrici1991}. Since the finite-type functions in $\E_3$ are precisely those in $\E_2$, the new content of the theory concerns entire functions of infinite (or maximal) type.

The inclusion of $\E_0$ is also natural from the point of view of type: the zero function is the unique function in $\E_3$ of negative type, namely $-\infty$. Since indicator functions in $\E_3$ may themselves take the value $-\infty$ (see \cref{e:hRGammaintro}), the zero function no longer needs to be treated as a completely separate case.

The reader is especially encouraged to consult \Cref{s:ex}; its examples include the functions in \eqref{e:phi123} and provide much of the motivation for the theory developed here.

\section{Preliminaries}

\subsection{Angles}

Let $\Theta := \R/2\pi\Z$ denote the \emph{set of angles}. When a numerical representative is needed, we identify $\t\in\Theta$ with its canonical representative in $[0, 2\pi)$. For $\alpha, \beta\in\Theta$, an interval of $\Theta$ of length $\beta - \alpha$ is one of
\begin{equation}\label{e:intervals}
    \begin{gathered}
        (\alpha, \beta) := \cond{\t\in\Theta}{0 < \t - \alpha < \beta - \alpha}, \quad [\alpha, \beta) := \cond{\t\in\Theta}{\t - \alpha < \beta - \alpha}, \\
	    (\alpha, \beta] := \cond{\t\in\Theta}{0 < \t - \alpha \le \beta - \alpha}, \quad [\alpha, \beta] := \cond{\t\in\Theta}{\t - \alpha \le \beta - \alpha},
    \end{gathered}
\end{equation}
or $\Theta$ itself, which cannot be written in the form of \cref{e:intervals}. Intervals of $\Theta$ are closed under complements: for instance, the complement of $[\alpha, \beta]$ is $(\beta, \alpha)$.

The topology on $\Theta$ is induced by the metric
\begin{equation}\label{e:dTheta}
    \dTheta :
    \left\{
    \begin{aligned}
        \Theta \times \Theta &\longrightarrow [0, \pi] \\
        (\t_1, \t_2) &\longmapsto \min(\t_2 - \t_1, \t_1 - \t_2).
    \end{aligned}
    \right.
\end{equation}

\subsection{Set notation}\label{s:setnot}
For $S, S'\subset\C$ and $u, v\in\C$, we write
\begin{multicols}{2}
\begin{itemize}
    \item $S^* := S \setminus \{0\}$.
    \item $S^\dagger := \cond{\bar z}{z\in S}$.
    \item $u + vS := \cond{u + vz}{z\in S}$.
    \item $S + S' := \cond{z+z'}{z\in S, z'\in S'}$.\footnote{This is usually referred to as the \emph{Minkowski addition} \cite{schneider2014}.}
\end{itemize}
\end{multicols}
Here $\bar z$ denotes the complex conjugate of $z$. Throughout the paper, we use the argument function in the form $\arg : \C^* \to \Theta$. With this convention, $\arg(uv) = \arg u + \arg v$ and $u = |u| e^{i \arg u}$ for all nonzero $u$ and $v$. For $r>0$, $h\in\R$, and $\t\in\Theta$, we define the following sets.
\begin{multicols}{2}
\begin{itemize}
    \item The nonnegative real axis
    \[
    \R_+ := \cond{x\in\R}{x\ge0}.
    \]
    \item The extended real line
    \[
    \bar\R := \R \cup \{-\infty, \infty\} = [-\infty, \infty].
    \]
    \item The circle
    \[
     C_r := \cond{z\in\C}{|z| = r}.
    \]
    \item The open disk
    \[
    D_r := \cond{z\in\C}{|z| < r}.
    \]
    \item The rotated open half-plane
    \[
    \Pi_{h,\t} := \cond{z\in\C}{\Re(ze^{-i\t}) > h}.
    \]
    \item For a proper subset $A$ of $\Theta$,
    \[
    S_A := \cond{z\in\C^*}{\arg z \in A},
    \]
    and
    \[
    S_\Theta := \C.
    \]
    When $A$ is an interval, $S_A$ is called a sector.
\end{itemize}
\end{multicols}

\begin{remark}\label{r:sector}
    The distinction in the definition of $S_A$ is motivated by practical considerations in the theory developed below. It also admits the following heuristic interpretation: one may regard "$\arg 0$" as encompassing \emph{all} values of $\Theta$, since $0 = 0 e^{i\t}$ for every $\t \in \Theta$. Thus $0$ belongs to $S_A$ if and only if $A = \Theta$.
\end{remark}

The \emph{effective domain} of a function $g:X\to\bar\R$ is defined as 
\begin{equation}
    \dom g := \cond{x\in X}{g(x) < \infty}.
\end{equation}
When real variables are introduced by inequalities such as "$x > 0$", infinities are not included.

\subsection{Trigonometric convexity}\label{s:tcvx}

In the sense of Valiron \cite{valiron1932}, trigonometric convexity is the analog of ordinary convexity obtained by replacing affine functions with trigonometric functions. It is an important property shared by both \emph{support functions} and \emph{indicator functions}, which are introduced in \Cref{s:support} and \Cref{s:indic}, respectively, and shown to be trigonometric-convex in \Cref{t:tcvxK} and \Cref{t:hphitcvx}.

\begin{definition}[Trigonometric convexity]\label{d:tcvx}
	A function $k : \Theta \to \bar\R$ is \emph{trigonometric-convex} if either $k\equiv-\infty$ or the following condition holds: for
	\begin{equation}\label{e:thetaimp}
		\t_1, \t_3 \in \dom k, \enspace\text{such that}\enspace \t_3 - \t_1\le\pi, \enspace\text{and}\enspace \t_2 \in [\t_1, \t_3],
	\end{equation} 
	one has
	\begin{equation}\label{e:tcvximp}
		k(\t_2) \sin(\t_3 - \t_1) \le k(\t_1) \sin(\t_3 - \t_2) + k(\t_3) \sin(\t_2 - \t_1).
	\end{equation}
\end{definition}

This definition includes the constant function equal to $-\infty$, for reasons that will become clear in the contexts of support and indicator functions. In these settings, the following characterization will be useful.

\begin{proposition}\label{p:diameter}
	$k:\Theta\to\bar\R$ is trigonometric-convex if and only if \cref{e:tcvximp} holds for
    \begin{equation}\label{e:thetaimp2}
		\t_1, \t_3 \in \dom k, \enspace\text{such that}\enspace 0 < \t_3 - \t_1<\pi, \enspace\text{and}\enspace \t_2 \in [\t_1, \t_3]
	\end{equation} 
    and either $k\equiv -\infty$ or for all $\t,\t+\pi\in\dom k$,
    \begin{equation}\label{e:diameter}
		k(\t) + k(\t+\pi) \ge 0.
	\end{equation}
\end{proposition}

\begin{proof}
    \Cref{e:diameter} follows from \cref{e:tcvximp} with $\t_1 = \t$, $\t_2 = \t + \pi/2$, and $\t_3 = \t + \pi$. Conversely, when $\t_3 - \t_1 = \pi$, \cref{e:tcvximp} follows from \cref{e:diameter}.
\end{proof}

Note that if $k \not\equiv -\infty$ is trigonometric-convex, then for all $\t \in \Theta$
\begin{equation}\label{e:k>-k}
    k(\t) \ge -k(\t+\pi),
\end{equation}
since, whenever both $\t$ and $\t+\pi$ belong to $\dom k$, the inequality follows from \cref{e:diameter}. Otherwise, at least one of $k(\t)$ or $k(\t+\pi)$ is equal to $\infty$, and \cref{e:k>-k} holds trivially.

\Cref{d:tcvx} extends the classical notion of trigonometric convexity \cite{chebotarev1949,boas1954,levin1964,levin1996}, originally defined for real-valued functions.

\begin{proposition}\label{p:tcvx}
	A function $k : \Theta \to \R$ is trigonometric-convex if and only if, for all $\t_1, \t_2, \t_3\in\Theta$ such that
	\begin{equation}\label{e:theta}
		\t_3 - \t_1 \le \pi \quad\text{and}\quad \t_2\in[\t_1, \t_3],
	\end{equation}
	the inequality
	\begin{equation}\label{e:tcvx}
		k(\t_1) \sin(\t_2 - \t_3) + k(\t_2) \sin(\t_3 - \t_1) + k(\t_3) \sin(\t_1 - \t_2) \le 0.
	\end{equation}
\end{proposition}

The indices of the $\t$ variables in the three terms of \cref{e:tcvx} are cyclic permutations of $1,2,3$ when read from left to right; this gives a simple rule for their placement. By commutativity, \cref{e:tcvx} remains valid after cyclically shifting these indices. It also has the equivalent determinantal form \cite{valiron1932,gelfond1938,levin1964}
\begin{equation}\label{e:tcvxdet}
	\begin{vmatrix}
		k(\t_1) & \cos\t_1 & \sin\t_1 \\
		k(\t_2) & \cos\t_2 & \sin\t_2 \\
		k(\t_3) & \cos\t_3 & \sin\t_3
	\end{vmatrix} \le 0
\end{equation}
because the left-hand sides of \cref{e:tcvxdet,e:tcvx} coincide. The usual notion of convexity corresponds to replacing the sine and cosine functions by the identity and the constant $1$ function, respectively, in \cref{e:tcvx} or \cref{e:tcvxdet} \cite{valiron1932}.

\begin{example}\label{x:sincos}
	From \cref{e:tcvxdet}, we see that linear combinations of the cosine and sine functions are the only bounded trigonometric-convex functions for which equality holds in \cref{e:tcvx}. The following proposition gives a slightly more general statement.
\end{example}

\begin{proposition}\label{p:sincos}
	Let $s:\Theta\to\C$. The equation
	\begin{equation}
		s(\t_1) \sin(\t_2 - \t_3) + s(\t_2) \sin(\t_3 - \t_1) + s(\t_3) \sin(\t_1 - \t_2) = 0
	\end{equation}
	holds for all $\t_1, \t_2, \t_3\in\Theta$ satisfying \cref{e:theta} if and only if $s(\t) = A \cos\t + B \sin\t$, for some $A, B\in\C$.
\end{proposition}
\noindent The function $s$ is trigonometric-convex only when $A, B\in\R$.

Several further properties of trigonometric-convex functions, including continuity and the existence of left and right derivatives, are established in \Cref{s:tcvxprop}. Only the following proposition will be needed in the subsequent sections.

\begin{proposition}\label{p:domk}
	For a trigonometric-convex function $k:\Theta\to\bar\R$, $\dom k$ is either equal to $\Theta$, an interval of $\Theta$ of length at most $\pi$, or of the form $\{\alpha, \alpha+\pi\}$ for some $\alpha\in\Theta$.
\end{proposition}

\subsection{Extensions of the complex plane}\label{s:Cbar}

For $A\subset\C$ and $B \subset\C^*$, let $\infty = \infty_{A, B}$ be an infinity symbol satisfying the following rules:
\begin{multicols}{2}
\begin{enumerate}[(i)]
    \item $z + \infty = \infty + z$ for $z\in\C$.
    \item $z \cdot \infty = \infty \cdot z$ for $z\in\C$.
    \item $0\cdot\infty = 0$.
    \item $x+\infty = \infty$ for $x\in A$.
    \item $r\cdot\infty = \infty$ for $r\in B$.
    \item[]
\end{enumerate}
\end{multicols}

We denote by $\hat\infty$ the case $A = \C$ and $B = \C^*$. The \emph{one-point compactification} of the complex plane, $\hat\C := \C \cup \{\hat\infty\}$, is the \emph{Riemann sphere}. In this setting, one has $-\hat\infty = \hat\infty$. This property no longer holds for the \emph{radial compactification} of $\C$, defined by $\tilde\C := \C \cup \tilde C_\infty$, where $\tilde C_\infty := \cond{e^{i\t} \tilde\infty}{\t\in\Theta}$ and the symbol $\tilde\infty$ corresponds to $A = \C$ and $B = \R_+^*$.

The infinity symbol used in the definition of the extended real line $\bar\R$ corresponds to the case $A = \R$ and $B = \R_+^*$. Recall that expressions such as "$\infty - \infty$" and "$-\infty + \infty$" are undefined. With the same convention, we define the \emph{extended complex plane} by
\begin{equation}\label{e:barC}
\bar\C := \C + \infty\cdot\C,
\end{equation}
which contains both $\C$ and $\bar\R$. In this setting, operations of the form "$\infty e^{i\alpha} + \infty e^{i\beta}$" are undefined whenever $\alpha \neq \beta$ in $\Theta$. In particular, one must not distribute the infinity in expressions such as $\infty e^{i\t} = \infty(\cos\t + i\sin\t)$.

From this point on, the set notation of the previous \Cref{s:setnot} is used for subsets of $\bar\C$ whenever the operations involved are defined. We also extend the notation for circles by setting $C_\infty := \bar\C \setminus \C$, so that
\begin{equation}
\bar\C = \C \cup C_\infty,
\end{equation}
and elements $z \in C_\infty$ are characterized by the condition $|z| = \infty$.

By definition \eqref{e:barC}, for any $z \in \bar\C$, there exist $u, v \in \C$ such that $z = u + \infty v$. If $|z| = \infty$, then $v \neq 0$. Hence we may write $v = r e^{i\t}$ and $u = (x + i y)e^{i\t}$ with $r > 0$, $\t \in \Theta$, and $x, y \in \R$. Using properties (iv) and (v), it follows that $z = (\infty + i y)e^{i\t}$, which we call its \emph{canonical form}. Thus, elements of $C_\infty$ are characterized by two parameters: their argument $\arg z := \t \in \Theta$ and their offset $\Im(z e^{-i\t}) := y \in \R$. Consequently,
\begin{equation}\label{e:Cinfty}
C_\infty = \cond{(\infty + i y)e^{i\t}}{y \in \R, \t \in \Theta} = \bigcup_{\t \in \Theta} (\infty + i\R)e^{i\t}.
\end{equation}

We now define a topology on $\bar\C$ by specifying neighborhood bases. For $\delta > 0$ and $z \in \C$, set $U_\delta(z) := z + D_\delta$. For $z = (\infty + i y)e^{i\t} \in C_\infty$, set
\begin{equation}\label{e:nhood}
U_\delta(z) := \cond{s \in \bar\C}{\Re(s e^{-i\t}) > \inv\delta,\; \abs{\Im(s e^{-i\t}) - y} < \delta},
\end{equation}
so that $(U_\delta(z))_{\delta > 0}$ forms a neighborhood basis at $z$. The resulting notion of convergence coincides with the usual one in $\C$. For points in $C_\infty$, \cref{e:nhood} implies that a sequence $(z_n)_{n \in \N} \subset \bar\C$ converges to $(\infty + i y)e^{i\t}$ as $n \to \infty$ if and only if
\begin{equation}
\Re(z_n e^{-i\t}) \to_{n\to\infty} \infty,\quad \Im(z_n e^{-i\t}) \to_{n\to\infty} y.
\end{equation}
Note that, under this topology, $\bar\C$ is not compact; indeed, $\C$ is not relatively compact in $\bar\C$\footnote{For example, if $\theta\in\Theta$ and $u_n := (n+i\sqrt{n})e^{i\theta}$, then $(u_n)_{n\in\N}$ has no convergent subsequence in $\bar\C$: its argument tends to $\theta$, while its offset $\Im(u_n e^{-i\theta}) = \sqrt n$ diverges. One could add further boundary points to compactify $\bar\C$, but we avoid doing so because these points introduce geometric and topological complications that are not needed for the analytic results below.}.

For a subset $S \subset \bar\C$, a set $U \subset \bar\C$ is called a \emph{neighborhood} of $S$ if, for every $z \in S$, there exists $\delta > 0$ such that $U \supset U_\delta(z)$. If $\delta$ can be chosen independently of $z$, we say that the neighborhood is \emph{uniform}. A subset of $\bar\C$ that is a neighborhood of itself is thus \emph{open}, and a set is \emph{closed} if its complement is open. We denote the \emph{closure} and \emph{boundary} of $S$ by $\bar S$ and $\partial S$, respectively.

The topologies on $\tilde\C$ and $\hat\C$ are obtained by identifying $\infty + i\R$ with $\tilde\infty$ and $\tilde C_\infty$ with $\hat\infty$, respectively.

\section{Convex analysis in the extended complex plane}\label{s:cvxanalysis}

\subsection{Geodesics and convex sets}\label{s:cvxset}

For $u, v\in\C$, define the complex segment
\begin{equation}
[u, v] := \cond{tu + (1-t)v}{t\in [0, 1]}.
\end{equation}
In $\C$, the set $[u, v]$ is a geodesic segment; we extend this notion to $\bar\C$ below.
\begin{definition}
    For $u, v\in\bar\C$, a set $G \subset\bar\C$ is a geodesic segment connecting $u$ and $v$ if there exist sequences $u_n \to u$ and $v_n \to v$ of complex numbers such that
    \begin{equation}
        [u_n, v_n] \to_{n\to\infty} G,
    \end{equation}
    where the convergence is in the sense of Painlevé--Kuratowski.\footnote{See Chapter 4 of \cite{rockafellar2004}.}
\end{definition}

The Painlevé--Kuratowski limit is taken with respect to the topology introduced in \Cref{s:Cbar}. This construction of geodesic segments of $\bar\C$ provides the appropriate definition for developing the theory presented in this paper.\footnote{Our construction is in fact quite general: given a geodesic metric space $X$ and an extension $\bar X$, one can define a geodesic segment of $\bar X$ as the Painlevé--Kuratowski limit of geodesic segments of $X$. A similar construction appears in \cite[Sec.~II.9]{bridson1999}.} They are explicitly described in the following theorem.

\begin{theorem}[Classification of geodesic segments of $\bar\C$]\label{t:geodesic}
    Let $u, v\in\bar\C$.
    \begin{alphabetize}
    \item If $u, v\in\C$, the only geodesic connecting $u$ and $v$ is $[u, v]$.
    \item If $u = (x + iy_1) e^{i\t} \in\C$ and $v = (\infty + iy_2) e^{i\t}\in C_\infty$, the only geodesic connecting $u$ and $v$ is
    \begin{equation}\label{e:geob}
        (([x, \infty) + i y_1) \cup (\infty + i[y_1,y_2]))e^{i\t}.
    \end{equation}
    \item Assume $u=(\infty + iy_1)e^{i\alpha}$ and $v = (\infty + iy_2) e^{i\beta}$ belong to $C_\infty$.
    \begin{enumerate}[(\alph{enumi}.1)]
        \item If $\alpha = \beta$, the only geodesic connecting $u$ and $v$ is
        \begin{equation}\label{e:geoc1}
            (\infty + i [y_1, y_2])e^{i\alpha}
        \end{equation}
        \item If $0 < \beta - \alpha < \pi$, the only geodesic connecting $u$ and $v$ is
        \begin{equation}\label{e:geoc2}
            (\infty + i [y_1, \infty))e^{i\alpha} \cup \bigcup_{\t\in(\alpha, \beta)} (\infty + i \R)e^{i\t} \cup (\infty + i (-\infty, y_2])e^{i\beta}.
        \end{equation}
        \item If $\beta - \alpha = \pi$, the geodesics connecting $u$ and $v$ are
        \begin{equation}\label{e:geoc3}
            ((\infty + i [y_1, y]) \cup (\R + iy) \cup (-\infty + [y, y_2]))e^{i\alpha},
        \end{equation}
        for each $y\in [y_1, y_2]$.
    \end{enumerate}
    \end{alphabetize}
\end{theorem}

The Riemann sphere $\hat\C$ and the radial compactification $\tilde\C$ admit natural depictions as a sphere and a closed disk, respectively. For the extended complex plane $\bar\C$, case (c.2) of \Cref{t:geodesic} suggests viewing $C_\infty$ as a polygon with uncountably many sides of the form $(\infty + i\R)e^{i\theta}$; see \cref{e:Cinfty}. \Cref{f:Cbar} illustrates this viewpoint by replacing $C_\infty$ with a discretized collection of sides and showing the five types of geodesics described in \Cref{t:geodesic}.

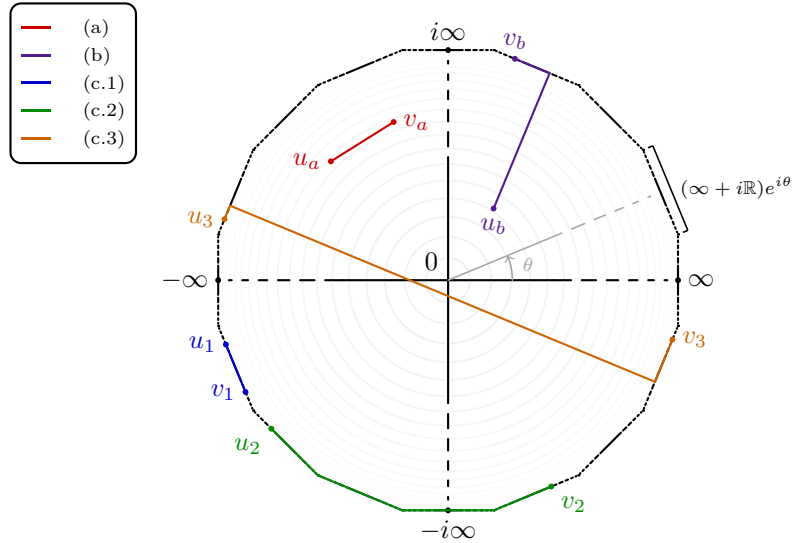
\begin{figure}[ht]
    \centering
    \begin{tikzpicture}[line cap=round, line join=round, thick, scale=1]
        \DrawExtendedComplexPlane{
            \colorlet{caseacolor}{Red3}
            \colorlet{casebcolor}{Purple4}
            \colorlet{caseconecolor}{Blue3}
            \colorlet{casectwocolor}{Green4}
            \colorlet{casecthreecolor}{DarkOrange3}
            
            \pgfmathsetmacro{\gap}{2}
            \pgfmathsetmacro{\ticklen}{2}
            \pgfmathsetmacro{\outeroff}{\gap+\ticklen}
            \coordinate (SideMid) at ($(CP0)!0.5!(CP1)$);
            \path let \p1 = ($(CP1)-(CP0)$), \n1={veclen(\x1,\y1)}
                in coordinate (StartInner) at ($(CP0)+({\gap*\y1/\n1},{-\gap*\x1/\n1})$)
                   coordinate (StartOuter) at ($(CP0)+({\outeroff*\y1/\n1},{-\outeroff*\x1/\n1})$)
                   coordinate (EndInner) at ($(CP1)+({\gap*\y1/\n1},{-\gap*\x1/\n1})$)
                   coordinate (EndOuter) at ($(CP1)+({\outeroff*\y1/\n1},{-\outeroff*\x1/\n1})$)
                   coordinate (LabelMid) at ($(StartOuter)!0.5!(EndOuter)$);
            \DrawExtendedComplexPlaneRay[draw=gray!70,line width=0.6pt]{2*\ComplexPlaneRotation}{0.46}
            \draw[->,draw=gray!70,line width=0.6pt] (0.85,0) arc[start angle=0,end angle={2*\ComplexPlaneRotation},radius=0.85];
            \node[text=gray!70,font=\scriptsize] at ({\ComplexPlaneRotation}:1.08) {$\theta$};
            \draw[line cap=butt, line width=0.6pt] (StartInner) -- (StartOuter);
            \draw[line cap=butt, line width=0.6pt] (StartOuter) -- (EndOuter);
            \draw[line cap=butt, line width=0.6pt] (EndOuter) -- (EndInner);
            \node[font=\scriptsize,right] at (LabelMid) {$(\infty + i\R)e^{i\theta}$};

            \coordinate (CaseAU) at (-1.55,1.57);
            \coordinate (CaseAV) at (-0.72,2.09);
            \draw[caseacolor] (CaseAU) -- (CaseAV);
            \fill[caseacolor] (CaseAU) circle (1.1pt);
            \fill[caseacolor] (CaseAV) circle (1.1pt);
            \node[text=caseacolor, left] at (CaseAU) {$u_a$};
            \node[text=caseacolor, right] at (CaseAV) {$v_a$};

            \coordinate (CaseBFoot) at ($(CP2)!0.34!(CP3)$);
            \coordinate (CaseBV) at ($(CP2)!0.75!(CP3)$);
            \path let \p1 = ($(CP3)-(CP2)$) in coordinate (CaseBU) at ($(CaseBFoot)+(-1.6*\y1,1.6*\x1)$);
            \draw[casebcolor] (CaseBU) -- (CaseBFoot) -- (CaseBV);
            \fill[casebcolor] (CaseBU) circle (1.1pt);
            \fill[casebcolor] (CaseBV) circle (1.1pt);
            \node[text=casebcolor, below] at (CaseBU) {$u_b$};
            \node[text=casebcolor, above] at (CaseBV) {$v_b$};

            \coordinate (CaseCOneU) at ($(CP8)!0.22!(CP9)$);
            \coordinate (CaseCOneV) at ($(CP8)!0.78!(CP9)$);
            \draw[caseconecolor] (CaseCOneU) -- (CaseCOneV);
            \fill[caseconecolor] (CaseCOneU) circle (1.1pt);
            \fill[caseconecolor] (CaseCOneV) circle (1.1pt);
            \node[text=caseconecolor, left] at (CaseCOneU) {$u_1$};
            \node[text=caseconecolor, left] at (CaseCOneV) {$v_1$};

            \coordinate (CaseCTwoU) at ($(CP9)!0.28!(CP10)$);
            \coordinate (CaseCTwoV) at ($(CP12)!0.68!(CP13)$);
            \draw[casectwocolor] (CaseCTwoU) -- (CP10) -- (CP11) -- (CP12) -- (CaseCTwoV);
            \fill[casectwocolor] (CaseCTwoU) circle (1.1pt);
            \fill[casectwocolor] (CaseCTwoV) circle (1.1pt);
            \node[text=casectwocolor, below left] at (CaseCTwoU) {$u_2$};
            \node[text=casectwocolor, below right] at (CaseCTwoV) {$v_2$};

            \coordinate (CaseCThreeU) at ($(CP7)!0.18!(CP6)$);
            \coordinate (CaseCThreeLeft) at ($(CP7)!0.34!(CP6)$);
            \coordinate (CaseCThreeRight) at ($(CP14)!0.34!(CP15)$);
            \coordinate (CaseCThreeV) at ($(CP14)!0.84!(CP15)$);
            \draw[casecthreecolor] (CaseCThreeU) -- (CaseCThreeLeft) -- (CaseCThreeRight) -- (CaseCThreeV);
            \fill[casecthreecolor] (CaseCThreeU) circle (1.1pt);
            \fill[casecthreecolor] (CaseCThreeV) circle (1.1pt);
            \node[text=casecthreecolor, left] at (CaseCThreeU) {$u_3$};
            \node[text=casecthreecolor, right] at (CaseCThreeV) {$v_3$};

            \node[
                overlay,
                anchor=north east,
                draw,
                rounded corners,
                fill=white,
                fill opacity=0.92,
                text opacity=1,
                inner sep=4pt,
                font=\scriptsize
            ] at ($(current bounding box.north west)+(-0.2,0.15)$) {
                \begin{tabular}{@{}ll@{}}
                    \textcolor{caseacolor}{\rule[0.55ex]{1.4em}{1.1pt}} & (a) \\
                    \textcolor{casebcolor}{\rule[0.55ex]{1.4em}{1.1pt}} & (b) \\
                    \textcolor{caseconecolor}{\rule[0.55ex]{1.4em}{1.1pt}} & (c.1) \\
                    \textcolor{casectwocolor}{\rule[0.55ex]{1.4em}{1.1pt}} & (c.2) \\
                    \textcolor{casecthreecolor}{\rule[0.55ex]{1.4em}{1.1pt}} & (c.3)
                \end{tabular}
            };
        }
    \end{tikzpicture}
    \caption{A polygonal depiction of the extended complex plane and the five types of geodesics from \Cref{t:geodesic}.}
    \label{f:Cbar}
\end{figure}

\begin{proof}[of \Cref{t:geodesic}]
    See \Cref{s:cvxinf}.
\end{proof}

The classical notion of convexity is not well-defined on the extended complex plane, since the addition of two distinct points at infinity is not permitted. We therefore define convexity in $\bar\C$ using the corresponding notion from geodesic metric spaces \cite{bridson1999}.

\begin{definition}[Convexity]\label{d:cvx}
    A subset $K$ of $\bar\C$ is said to be \emph{(geodesically) convex} if, for all $u,v\in K$, every geodesic segment joining $u$ and $v$ is contained in $K$.
\end{definition}

An equivalent sequential characterization is given below.

\begin{proposition}[Sequential characterization of convexity]\label{p:cvxdef}
	$K\subset\bar\C$ is convex if and only if, for all sequences $(u_n)_{n\in\N}$ and $(v_n)_{n\in\N}$ in $\C$ converging in $\bar\C$ to points of $K$, and every convergent sequence $(t_n)_{n\in\N}$ in $[0, 1]$ such that
	\begin{equation}\label{e:cvxseq}
		t_n u_n + (1-t_n) v_n \to_{n\to\infty} w,
	\end{equation}
	for some $w\in\bar\C$, one has $w\in K$.
\end{proposition}

When $K\subset\C$, if $u_n \to u$, $v_n \to v$, and $t_n\to t\in [0, 1]$, then $w = t u + (1-t) v$, so \Cref{p:cvxdef} coincides with the usual notion of convexity in $\C$. The nonstandard cases occur when $u\in C_\infty$ and $t = 0$, when $v\in C_\infty$ and $t=1$, or when both $u, v\in C_\infty$, since $w$ is then not uniquely determined.

We now classify the infinite points of a general closed unbounded convex subset of $\bar\C$.

\begin{proposition}\label{p:clcvxinf}
    Let $K$ be a closed unbounded convex subset of $\bar\C$. Then $\cond{\arg p}{p\in K\cap C_\infty}$ is either equal to $\Theta$, an interval of $\Theta$ of length at most $\pi$, or of the form $\{\alpha, \alpha+\pi\}$ for some $\alpha\in\Theta$.
\end{proposition}

\begin{proof}
    Let $I := \cond{\arg s}{s\in K\cap C_\infty}$ and $\alpha, \beta\in I$.
    \begin{itemize}
        \item If $\beta - \alpha < \pi$, the geodesic \eqref{e:geoc2} is included in $K$, and it follows that $[\alpha, \beta]\subset I$. Therefore, $I$ is an interval. If $I\neq\Theta$, there exists $\alpha, \beta\in\Theta$ so that $\bar I = [\alpha, \beta]$. If $\beta - \alpha > \pi$, there is $\alpha', \beta' \in I$ such that $\beta' - \alpha' > \pi$, i.e., $\alpha' - \beta' < \pi$. Hence, $[\beta', \alpha'] \subset I$, which contradicts $I\neq\Theta$, and therefore $\beta - \alpha \le \pi$.
        \item Otherwise, for all $\alpha, \beta\in I$, $\beta - \alpha = \pi$, that is, $I = \{\alpha, \alpha+\pi\}$.
    \end{itemize}
\end{proof}

It follows from \Cref{d:cvx} that the closed unbounded convex subsets $K$ of $\bar\C$ whose set $I = \cond{\arg p}{p\in K\cap C_\infty}$ has length $\pi$ are precisely the half-planes. On the other hand, if $I = \{\alpha, \alpha + \pi\}$, then $K = \cond{p\in\bar\C}{\Im(p e^{-i\alpha}) \in J}$ for some bounded interval $J$ of $\R$; this is a strip of $\bar\C$.


\subsection{Support functions}\label{s:support}

Minkowski's support function of a set $U \subset \R^n$ is usually defined by $s \mapsto \sup_{p\in U} \scal{p, s}$ \cite[Sec.~13]{rockafellar1970}, where $\scal{}$ denotes the canonical dot product. We use here the complex analogue for $n=2$. In $\C$, the dot product can be written as $\Re(p)\Re(s) + \Im(p)\Im(s) = \Re(p\bar s)$, which motivates the definition
\begin{equation}\label{e:Kappa}
\Kappa_U(s) := \sup_{p\in U} \Re(p \bar s),
\end{equation}
for $U\subset\bar\C$. The \emph{polar support function} is defined analogously by
\begin{equation}\label{e:kdef}
\k_U(\t) := \Kappa_U(e^{i\t}) = \sup_{p\in U} \Re(p e^{-i\t}).
\end{equation}
The two support functions are related by
\begin{equation}\label{e:Kk}
    \Kappa_U(s) =
    \begin{cases}
        \k_U(\arg s) |s| &\com{if} s \neq 0 \\
        0 &\com{if} s = 0 \com{and} U \neq \emptyset \\
        -\infty &\com{if} s = 0 \com{and} U = \emptyset.
    \end{cases}
\end{equation}

\begin{example}\label{x:k}
    The following examples will be useful.
    \begin{itemize}
        \item Empty set: by the usual convention of the supremum, $\k_\emptyset(\t) = -\infty$.
        \item Singleton: $\k_{\{a\}}(\t) = \Re a \cos\t + \Im a \sin\t$.
        \item Disk: $\k_{D_r}(\t) = r$.
        \item Complex plane: $\k_\C(\t) = \infty$.
        \item Real line: $\k_\R(\t) = 0$ for $\t = \pm \frac\pi2$ and $\infty$ elsewhere.
        \item T-shape: $\k_{\R_+ \cup (\infty + i\R)}(\t) = \infty$ for $\t \in \bra{-\frac\pi2, \frac\pi2}$ and $0$ elsewhere.
    \end{itemize}
\end{example}

Bounded convex subsets of $\C$, together with unbounded convex subsets of $\C$ augmented by suitable points of $C_\infty$, do not exhaust all convex subsets of $\bar\C$. There are also convex subsets contained entirely in $C_\infty$, which are classified below.
\begin{theorem}[Classification of infinite convex sets]\label{t:cvxinf}
    Let $K$ be a convex subset of $\bar\C$. The following are equivalent.
    \begin{alphabetize}
        \item There exists $\t\in\Theta$ such that $\k_K(\t)= -\infty$.
        \item $K\subset C_\infty$.
        \item $K$ is of the form
        \begin{equation}\label{e:cvxinf}
            K=\cond{(\infty + iy)e^{i\t}}{\t\in[\alpha, \beta], y\in I(\t)},
        \end{equation}
        where $\beta - \alpha \le \pi$, $I(\t) = \R$ for $\t\in(\alpha, \beta)$, $I(\alpha)$ and $I(\beta)$ are intervals of $\R$, and if $\alpha \neq \beta$,
        \begin{itemize}
            \item $I(\alpha)$ is empty or unbounded from above.
            \item $I(\beta)$ is empty or unbounded from below.
            \item If $I(\alpha)$ and $I(\beta)$ are nonempty, then $\beta - \alpha < \pi$.
        \end{itemize}
        \item Either $\k_K \equiv -\infty$ or there exist $\alpha, \beta\in\Theta$ and $a, b\in\bar\R$ such that $\beta - \alpha \le \pi$, $a \le b$ when $\alpha = \beta$, and
        \begin{equation}\label{e:cvxinfk}
            \k_K(\t) =
            \begin{cases}
                -a &\com{if} \t = \alpha - \frac\pi2 \\
                \infty &\com{if} \t\in \pa{\alpha - \frac\pi2, \beta + \frac\pi2} \\
                b &\com{if} \t = \beta + \frac\pi2 \\
                -\infty &\com{if} \t\in \pa{\beta + \frac\pi2, \alpha - \frac\pi2}.
            \end{cases}
        \end{equation}
        \end{alphabetize}
\end{theorem}

\begin{proof} \hfill
    \begin{itemize}
        \item[(a) $\imp$ (b)] We prove the contrapositive. Assume that $K\not\subset C_\infty$. Then there exists $p\in K\cap\C$, and hence, for all $\t\in\Theta$, $\k_K(\t) \ge \Re(p e^{-i\t}) > -\infty$.
        \item[(b) $\imp$ (c)] The only possible geodesics contained in $K$ are those of \cref{e:geoc1,e:geoc2}, and the claim follows.
        \item[(c) $\imp$ (d)] If $K = \emptyset$, then $\k_K \equiv -\infty$. Otherwise, $K$ is of the form \eqref{e:cvxinf}, with $I(\alpha) \neq \emptyset$ when $\alpha = \beta$.
        \begin{itemize}
            \item If $\alpha = \beta$, then $\oline{I(\alpha)} = [a, b]$ for some $a, b\in\bar\R$ with $a \le b$.
            \item If $\alpha \neq \beta$, then $\oline{I(\alpha)} = [a, \infty]$ and $\oline{I(\beta)} = [-\infty, b]$ for some $a, b\in\bar\R$, while $I(\t) = \R$ for $\t\in(\alpha, \beta)$. In addition, if $a < \infty$ and $b > -\infty$, then $\beta - \alpha < \pi$.
        \end{itemize}
        For $z\in K$,
		\[
		\Re\pa[big]{ze^{-i\pa{\alpha - \frac\pi{2}}}} = -\Im(ze^{-i\alpha}) \le -a,\quad\Re\pa[big]{ze^{-i\pa{\beta + \frac\pi{2}}}} = \Im(ze^{-i\beta}) \le b,
		\]
		and since these inequalities are attained, the claim follows by a direct computation.
        \item[(d) $\imp$ (a)] is immediate.
    \end{itemize}
\end{proof}

We state and prove several properties of support functions below and in the next section.

\begin{proposition}\label{p:kunionsum}
	Let $U,V\subset\bar\C$.
	\begin{enumerate}
		\item $\k_{U\cup V} = \max(\k_U, \k_V)$.
		\item If $U$ or $V$ is bounded and nonempty, $\k_{\lambda U + \mu V} = \lambda \k_U + \mu \k_V$ for all $\lambda, \mu > 0$.
	\end{enumerate}
\end{proposition}

\begin{proof}
	\hfill
	\begin{enumerate}
		\item Since $U\subset U\cup V$ and $V\subset U\cup V$, \Cref{l:kless} gives $\max(\k_U(\t), \k_V(\t)) \le \k_{U\cup V}(\t)$ for every $\t\in\Theta$. Conversely, if $z\in U\cup V$, then either $z\in U$ or $z\in V$, so $\Re(ze^{-i\t}) \le \max(\k_U(\t), \k_V(\t))$. Taking the supremum over $z\in U\cup V$ gives the reverse inequality.
		\item For all $\t\in\Theta$,
		      \begin{align*}
			      \k_{\lambda U + \mu V}(\t)
			       & = \sup_{(u, v)\in U\times V} \Re((\lambda u + \mu v)e^{-i\t}) = \sup_{u\in U}\sup_{v\in V} \pa{\lambda \Re(ue^{-i\t}) + \mu \Re(ve^{-i\t})} \\
			       & = \sup_{u\in U} \pa{\lambda \Re(ue^{-i\t}) + \mu \sup_{v\in V} \Re(ve^{-i\t})} = \lambda \k_U(\t) + \mu \k_V(\t).
		      \end{align*}
	\end{enumerate}
\end{proof}

\begin{lemma}\label{l:kless}
	If $U\subset V$, then $\k_U(\t) \le \k_V(\t)$ for all $\t\in\Theta$. The converse holds whenever $V$ is closed and convex.
\end{lemma}

\begin{proof}
	See \Cref{s:propsupport}.
\end{proof}

\subsection{Convex hull}\label{s:cvxhull}

The \emph{convex hull} of $U\subset\bar\C$, denoted $\conv U$, is the smallest convex set containing $U$, that is
\begin{equation}\label{e:cvxdef}
	\conv U := \bigcap_{\substack{V\,\text{convex} \\ V \supset U}} V.
\end{equation}

\begin{definition}\label{d:wulff}
    For a given function $k : \Theta \to \bar\R$, the Wulff shape\footnote{See Section 7.5 of \cite{schneider2014} and references therein.}  $\W(k)\subset\bar\C$ of $k$ is defined as the intersection of half-planes
    \begin{equation}
        \W(k) := \bigcap_{\t\in\Theta} \cond{p\in\bar\C}{\Re(pe^{-i\t}) \le k(\t)}.
    \end{equation}
\end{definition}

When $U$ is closed, it is possible to improve on \cref{e:cvxdef} by using the support function to express $\conv U$ as an intersection of closed half-planes only.

\begin{theorem}\label{t:convPi}
	If $U\subset\bar\C$, then
	\begin{equation}\label{e:cvxhull}
		\conv \bar U = \W(\k_U).
	\end{equation}
\end{theorem}
\Cref{e:cvxhull} is illustrated in \Cref{f:cvxhull}.

\begin{figure}[ht]
	\centering
	\begin{tikzpicture}[use Hobby shortcut]
		\newcommand{\drawLine}[4]{%
			\pgfmathsetmacro{\dx}{#1*cos(#2)}
			\pgfmathsetmacro{\dy}{#1*sin(#2)}
			\pgfmathsetmacro{\halfL}{#3/2}
			\pgfmathsetmacro{\vx}{-sin(#2)}
			\pgfmathsetmacro{\vy}{cos(#2)}
			\coordinate (A) at ({\dx - #3*\vx}, {\dy - #3*\vy});
			\coordinate (B) at ({\dx + #4*\vx}, {\dy + #4*\vy});
			\draw[dashed,->] (0,0) -- (\dx, \dy);
			\draw[thick] (A) -- (B);
		}
		\draw[->] (-3.5, 0) -- (3.7, 0) node[right] {$\Re z$};
		\draw[->] (0, -2.7) -- (0, 2.5) node[above] {$\Im z$};
		\draw[closed,pattern=north west lines] (0,-1.4).. (-1.5,-1) .. (-1.1,1.8) .. (0.2,1.2) .. (2.6,1.2) .. (2,-0.7) ..  (1.7,-1.9);
		\node[scale=1.5,below right] at (0,0) {\contour{white}{\textit{S}}};
		\drawLine{1.71}{85}{3.2}{2.6}
		\drawLine{2.56}{160}{1.5}{1.2}
		\drawLine{1.97}{220}{2.2}{0.3}
		\drawLine{1.48}{249}{1.8}{2.8}
		\drawLine{2.32}{340.5}{2}{2.8}
	\end{tikzpicture}
	\caption{Construction of the convex hull of $S$ by intersection of half-planes (here, only finitely many are drawn).}
	\label{f:cvxhull}
\end{figure}
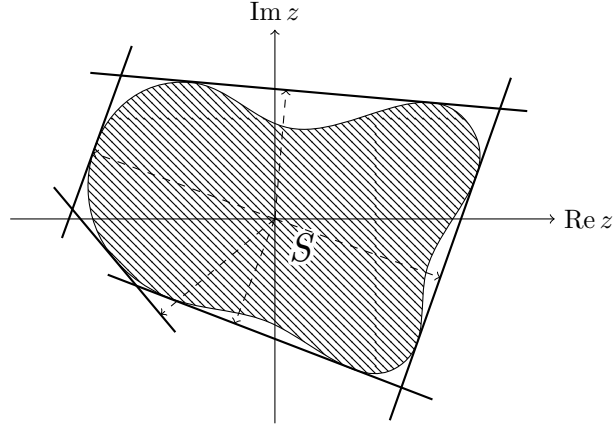
\FloatBarrier

\begin{proof}
	Let $K := \W(\k_U)$. For each $\t\in\Theta$, the half-plane $P_\t := \cond{p\in\bar\C}{\Re(pe^{-i\t}) \le \k_U(\t)}$ contains $U$, hence also $\bar U$ because $P_\t$ is closed. Since every $P_\t$ is convex, their intersection $K$ is convex as well. It follows that $K$ contains the convex hull of $\bar U$, that is, $\conv \bar U \subset K$. Conversely, if $p\in K$, then by definition $\Re(pe^{-i\t}) \le \k_U(\t)$ for every $\t\in\Theta$. Therefore, $\k_K(\t) = \sup_{p\in K} \Re(pe^{-i\t}) \le \k_U(\t) \le \k_{\conv\bar U}(\t)$,
	where the last inequality follows from \Cref{l:kless}, since $U \subset \bar U \subset \conv\bar U$. By the converse of \Cref{l:kless}, this implies $K \subset \conv\bar U$.
\end{proof}

\begin{corollary}\label{c:kconv}
	For all $U\subset\bar\C$, $\k_{\conv U} \equiv \k_U$.
\end{corollary}

\begin{proof}
	Since $\conv\emptyset = \emptyset$, we may assume that $U$ is nonempty. Because $U\subset\conv U$, \Cref{l:kless} gives $\k_U(\t) \le \k_{\conv U}(\t)$ for every $\t\in\Theta$. Conversely, let $p\in\conv U$. By \Cref{t:convPi}, we have $p\in \W(\k_U)$, that is, $\Re(pe^{-i\t}) \le \k_U(\t)$ for all $\t\in\Theta$. Taking the supremum over $p\in\conv U$, we obtain $\k_{\conv U}(\t) \le \k_U(\t)$ for every $\t\in\Theta$.
\end{proof}

Blaschke observed in 1914 \cite{blaschke1914} that the polar form of Minkowski's support function of a convex set is trigonometric-convex. It is clear from \Cref{c:kconv} that this property is not limited to convex sets.

\begin{proposition}\label{p:supptcvx}
	If $U\subset\bar\C$, $\k_U$ is trigonometric-convex.
\end{proposition}

In fact, Gelfond stated the converse in 1938 \cite{gelfond1938}: every real-valued $2\pi$-periodic trigonometric-convex function is the support function of some compact subset of $\C$. This was later proved by Chebotarev and Meiman \cite[Thm.~29]{chebotarev1949}; see also \cite[\S 5.2]{boas1954}. We now extend this characterization to periodic functions taking values in $\bar\R$.

\begin{theorem}\label{t:tcvxK}
	The function $k : \Theta\to\bar\R$ is trigonometric-convex if and only if there exists a (unique) closed convex subset $K$ of $\bar\C$ such that $k \equiv \k_K$.
\end{theorem}

\begin{proof}[of \Cref{p:supptcvx} and \Cref{t:tcvxK}]
    See \Cref{s:propsupport}.
\end{proof}

\begin{corollary}\label{c:wulffinv}
	The map $K \mapsto \k_K$ is invertible over closed convex sets of $\bar\C$ with the inverse map given by $k\mapsto\W(k)$ restricted to trigonometric-convex functions from $\Theta$ to $\bar\R$.
\end{corollary}

\begin{proof}
	We verify the two compositions. If $K$ is closed and convex, then \Cref{t:convPi} gives
	\begin{equation}\label{e:W(k)}
		\W(\k_K) = \conv K = K.
	\end{equation}
	Conversely, if $k$ is trigonometric-convex, then by \Cref{t:tcvxK} there exists a closed convex subset $K$ of $\bar\C$ such that $k = \k_K$. Applying \cref{e:W(k)} to this set $K$, we obtain $\W(k) = \W(\k_K) = K$, and hence $\k_{\W(k)} = \k_K = k$.
\end{proof}

\section{Growth of complex functions}\label{s:growth1}

This section introduces the basic tools used to measure the growth of general complex functions. In \Cref{s:growth2}, we specialize to the analytic case. We first fix the asymptotic notation used when the argument approaches a set.

\begin{definition}
    Let $f, g$ be functions on $A\subset\bar\C$ and let $B \subset \bar A$. We define the following asymptotic relations:
    \begin{itemize}
        \item $f(z) = o(g(z))$ as $z\to B$ if and only if for every $\epsilon > 0$, there exists a neighborhood $U\subset A$ of $B$ such that $|f(z)| \le \epsilon |g(z)|$ when $z\in U$.
        \item $f(z) = O(g(z))$ as $z\to B$ if and only if there exists $M > 0$ and a neighborhood $U\subset A$ of $B$ such that $|f(z)| \le M |g(z)|$ when $z\in U$.
        \item $f(z) \sim g(z)$ as $z\to B$ if and only if $f(z) - g(z) = o(g(z))$ as $z\to B$.
    \end{itemize}
    If the neighborhood $U$ in the definitions above can be chosen uniform, we say that the corresponding asymptotic relation holds uniformly as $z\to B$. If $B = \{z_0\}$, we simply write $z\to z_0$ instead of $z\to \{z_0\}$.
\end{definition}

\subsection{Order and type}

For a complex function $\phi$ and a set $U\subset\C$, define
\begin{equation}\label{e:Mphi}
    M_\phi(r; U) := \sup_{\substack{|s| = r \\ s\in U}} |\phi(s)|.
\end{equation}
For $U=\C$, this gives the classical maximum modulus function
\begin{equation}
    M_\phi(r) := M_\phi(r; \C) = \sup_{\t\in\Theta} |\phi(r e^{i\t})|.
\end{equation}

\begin{definition}\label{d:exp}
    Let $\phi$ be a complex function and let $U$ be an unbounded subset of $\C$. The \emph{order} $\rho$ of $\phi$ in $U$ is defined by
    \begin{equation}\label{e:rho}
        \rho := \inf\cond{\varrho>0}{M_\phi(r; U) = O(\exp\pa{r^\varrho})} \in \bar\R_+,
    \end{equation}
    and, for $\varrho > 0$, the \emph{$\varrho$-type} $\sigma$ of $\phi$ in $U$ is defined by
    \begin{equation}\label{e:sigma}
        \sigma := \inf\cond{\varsigma\in\R}{M_\phi(r; U) = O(\exp(\varsigma r^\varrho))} \in \bar\R.
    \end{equation}
\end{definition}

The $\varrho$-type may take the values $\pm\infty$: it is $-\infty$ when $\phi\equiv0$, and it is $\infty$ when no finite $\varsigma$ satisfies \cref{e:sigma}. Similarly, the order is $\infty$ when no $\varrho>0$ satisfies \cref{e:rho}.

Moreover, in these definitions the infimum is not necessarily a minimum: $\phi(s) = s$ has order $0$ and $1$-type $0$, but $|\phi(s)|$ is not $O(1)$.

\begin{proposition}\label{p:order}
    If $\phi$ is nonconstant and of order $\rho$ in $U$,
    \begin{equation}\label{e:rholimsup}
        \rho = \limsup_{\substack{|s|\to\infty \\ s\in U}} \frac{\log\log |\phi(s)|}{\log |s|} = \limsup_{r\to\infty} \frac{\log\log M_\phi(r; U)}{\log r}.
    \end{equation}
\end{proposition}

\begin{proposition}\label{p:type}
    For $\varrho > 0$, if $\phi$ is of $\varrho$-type $\sigma$ in $U$,
    \begin{equation}\label{e:typelimsup}
        \sigma = \limsup_{\substack{|s|\to\infty \\ s\in U}} \frac{\log |\phi(s)|}{|s|^\varrho} = \limsup_{r\to\infty} \frac{\log M_\phi(r; U)}{r^\varrho}.
    \end{equation}
\end{proposition}

\begin{proof}
    Let $\ell$ be the right-hand side of \cref{e:typelimsup}. If $\phi$ has $\varrho$-type $\sigma$ in $U$, then for every $\varsigma > \sigma$,
    $M_\phi(r; U) = O(e^{\varsigma r^\varrho})$, which implies $\ell \le \varsigma$. Hence $\ell \le \sigma$. Conversely, if $l > \ell$, then $\log M_\phi(r;U)/r^\varrho \le l$ for all sufficiently large $r$. Thus $M_\phi(r; U) = O(e^{l r^\varrho})$, and the definition of the type gives $\sigma \le l$. Letting $l\to\ell$ gives $\sigma \le \ell$, and therefore $\sigma = \ell$. The first equality follows from the calculation
    \[
    \limsup_{\substack{|s|\to\infty \\ s\in U}} \frac{\log |\phi(s)|}{|s|^\varrho} = \lim_{R\to\infty} \sup_{\substack{|s| \ge R \\ s\in U}} \frac{\log |\phi(s)|}{|s|^\varrho} = \lim_{R\to\infty} \sup_{r \ge R} \sup_{\substack{|s| = r \\ s\in U}} \frac{\log |\phi(s)|}{|s|^\varrho}  = \limsup_{r\to\infty} \frac{\log M_\phi(r; U)}{r^\varrho}.
    \]
    The proof for the order is identical.
\end{proof}

In particular, for all $\varrho\ge0$, the zero function has $\varrho$-type $-\infty$. If $\phi$ has order $\rho$ and $\rho$-type greater than $-\infty$, then it has $\varrho$-type $0$ for all $\varrho > \rho$, and $\varrho$-type $\infty$ for all $\varrho < \rho$.

\subsection{Exponential type}\label{s:indic}

\begin{definition}
    A function $\phi:\C\to\C$ is said to be of \emph{exponential type} in $U\subset\C$ if there exists $A, B \ge 0$ such that
    \[
    |\phi(s)| \le A e^{B|s|},
    \]
    for $s\in U$. If $U = \C$, we simply say that $\phi$ is of exponential type.
\end{definition}
The term "exponential type" was introduced by Pólya \cite{polya1923}.

\begin{definition}
    Let $\phi:\C\to\C$.
    \begin{enumerate}
        \item The \emph{order} of $\phi$ is the order of $\phi$ in $\C$.
        \item The \emph{exponential type} of $\phi$ is the $1$-type of $\phi$ in $\C$.
    \end{enumerate}
\end{definition}

\begin{proposition}\label{p:rho}
    A continuous function $\phi:\C\to\C$ is of order $\le \rho$ if and only if, for every $\varrho > \rho$, there exists $B > 0$ such that
    \begin{equation}
        |\phi(s)| \le B e^{|s|^\varrho}
    \end{equation}
    for $s\in\C$.
\end{proposition}

\begin{proof}
    By \cref{e:rholimsup},
    \[
    \frac{\log\log |\phi(s)|}{\log |s|} \le \rho + o(1), \quad |s| \to \infty
    \]
    So for all $\varrho > \rho$, there exists $R > 0$ such that $|\phi(s)| \le e^{|s|^{\varrho}}$ for $|s| > R$. Setting $B := \max(1, M_\phi(R))$ suffices.
    
    Conversely, if the stated bound holds for every $\varrho>\rho$, then $M_\phi(r) = O(e^{r^\varrho})$ for every $\varrho>\rho$. Hence the order of $\phi$ is at most $\rho$.
\end{proof}

\begin{proposition}
    A continuous function $\phi:\C\to\C$ is of exponential type if and only if its order $\le 1$ and its $1$-type is finite.
\end{proposition}

\begin{proof}
    If $\phi$ is of exponential type, then $|\phi(s)| \le A e^{B|s|}$ for some $A,B\ge0$, so its order is at most $1$ and its $1$-type is finite. Conversely, if the order of $\phi$ is at most $1$ and its $1$-type is finite, then for some $\varsigma\in\R$ we have $M_\phi(r)=O(e^{\varsigma r})$. Since $\phi$ is continuous, increasing the implicit constant if necessary gives $|\phi(s)| \le A e^{\varsigma|s|}$ for all $s\in\C$ and suitable $A\ge0$.
\end{proof}

\subsection{The Phragmén--Lindelöf indicator function}

\begin{definition}
    For $\t\in\Theta$, the \emph{indicator function} $h_\phi(\t)$ of a function $\phi:\C\to\C$ is the $1$-type of $\phi$ in $e^{i\t} \R_+$.
\end{definition}

The indicator function was first introduced by Phragmén and Lindelöf \cite{phragmen1908} in the form
\begin{equation}\label{e:hphi}
    h_\phi(\t) = \limsup_{r\to\infty} \frac{\log |\phi(re^{i\t})|}{r},
\end{equation}
which follows from \cref{e:typelimsup}. The following estimates follow immediately.
\begin{proposition}\label{p:hsumprod}
    $h_{\phi + \psi} \le \max(h_\phi, h_\psi)$ and $h_{\phi \psi} \le h_\phi + h_\psi$.
\end{proposition}

The indicator is used to measure the growth of entire functions. The following lemmas will be useful in applying this information back to $\phi$.

\begin{proposition}\label{p:h}
    Let $\t\in\Theta$, $h\in\bar\R$, and let $\phi:e^{i\t}\R_+\to\C$ be continuous. The following are equivalent.
    \begin{alphabetize}
        \item $h_\phi(\t) \le h$.
        \item For every $h' > h$, there exists $R\ge 0$ such that $|\phi(re^{i\t})| \le e^{h' r}$ for $r > R$.
        \item For every $h' > h$, there exists $B\ge 0$ such that $|\phi(re^{i\t})| \le B e^{h' r}$ for $r \ge 0$.
    \end{alphabetize}
\end{proposition}

\begin{proof}
    \hfill
    \begin{itemize}
        \item[(a) $\imp$ (b)] This follows from \cref{e:hphi}.
        \item[(b) $\imp$ (c)] Since $|\phi(re^{i\t})| \le e^{h' r}$ for $r>R$ and $\phi$ is continuous on $0\le r\le R$, choosing $B$ large enough gives the bound for all $r\ge0$.
        \item[(c) $\imp$ (a)] Using \cref{e:hphi}, we have $h_\phi(\t) \le h'$ for all $h' > h$, hence $h_\phi(\t) \le h$.
    \end{itemize}
\end{proof}

\section{Growth of analytic functions}\label{s:growth2}

This section proves the main growth estimates needed to extend indicator diagram theory in \Cref{s:IDT}.

\subsection{Order and type for entire functions}

We begin by recalling a standard notion associated with power series. Given
\begin{equation}\label{e:phipow}
    \phi(s) = \sum_{n=0}^\infty \phi_n s^n
\end{equation}
the \emph{radius of convergence} is the value $R_\phi \in [0, \infty]$ such that the series converges for $s\in D_{R_\phi}$ and diverges for $s\notin \bar D_{R_\phi}$. The \emph{Cauchy--Hadamard formula} states that
\begin{equation}\label{e:radconv}
    R_\phi = \liminf_{n\to\infty} \inv{\sqrt[n]{|\phi_n|}}.
\end{equation}
Entire functions correspond to the case $R_\phi = \infty$.

By analogy, for a Laurent series of the form $\varphi(p) := \phi(p^{-1})$, the \emph{radius of divergence} is the value $r_\varphi \in [0, \infty]$ such that the series diverges for $p\in D_{r_\varphi}$ and converges for $p\notin \bar D_{r_\varphi}$. By the change of variables, $r_\varphi = R_\phi^{-1}$, or explicitly
\begin{equation}\label{e:raddiv}
    r_\varphi =  \limsup_{n\to\infty} \sqrt[n]{|\phi_n|}.
\end{equation}

\begin{proposition}[Hadamard]
    If $\phi$ is an entire function given by \cref{e:phipow} and of order $\rho$, then
    \begin{equation}\label{e:ordercoeff}
        \rho = \limsup_{n\to\infty} \frac{n\log n}{\log(1/|\phi_n|)}.
    \end{equation}
\end{proposition}

\begin{proof}
    We refer to \cite[\S 2.2]{boas1954} or \cite[Lec.~1]{levin1996}.
\end{proof}

\begin{proposition}[Lindelöf]
    If $\phi$ is a nonzero entire function given by \cref{e:phipow} and of $1$-type $\sigma$, then
    \begin{equation}\label{e:typecoeff}
        \sigma = \limsup_{n\to\infty} \sqrt[n]{|\phi_n|n!}.
    \end{equation}
\end{proposition}

\begin{proof}
    Let $\ell$ denote the right-hand side of \cref{e:typecoeff}. Assume first that $\ell < \infty$, and let $l > \ell$. There exists $N\in\N$ such that $\sqrt[n]{|\phi_n|n!} \le l$ for all $n \ge N$, or equivalently, $|\phi_n| \le l^n/n!$. Hence,
    \[
    |\phi(s)| \le \sum_{n=0}^\infty |\phi_n| |s|^n \le \sum_{n=0}^{N-1} |\phi_n| |s|^n + \sum_{n=N}^\infty \frac{(l|s|)^n}{n!}.
    \]
    For $|s|$ large enough, the finite sum is dominated by $e^{l|s|}$, and therefore $|\phi(s)| \le 2e^{l|s|}$. It follows that $\sigma \le l$. Since this holds for every $l > \ell$, we obtain $\sigma \le \ell$. In particular, if $\sigma = \infty$, then necessarily $\ell = \infty$.

    Conversely, assume that $\sigma < \infty$, and let $\varsigma > \sigma$. By definition of the $1$-type, $M_\phi(r) \le e^{\varsigma r}$ for all sufficiently large $r$. Applying Cauchy's inequality with $r = n/\varsigma$, we obtain for all sufficiently large $n$
    \[
        |\phi_n| \le \frac{M_\phi(r)}{r^n} \le \frac{e^{\varsigma r}}{r^n} = \left(\frac{e\varsigma}{n}\right)^n.
    \]
    Multiplying by $n!$ and using Stirling's approximation, we get $\sqrt[n]{|\phi_n|n!} \le \varsigma + o(1)$, so that $\ell \le \varsigma$. Since this holds for every $\varsigma > \sigma$, it follows that $\ell \le \sigma$. In particular, if $\ell = \infty$, then necessarily $\sigma = \infty$.

    Therefore, $\sigma = \ell$ in all cases.
\end{proof}

In particular, if $\phi$ is a nonzero entire function, then $\sigma \ge 0$. Equivalently, the zero function is the only entire function with negative $1$-type, namely $-\infty$.

\subsection{The Phragmén--Lindelöf principle}

The \emph{Phragmén--Lindelöf principle}, introduced by Phragmén and Lindelöf as an extension of the \emph{maximum modulus principle} \cite{phragmen1908}, is a fundamental result in the theory of analytic functions. Roughly speaking, it allows one to propagate bounds from the boundary of a domain to its interior under mild additional assumptions. We state here the version for sectors.

\begin{theorem}[Phragmén--Lindelöf]\label{t:phrag}
    Let $\alpha,\beta \in \Theta$ with $\alpha \neq \beta$, and let $\phi$ be holomorphic in $S_{(\alpha,\beta)}$ of order $< \frac\pi{\beta - \alpha}$. If there exists $M > 0$ such that $|\phi(s)| \le M$ for every $s \in S_{\{\alpha,\beta\}}$, then $|\phi(s)| \le M$ for every $s \in S_{[\alpha,\beta]}$.
\end{theorem}

\begin{proof}
    We refer to \cite[\S 1.4]{boas1954} or \cite[Lec.~6]{levin1996}.
\end{proof}

The Phragmén--Lindelöf principle has notable applications to uniqueness theorems, such as the following corollary due to Riesz\footnote{It is sometimes referred to as "Watson's lemma" \cite{loday-richaud2016}, since a weaker version was proved earlier by Watson \cite{watson1911}.} \cite{riesz1920}.

\begin{lemma}[Riesz]\label{l:riesz}
    Let $\phi$ be analytic and of exponential type in the half-plane $\bar\Pi_{0, 0}$. If $h_\phi\pa{\pm\frac\pi2} < 0$, then $\phi\equiv0$.
\end{lemma}

\begin{proof}
    Let $a, b, M > 0$ be such that $|\phi(s)| \le M e^{a|s|}$ for $s\in\Pi_{0,0}$ and $|\phi(iy)| \le M e^{-b|y|}$ for $y\in\R$. Define
    \[
    P(s) := \phi(s) e^{-(a + \sgn(\Im s)ib)s},
    \]
    where $\sgn$ denotes the sign function. Then
    \[
    |P(s)| = |\phi(s)| e^{-a \Re s + b \abs{\Im s}}.
    \]
    The function $P$ is analytic and of order $<2$ in each of the sectors $S_{[-\frac\pi2,0]}$ and $S_{[0,\frac\pi2]}$. Moreover, $|P(s)| \le M$ on $\R_+ \cup i\R$. By the Phragmén--Lindelöf principle, it follows that $|P(s)| \le M$ throughout the half-plane $\Re s \ge 0$, that is,
    \[
    |\phi(s)| \le M e^{a \Re s - b \abs{\Im s}}.
    \]

    Now fix $t > 0$ and set $Q_t(s) := \phi(s)e^{ts}$. Then $|Q_t(iy)| \le M$ for all $y\in\R$. Let $\theta := \arctan\pa{\frac{a + t}{b}} \in \pa{0, \frac\pi2}$. For $s = re^{\pm i\theta}$, we obtain
    \[
    |Q_t(re^{\pm i \theta})| \le M e^{a r\cos\theta - b r\abs{\sin(\pm\theta)}} e^{tr\cos\theta} = M e^{r((a+t)\cos\theta - b\sin \theta)} = M.
    \]
    Since $Q_t$ is of order at most $1$, we may apply the Phragmén--Lindelöf principle once again in the sectors $S_{[-\frac\pi2,-\theta]}$, $S_{[-\theta,\theta]}$, and $S_{[\theta,\frac\pi2]}$. We conclude that $|Q_t(s)| \le M$ throughout the half-plane $\Re s \ge 0$, or equivalently, $|\phi(s)| \le M e^{-t\Re s}$. Letting $t\to\infty$ for $\Re s > 0$, we obtain $\phi \equiv 0$.
\end{proof}

\begin{lemma}\label{l:secbound1}
    Let $\phi$ be an entire function of order $\le 1$ in $S_{[\alpha,\beta]}$, where $\beta - \alpha \in (0, \pi)$ and $\alpha, \beta\in\dom h_\phi$. Then, for $a > h_\phi(\alpha)$ and $b > h_\phi(\beta)$, there exists $B > 0$ such that, for all $r > 0$ and $\t\in[\alpha, \beta]$,
    \begin{equation}\label{e:philemma}
    |\phi(r e^{i\t})| \le B e^{H(\alpha, \t, \beta; a, b) r}
    \end{equation}
    where $H$ is the sinusoid taking the values $a$ and $b$ at the angles $\alpha$ and $\beta$, respectively, namely
    \begin{equation}\label{e:Hab}
        H(\alpha, \t, \beta; a, b) := \frac{a\sin(\beta - \t) + b\sin(\t - \alpha)}{\sin(\beta - \alpha)}.
    \end{equation}
\end{lemma}

\begin{proof}
    By \Cref{p:h}, for $a > h_\phi(\alpha)$ and $b > h_\phi(\beta)$, choose $B > 0$ such that
    \begin{equation}\label{e:P1}
        |\phi(re^{i\alpha})| \le B e^{a r}, \quad |\phi(re^{i\beta})| \le B e^{b r}
    \end{equation}
    for $r>0$. Define
    \[
    P(s) := \phi(s) \exp\pa{\frac{a e^{-i\beta}-b e^{-i\alpha}}{i\sin(\beta - \alpha)} s},
    \]
    for $s\in S_{[\alpha,\beta]}$. Then $|P(s)| = |\phi(s)| e^{-H(\alpha, \arg s, \beta; a, b)|s|}$ for $s\neq 0$. By \cref{e:P1}, $|P(s)| \le B$ on $S_{\{\alpha, \beta\}}$. Since $P$ is holomorphic and of order $\le 1 < \frac{\pi}{\beta - \alpha}$ in $S_{(\alpha, \beta)}$, \Cref{t:phrag} gives $|P(s)| \le B$ in $S_{[\alpha, \beta]}$, which is \cref{e:philemma}.
\end{proof}

\begin{definition}\label{d:E3}
    Let $\E_3$ be the set of entire functions $\phi$ of order $\le 1$ for which $\dom h_\phi$ is a nonempty interval of $\Theta$.
\end{definition}

\begin{theorem}\label{t:hphitcvx}
    For $\phi\in\E_3$, $h_\phi:\Theta\to\bar\R$ is trigonometric-convex.
\end{theorem}

\begin{proof}
    We make use of the characterization of \Cref{p:diameter}.
    \begin{itemize}
        \item Let $\t_1, \t_2, \t_3\in\Theta$ satisfy \cref{e:thetaimp2}, and let $a > h_\phi(\t_1)$ and $b > h_\phi(\t_3)$. By \Cref{l:secbound1}, there exists $B > 0$ such that $|\phi(re^{i\t_2})| \le B e^{H(\t_1, \t_2, \t_3; a, b) r}$ for all $r > 0$, where $H$ is defined by \cref{e:Hab}. By definition of $h_\phi$, this implies $h_\phi(\t_2) \le H(\t_1, \t_2, \t_3; a, b)$. Letting $a\to h_\phi(\t_1)$ and $b\to h_\phi(\t_3)$, we obtain an inequality equivalent to \cref{e:tcvximp}.
        \item Suppose there exists $\alpha\in\dom h_\phi$ such that $h_\phi(\alpha) + h_\phi(\alpha+\pi) < 0$. Since $\dom h_\phi$ is an interval, we may assume without loss of generality that $[\alpha, \alpha+\pi] \subset \dom h_\phi$. Define $P(s) := \phi(sie^{i\alpha})\phi(-sie^{i\alpha})$. Then $P$ is of exponential type in $\Pi_{0,0}$ and, by \Cref{p:hsumprod}, satisfies $h_P\pa{\pm\frac\pi2} \le h_\phi(\alpha) + h_\phi(\alpha+\pi) < 0$. Hence, \Cref{l:riesz} yields $P \equiv 0$, and therefore $\phi \equiv 0$.
    \end{itemize}
    Thus, either $\phi \equiv 0$, in which case $h_\phi \equiv -\infty$, or else $h_\phi$ satisfies \cref{e:diameter} for every $\t\in\dom h_\phi$. By \Cref{p:diameter}, it follows that $h_\phi$ is trigonometric-convex.
\end{proof}

\Cref{t:hphitcvx} allows us to use the properties of trigonometric-convex functions listed in \Cref{s:tcvx}.

\begin{remark}
    Since $h_\phi$ is trigonometric-convex, \Cref{p:domk} shows that \Cref{d:E3} excludes precisely those entire functions of order $\le 1$ for which $\dom h_\phi$ is either empty or of the form $\{\alpha, \alpha+\pi\}$ for some $\alpha\in\Theta$. Thus, $\E_3$ remains a large subclass of the entire functions of order $\le 1$.
\end{remark}

\subsection{The fundamental bound}

A function of finite type in a given direction satisfies the bounds of \Cref{p:h}, but these estimates hold only along the corresponding ray from the origin to infinity. Under the additional global assumption that the function has order at most $1$, we seek to control its growth in tubular neighborhoods of that ray. The first step is to combine these two growth restrictions into a single estimate by means of the Phragmén--Lindelöf principle.

\begin{lemma}\label{l:secbound2}
    Let $\phi$ be entire and of order at most $1$. Let $\t\in\dom h_\phi$, let $h > h_\phi(\t)$, and let $\varrho > 1$. Then there exists $B > 0$ such that
    \begin{equation}\label{e:secbound2}
        |\phi(s)| \le B \exp\pa{\frac{\sin(\varrho \dTheta(\arg s, \t))}{\sin(\varrho\delta)} |s|^\varrho + h\frac{\sin(\delta - \dTheta(\arg s, \t))}{\sin(\delta)} |s|}
    \end{equation}
    for every $\delta \in (0, \pi/\varrho)$ and $s\in S_{[\t-\delta,\t+\delta]}$.
\end{lemma}
The statement uses the metric on angles defined in \cref{e:dTheta}. An important feature of \cref{e:secbound2} is that the constant $B$ is independent of $\delta$.

\begin{proof}
    Let  $\varrho > 1$ and $h > h_\phi(\t)$. By \Cref{p:order,p:h} there exists $B_\varrho, B_h > 0$
    \begin{equation}\label{e:bounds}
        |\phi(s)| \le B_\varrho e^{|s|^\varrho} \quad\text{and} \quad |\phi(re^{i\t})| \le B_h e^{hr},
    \end{equation}
    for every $s\in\C$ and $r > 0$. Fix $\delta\in(0, \pi/\varrho)$ and define
    \[
    P(s) := \phi(s) \exp\pa{\frac{i(e^{-i\t} s)^\varrho}{\sin(\varrho\delta)} + h\frac{ie^{i(\delta - \t)}s}{\sin(\delta)}},
    \]
    so that
    \[
    |P(s)| = |\phi(s)| \exp\pa{- \frac{\sin(\varrho(\arg s - \t))}{\sin(\varrho\delta)} |s|^\varrho - h\frac{\sin(\delta - \t + \arg s)}{\sin(\delta)} |s|}.
    \]
    The function $P$ is holomorphic in $S_{[\t-\delta,\t]}$ and has order at most $\varrho < \pi/\delta$ and, according to \cref{e:bounds}, on the boundary rays we have $|P(re^{i\t})| \le B_h$ and $|P(re^{i(\t - \delta)})| \le B_\varrho$ for every $r > 0$. Therefore, by \Cref{t:phrag},
    \[
    |P(s)| \le B := \max(B_h, B_\varrho)
    \]
    throughout $S_{[\t-\delta,\t]}$. Applying the same argument to $S_{[\t,\t+\delta]}$, we obtain \cref{e:secbound2}.
\end{proof}

We now convert the information from \Cref{l:secbound2} into a more practical estimate for tubes, which will be fundamental in \Cref{s:IDT}.

\begin{theorem}\label{t:fundbound}    
    Let $\phi$ be entire and of order $\le 1$. Let $\t\in\dom h_\phi$, $h > h_\phi(\t)$ and $\varrho > 1$. Then there exists a constant $B > 0$ such that
    \begin{equation}\label{e:fundbound}
        |\phi(s)| \le B e^{r^\varrho + h |s|},
    \end{equation}
    for all $r > 0$ and $s\in e^{i\t}\R_+ + \bar D_r$.
\end{theorem}

The set $e^{i\t}\R_+ + \bar D_r$ consists of the points at distance at most $r$ from the ray $e^{i\t}\R_+$.

\begin{proof}
    Choose $h_0$ with $h_\phi(\t)<h_0<h$, and choose $\varrho_0$ such that
    \[
    1<\varrho_0< \min\pa{2-\inv\varrho, 1 + h - h_0}
    \]
    Then $\varrho_0<\varrho$ and, setting $\varrho_1 :=\frac{1}{2-\varrho_0} < \varrho$. Fix $\delta\in(0,\pi/2)$. Let $r>0$ and $s\in e^{i\t}\R_+ + \bar D_r$. 
    
    By \Cref{p:order}, there is $B_0>0$ such that $|\phi(s)| \le B_0 e^{|s|^{\varrho_0}}$. If $|s|\le r/\sin(\delta)$, then
    \[
    |s|^{\varrho_0}-h|s|\le \frac{r^{\varrho_0}}{\sin(\delta)^{\varrho_0}} + |h| \frac{r}{\sin(\delta)} \le r^\varrho+C_0,
    \]
    for some suitable $C_0 > 0$, since $\varrho > \varrho_0 > 1$. Thus \cref{e:fundbound} holds with $B := B_0 e^{C_0}$.

    We may therefore assume that $|s|>r/\sin(\delta)$. Since $s\in e^{i\t}\R_+ + \bar D_r$, the distance from $s$ to $e^{i\t} \R_+$ is at most $r$. Therefore, $D:=\dTheta(\arg s,\t)$ satisfies
    \begin{equation}\label{e:sinD}
        \sin(D) = \frac{\abs{\Im(se^{-i\t})}}{|s|} \le \frac{r}{|s|} \le \sin(\delta),
    \end{equation}
    which implies $D \le \delta$, because $D, \delta \in [0, \pi/2]$, and thus it follows that $s\in S_{[\t-\delta,\t+\delta]}$. Since $\delta < \pi/\varrho_0$, by \Cref{l:secbound2} there is $B_1>0$ such that
    \begin{equation}\label{e:tubeboundsector}
        |\phi(s)| \le B_1 \exp\pa{\frac{\sin(\varrho_0 D)}{\sin(\varrho_0\delta)} |s|^{\varrho_0} + h_0\frac{\sin(\delta-D)}{\sin(\delta)} |s|}.
    \end{equation}
    The elementary estimates
    \[
    \frac{\sin(\varrho_0 D)}{\sin(\varrho_0 \delta)} \le \frac{\varrho_0 \sin(D)}{\sin(\varrho_0 \delta)}
    \quad\text{and}\quad
    \abs{1 - \frac{\sin(\delta-D)}{\sin(\delta)}} \le \frac{\sin(D)}{\sin(\delta)} 
    \]
    combined with \cref{e:sinD} give the following upper bound for the exponent in \cref{e:tubeboundsector}:
    \[
    \frac{\sin(\varrho_0 D)}{\sin(\varrho_0\delta)} |s|^{\varrho_0} + h_0\frac{\sin(\delta-D)}{\sin(\delta)} |s| \le C r |s|^{\varrho_0-1} + h_0 |s| + |h_0| C r, \quad C := \frac{\varrho_0}{\sin(\varrho_0 \delta)}.
    \]
    Applying Young's inequality to the product $r|s|^{\varrho_0-1}$ gives
    \begin{align*}
        \frac{\sin(\varrho_0 D)}{\sin(\varrho_0\delta)} |s|^{\varrho_0} + h_0\frac{\sin(\delta-D)}{\sin(\delta)} |s|
        &\le \frac{r^{\varrho_1}}{\varrho_1} + (\varrho_0 - 1)|s| + h_0 |s| + |h_0| C_1 r \\
        &\le r^\varrho + C_1 + h |s|
    \end{align*}
    for some suitable $C_1 > 0$ since $\varrho > \varrho_1 > 1$ and $\varrho_0 - 1 < h - h_0$. This proves \cref{e:fundbound} with $B = B_1 e^{C_1}$.
\end{proof}

\begin{theorem}\label{t:phi(n)a}
    Let $\phi$ be entire and of order $\le 1$. Let $\t\in\dom h_\phi$, $h > h_\phi(\t)$ and $\varrho > 1$. Then there exists a constant $B > 0$ such that
    \begin{equation}\label{e:phi(n)a}
        |\phi^{(n)}(s+a)| \le \frac{B n!}{(n/\varrho)!} e^{|a|^\varrho + h |s|}
    \end{equation}
    for all $a\in\C$, $n\in\N$ and $s\in e^{i\t}\R_+$.
\end{theorem}

Here, $(n/\varrho)!$ represents $\Gamma(n/\varrho + 1)$. Note that \Cref{t:phi(n)a} is equivalent to itself with either $a = 0$ or $n = 0$.

\begin{proof}
    Let $r > 0$, $s\in e^{i\t}\R_+$ and $\zeta\in C_r$. For $a\in\C$, $s+a+\zeta\in e^{i\t}\R_+ + \bar D_{r+|a|}$. Let $\varrho > 1$, $\varrho_0 \in (1, \varrho)$ and $n\in\N$. By Cauchy's integral formula and \Cref{t:fundbound}
    \[
    |\phi^{(n)}(s+a)| \le \frac{n!}{2\pi} \oInt_{C_r} |\phi(s+a+\zeta)| \abs{\frac{\d\zeta}{\zeta^{n+1}}} \le \frac{n!}{2\pi} \oInt_{C_r} B_0 e^{(r+|a|)^{\varrho_0} + h |s+a+\zeta|} \frac{\abs{\d\zeta}}{{|\zeta|^{n+1}}},
    \]
    where $B_0 > 0$ is independent of $r$, $a$ and $n$. Given $|s| - |a + \zeta| \le |s+a+\zeta| \le |s| + |a + \zeta|$, it follows that $h|s+a+\zeta| \le h|s| + |h| |a + \zeta|$ and
    \[
    |\phi^{(n)}(s+a)| \le \frac{B_0 n!}{2\pi} \oInt_{C_r} e^{(r+|a|)^{\varrho_0} + h|s| + |h|(|a| + r)} \frac{\abs{\d\zeta}}{r^{n+1}} = \frac{B_0 n!}{r^n} e^{(r + |a|)^{\varrho_0} + |h| (r + |a|) + h|s|}.
    \]
    By a convexity inequality,
    \[
    (r + |a|)^{\varrho_0} + |h| (r + |a|) \le (1+|h|)(r + |a|)^{\varrho_0} \le b (r^{\varrho_0} + |a|^{\varrho_0}), \quad b := (1+|h|) 2^{\varrho_0-1}.
    \]
    The function $r \mapsto e^{br^{\varrho_0}} / r^n$ attains its minimum at $r = \pa{\frac{n}{\varrho_0 b}}^{1/\varrho_0}$, after which we arrive at the bound
    \[
    |\phi^{(n)}(s+a)| \le  B_0 n! \frac{eb}{n/\varrho_0}^{n/\varrho_0} e^{b |a|^{\varrho_0} + h|s|}.
    \]
    Since $b|a|^{\varrho_0} \le |a|^\varrho + C$, for some suitable $C > 0$, and by Stirling's approximation, $\frac{eb}{n/\varrho_0}^{n/\varrho_0} (n/\varrho)!$ is bounded by a constant $B_1 > 0$. \Cref{e:phi(n)a} follows with $B := B_0 B_1 e^C$.
\end{proof}

\begin{proposition}\label{p:hEa}
    For an entire function $\phi$ of order $\le 1$ and $a\in\C$, define $E^a \phi : s \mapsto \phi(s+a)$. We have
    \begin{equation}
        h_{E^a\phi} \equiv h_\phi,
    \end{equation}
     for all $a\in\C$.
\end{proposition}

\begin{proof}
    Let $\t\in\Theta$. We find via \Cref{t:phi(n)a} and \Cref{p:h} that $h_{E^a\phi}(\t)\le h_\phi(\t)$, the case $\t\notin\dom h_\phi$ being trivial. The reverse inequality follows by symmetry.
\end{proof}

\section{Extension of the indicator diagram theory}\label{s:IDT}

\subsection{The directional Laplace transform}\label{s:dirlapl}

\begin{definition}
    For $\t\in\Theta$, the \emph{directional Laplace transform} of $\phi$ in direction $\t$ is
    \begin{equation}\label{e:lapltheta}
        \L_\t \phi(p) := \int_0^{\tilde\infty e^{i\t}} \phi(s) e^{-ps} \d s = \int_0^\infty \phi(t e^{i\t}) e^{-pt e^{i\t}} e^{i\t} \d t.
    \end{equation}
\end{definition}

In many contexts, the Laplace transform corresponds to the case $\t = 0$. Our definition of the nondirectional Laplace transform is given in \Cref{s:WN} and is based on the preliminary results of this subsection. The first of these describes the region of convergence of the directional Laplace transform in terms of the indicator function introduced in \Cref{s:indic}.

\begin{proposition}\label{p:laplcv}
    Let $\phi$ be entire and of order $\le 1$, and let $\t\in\Theta$. Then $\L_\t\phi(p)$ converges if
    \begin{equation}\label{e:laplcv}
        \Re(pe^{i\t}) > h_\phi(\t)
    \end{equation}
    or, in other words, $p\in \Pi_{h_\phi(\t),-\t}$.
\end{proposition}

\begin{proof}
    By \Cref{p:h}, $|\phi(te^{i\t})| = O(e^{h t})$ for every $h > h_\phi(\t)$, and since
    \begin{equation}
        |\L_\t\phi(p)| \le \int_0^{\tilde\infty e^{i\t}} \abs{\phi(s) e^{-ps} \d s} = \int_0^\infty \abs{\phi(te^{i\t})} e^{-\Re(pe^{i\t})t} \d t,
    \end{equation}  
    the directional Laplace transform converges whenever $\Re(pe^{i\t}) > h$. Since this holds for every $h > h_\phi(\t)$, the claim follows.
\end{proof}

Directional Laplace transforms in different directions coincide on the intersection of their half-planes of convergence.

\begin{proposition}\label{p:laplAC}
    Let $\phi\in\E_3$ and let $\alpha, \beta\in\Theta$. Then
    \begin{equation}\label{e:laplAC}
        \L_\alpha \phi(p) = \L_\beta \phi(p), \quad \com{for} p\in \Pi_{h_\phi(\alpha), -\alpha} \cap \Pi_{h_\phi(\beta), -\beta}.
    \end{equation}
\end{proposition}

\begin{proof}
    Since $h_\phi$ is trigonometric-convex, \cref{e:k>-k} shows that the intersection in \cref{e:laplAC} is empty when $\beta - \alpha = \pi$. We may therefore assume that $\beta - \alpha < \pi$. Because $\phi$ is of exponential type in $S_{[\alpha, \beta]}$, there exist $B, H \ge 0$ such that $|\phi(s)| \le B e^{H |s|}$ there. Let $H' > H$ and $p\in\Pi_{H',-\alpha}\cap\Pi_{H',-\beta}$. For $r > 0$, let $\gamma_r$ be the circular arc joining $re^{i\beta}$ to $re^{i\alpha}$. Since $\phi$ is entire, Cauchy's theorem gives
    \begin{equation}\label{e:laplcontour}
        \int_0^{re^{i\beta}} \phi(s) e^{-ps} \d s + \Int_{\gamma_r} \phi(s) e^{-ps} \d s + \int_{re^{i\alpha}}^0 \phi(s) e^{-ps} \d s = 0.
    \end{equation}
    Because the constant function equal to $1$ is trigonometric-convex, for $\t\in[\alpha,\beta]$
    \[
    \begin{gathered}
        \Re(p e^{i\t}) \sin(\beta - \alpha) = \Re(p e^{i\alpha})\sin(\beta - \t) + \Re(p e^{i\beta})\sin(\t - \alpha) \\
        \ge H' (\sin(\beta - \t) + \sin(\t - \alpha)) \ge H' \sin(\beta - \alpha),
    \end{gathered}
    \]
    hence $p \in \Pi_{H',-\t}$ for all $\t\in[\alpha,\beta]$. Therefore,
    \[
    \abs{\Int_{\gamma_r} \phi(s) e^{-ps} \d s} \le \int_\alpha^\beta B e^{(H - H') r} \d\t r \to_{r\to\infty} 0.
    \]
    Letting $r\to\infty$ in \cref{e:laplcontour} yields $\L_\alpha \phi(p) = \L_\beta \phi(p)$, an equality that analytically continues to $p\in \Pi_{h_\phi(\alpha), -\alpha} \cap \Pi_{h_\phi(\beta), -\beta}$.
\end{proof}

\subsection{A Watson--Nevanlinna-type theorem}\label{s:WN}

The following theorem is the main result of this section. It characterizes functions in $\E_3$ by the existence of an associated holomorphic function satisfying growth conditions outside a closed convex set; this function will be our definition of the Laplace transform $\L$ in \Cref{s:lapl}.

\begin{theorem}\label{t:WN}
    Let $(\phi_n)_{n\in\N}\subset\C$. The following are equivalent:
    \begin{alphabetize}
        \item The function $\phi$ defined by
        \begin{equation}\label{e:phi}
            \phi(s) := \sum_{n=0}^\infty \phi_n s^n
        \end{equation}
        belongs to $\E_3$.
        \item There exists a proper closed convex subset $K$ of $\bar\C$ with connected complement, and a holomorphic function $\Phi : \bar\C\setminus K \to \C$ such that
        \begin{equation}\label{e:rem}
            \Phi(p) = \sum_{n=0}^{N-1} \frac{\phi_n n!}{p^{n+1}} + O\pa{\frac{N!}{(N/\varrho)! p^{N+1}}}, \quad p\to C_\infty \setminus K \text{ uniformly}
        \end{equation}
        for all $\varrho > 1$ and uniformly in $N\in\N$.
    \end{alphabetize}
\end{theorem}

The proof of \Cref{t:WN} is divided into the two implications: \textnormal{(a)} $\imp$ \textnormal{(b)} is proved in \Cref{s:lapl}, whereas \textnormal{(b)} $\imp$ \textnormal{(a)} is proved in \Cref{s:invlapl}.

\begin{remark}\label{r:Phi(infty)}
    In case \textnormal{(b)}, $\Phi$ being holomorphic in $\bar\C\setminus K$ means that $\Phi$ is holomorphic in $\C\setminus K$ and that $\Phi(p) \to c \in \C$ as $p\to C_\infty \setminus K$. We then set $\Phi(p) := c$ for $p\in C_\infty \setminus K$.

    In the present case, since $\Phi(p) = O(1/p)$ as $p\to C_\infty\setminus K$, it follows that $c=0$. In particular, if $K$ is bounded, then $\Phi$ can be defined on $\hat\C\setminus K$ with $\Phi(\hat\infty) = 0$. This case is the subject of \Cref{t:EFET} below.
\end{remark}

\begin{remark}
    \Cref{t:WN} is closely related to the \emph{Watson--Nevanlinna theorem} \cite{watson1911,nevanlinna1918} (see also \cite{hardy1949,sokal1980,loday-richaud2016}). That theorem is stated under very general assumptions, whereas our version is tailored to functions in $\E_3$ and therefore yields sharper conclusions. In particular, the estimate \eqref{e:rem} is a stronger form of a \emph{Gevrey asymptotic} \cite{loday-richaud2016}.
\end{remark}

When $K$ is bounded, one can say more. In that case, the partial sums appearing in \eqref{e:rem} actually converge for large $p$, and \Cref{t:WN} reduces to the following classical theorem of Pólya.
\begin{theorem}[Pólya]\label{t:EFET}
    The function $\phi$ defined by the power series \eqref{e:phi} is entire of exponential type $\sigma$ if and only if the Laurent series
    \begin{equation}\label{e:Phisum}
        \Phi(p) := \sum_{n=0}^\infty \frac{\phi_n n!}{p^{n+1}}
    \end{equation}
    has radius of divergence $\sigma$.
\end{theorem}

\begin{proof}
    By \cref{e:typecoeff}, the exponential type of $\phi$ is $\limsup_{n\to\infty} \sqrt[n]{|\phi_n|n!}$, which, by \cref{e:raddiv}, is exactly the radius of divergence of $\Phi$.
\end{proof}

In particular, if $\sigma = \infty$, the series \eqref{e:Phisum} diverges. In this case, its information is recovered from its partial sums through \cref{e:rem}. The map $\Phi \mapsto \phi$ from the formal algebra $p^{-1}\C\bbra{p^{-1}}$ to $\C\bbra{s}$ is commonly called the \emph{Borel transform} \cite{boas1954,levin1996,loday-richaud2016}, a term introduced by Pólya \cite{polya1923,polya1929}. As Pólya himself observed, although \Cref{t:EFET} follows from results already available at the time, the connection between radius of divergence and exponential type was an important new insight \cite[\S 21]{polya1929}.

\subsubsection{The indicator diagram}

The trigonometric convexity of the indicator function, established in \Cref{t:hphitcvx}, suggests the following definition of the \emph{indicator diagram}.

\begin{definition}
    The \emph{indicator diagram} of an entire function $\phi$ of order $\le 1$ is
    \begin{equation}\label{e:Kphi}
        \K_\phi := \W(h_\phi),
    \end{equation}
    which is a closed convex subset of $\bar\C$.
\end{definition}

The following basic property is immediate.
\begin{proposition}\label{p:h=k}
    If $\phi\in\E_3$ then
    \begin{equation}\label{e:h=k}
        h_\phi \equiv \k_{\K_\phi}.
    \end{equation}
\end{proposition}

\begin{proof}
    By \Cref{t:hphitcvx}, the function $h_\phi$ is trigonometric-convex. Therefore, \cref{e:h=k} follows from \cref{e:Kphi} and \Cref{c:wulffinv}.
\end{proof}

\subsubsection{The Laplace transform}\label{s:lapl}

We now prove the forward implication in \Cref{t:WN}, via a construction we call the \emph{Laplace transform}.

\begin{proposition}\label{p:lapl}
    Given a function $\phi\in\E_3$, there exists a unique function $\Phi$ satisfying \textnormal{(b)} of \Cref{t:WN} for $K = \K_\phi^\dagger$ and satisfying $\L_\t \phi(p) = \Phi(p)$ in $\Pi_{h_\phi(\t),-\t}$, for all $\t\in\dom h_\phi$.
    
    We call the map $\L : \phi \mapsto \Phi$ the \emph{Laplace transform}.
\end{proposition}

\begin{example}
    The following example shows that the holomorphy domain furnished by \Cref{t:WN} is, in general, optimal. Given a closed convex subset $K$ of $\bar\C$, choose a sequence $(b_n)_{n\in\N}\subset K$ that is dense in $K$. For a sufficiently rapidly decaying sequence $(c_n)_{n\in\N}$ of nonzero real numbers, the entire function
    \[
    \phi(s) := \sum_{n=0}^\infty c_n e^{b_n s},
    \]
    satisfies
    \begin{equation}\label{e:Phidense}
        \L\phi(p) = \sum_{n=0}^\infty \frac{c_n}{p - b_n}.
    \end{equation}
    It follows that $K = \K_\phi^\dagger$ and that $\L\phi$ cannot be analytically continued further. Without additional information about $\phi$, this is \emph{a priori} the general situation, which shows the optimality of \Cref{t:WN}.
\end{example}

In fact, we prove directly the following translated form of the Gevrey remainder.

\begin{proposition}\label{p:remshift}
    For $\phi\in\E_3$
    \begin{equation}\label{e:remshift}
        \L E^a \phi(p) = \sum_{n=0}^{N-1} \frac{\phi^{(n)}(a)}{p^{n+1}} + O\pa{\frac{e^{|a|^\varrho} N!}{(N/\varrho)! p^{N+1}}}, \quad p \to C_\infty\setminus\K_\phi^\dagger \text{ uniformly}
    \end{equation}
    for all $\varrho > 1$ and uniformly in $N\in\N$ and $a\in\C$.
\end{proposition}

The case $a=0$ is exactly the estimate needed in \Cref{t:WN}. Conversely, the translated version can also be recovered from the unshifted one. Thus this is a convenient strengthening, rather than an additional theorem.

\begin{proof}[of \Cref{p:lapl} and \Cref{p:remshift}]
    Fix $a\in\C$. By \Cref{p:hEa}, $h_{E^a\phi}\equiv h_\phi$, so $\K_{E^a\phi}=\K_\phi$. By \Cref{p:laplAC}, the directional transforms of $E^a\phi$ agree on every nonempty intersection of their half-planes of convergence. They therefore glue to a holomorphic function, denoted by $\L E^a\phi$, on
    \begin{equation}\label{e:C-Kphi}
        \bigcup_{\t\in\Theta} \cond{p}{\Re(pe^{i \t}) > h_\phi(\t)} = \C\setminus \K_\phi^\dagger,
    \end{equation}
    where the equality follows from \Cref{p:h=k} and \Cref{d:wulff}.

    It remains to estimate the asymptotic expansion near $C_\infty\setminus\K_\phi^\dagger$. It suffices to work locally in a closed half-plane: fix $\t\in\dom h_\phi$ and numbers $h,H$ such that $h_\phi(\t)<h<H$. We prove the estimate for $p$ satisfying $\Re(pe^{i\t})\ge H$.

    Repeated integration by parts along the ray $\R_+ e^{i\t}$ gives
    \begin{equation}\label{e:I(z)}
        \L_\t E^a\phi(p) = \sum_{n=0}^{N-1} \frac{\phi^{(n)}(a)}{p^{n+1}} + \inv{p^N}\L_\t E^a\phi^{(N)}(p).
    \end{equation}
    Thus, in $\C\setminus \K_\phi^\dagger$,
    \begin{equation}\label{e:remdef}
        R_{N}(p;a) := \L E^a\phi(p) - \sum_{n=0}^{N-1} \frac{\phi^{(n)}(a)}{p^{n+1}} = \inv{p^N}\L E^a\phi^{(N)}(p).
    \end{equation}

    We now bound the last factor by one more integration by parts. The boundary term at infinity vanishes by \Cref{t:phi(n)a}, because $h<\Re(pe^{i\t})$. We obtain
    \begin{equation}\label{e:IBP}
        \L_\t E^a\phi^{(N)}(p) = \frac{\phi^{(N)}(a)}{p} + \inv{p} \int_0^{\tilde\infty e^{i\t}} \phi^{(N+1)}(s+a)e^{-ps}\d s.
    \end{equation}

    Let $\varrho>1$ and choose $\varrho_0\in(1,\varrho)$. By \Cref{t:phi(n)a}, applied once with $\varrho$ and once with $\varrho_0$, there is a constant $B>0$ such that
    \begin{equation}
        |\phi^{(N)}(a)| \le \frac{B N!}{(N/\varrho)!}e^{|a|^\varrho}
    \end{equation}
    and
    \begin{equation}
        |\phi^{(N+1)}(te^{i\t}+a)| \le \frac{B (N+1)!}{((N+1)/\varrho_0)!}e^{|a|^{\varrho_0}+ht}
    \end{equation}
    for all $t\ge0$, $N\in\N$, and $a\in\C$. Since $\varrho_0<\varrho$, Stirling's formula gives a constant $B_1>0$ such that
    \begin{equation}
        \frac{(N+1)!}{((N+1)/\varrho_0)!} \le \frac{B_1 N!}{(N/\varrho)!}, \quad N\in\N.
    \end{equation}
    Increasing the constant once more to absorb $e^{|a|^{\varrho_0}}\le C e^{|a|^\varrho}$, we get
    \begin{equation}
        |\phi^{(N+1)}(te^{i\t}+a)| \le \frac{B N!}{(N/\varrho)!}e^{|a|^\varrho+ht}.
    \end{equation}
    Therefore,
    \begin{align*}
        \abs{\int_0^{\tilde\infty e^{i\t}} \phi^{(N+1)}(s+a)e^{-ps}\d s}
        &\le \frac{B N!}{(N/\varrho)!}e^{|a|^\varrho}\int_0^\infty e^{-(\Re(pe^{i\t})-h)t}\d t \\
        &\le \frac{B N!}{(N/\varrho)!}e^{|a|^\varrho}.
    \end{align*}
    Combining this estimate with \cref{e:IBP} gives
    \begin{equation}
        \L_\t E^a\phi^{(N)}(p) = O\pa{\frac{e^{|a|^\varrho}N!}{(N/\varrho)!p}},
    \end{equation}
    uniformly in $N$ and $a$. By \cref{e:remdef}, this proves \cref{e:remshift} in the chosen closed half-plane. Since every point of $C_\infty\setminus\K_\phi^\dagger$ admits such a half-plane neighborhood, the proposition follows.
\end{proof}

\begin{proof}[of \textnormal{(a)} $\imp$ \textnormal{(b)} of \Cref{t:WN}]
    Apply \Cref{p:remshift} with $a=0$ and set $\Phi:=\L\phi$. Since $\phi^{(n)}(0)=n!\phi_n$, \cref{e:remshift} gives exactly \cref{e:rem} with $K=\K_\phi^\dagger$. In particular, $\Phi(p)=O(1/p)$ at every point of $C_\infty\setminus\K_\phi^\dagger$, so $\Phi$ extends there by the value $0$. The uniqueness of $\Phi$ follows from the identity theorem on $\bar\C\setminus\K_\phi^\dagger$.
\end{proof}

\subsubsection{The inverse Laplace transform formula}\label{s:invlapl}

In proving \textnormal{(b)} $\imp$ \textnormal{(a)} of \Cref{t:WN}, we also establish the following inverse Laplace transform formula.

\begin{definition}[Contours in $\bar\C$]
    For a subset $K$ of $\bar\C$, we say that $\Gamma$ \emph{encircles} $K$ if $\Gamma = \partial U$ for some neighborhood $U$ of $K$ such that $\conv U \not\supset \C$. If $U$ is a uniform neighborhood, $\Gamma$ is said to encircle $K$ \emph{tightly}.
    
    The integral of a function along a contour $\Gamma$ is performed over $\Gamma \cap \C$, oriented as the positive boundary of $U$, that is, so that $U$ lies locally to the left of $\Gamma$.
\end{definition}

For instance, it is not possible to encircle $S_{[\alpha, \beta]}$ when $\beta - \alpha > \pi$. For a bounded set $K$, tight encirclement simply ensures that the contour remains bounded. As for an unbounded set, it forces the contour to stay asymptotically close to $K$.

\begin{theorem}\label{t:invlapl}
    A function $\phi\in\E_3$ is given in terms of $\Phi := \L\phi$ by
    \begin{equation}\label{e:invlapl}
        \phi(s) = \inv{2\pi i} \oInt_\Gamma \Phi(p) e^{ps} \d p,
    \end{equation}
    for $s \in S_{\dom h_\phi}$ and for any contour $\Gamma$ encircling $\K_\phi^\dagger$ tightly.
\end{theorem}

The region of convergence of \cref{e:invlapl} is \emph{precisely} the sector $S_{\dom h_\phi}$, that is, the sector in which $\phi$ is of finite type. The proof has three steps: we first determine the region of convergence in the interior, then treat the nonzero boundary points of $\partial S_{\dom h_\phi}$, and finally show that the resulting integral indeed recovers $\phi$. In addition, the distinction discussed in \Cref{r:sector} is useful here: only in the case $\dom h_\phi = \Theta$ can we guarantee the convergence of \cref{e:invlapl} at $s = 0$.

\begin{remark}
    Since $\phi$ is entire, one might expect \cref{e:invlapl} to represent $\phi$ on the whole complex plane when $\K_\phi^\dagger$ is unbounded. This can be achieved by considering the Laplace transform $\Phi_a$ of $s \mapsto \phi(s+a)$, $a\in\C$, which satisfies
    \begin{equation}
        \phi(s) = \inv{2\pi i} \oInt_\Gamma \Phi_a(p) e^{p(s-a)} \d p,
    \end{equation}
    for $s \in a + S_{\dom h_\phi}$ and $\Gamma$ encircling $\K_\phi^\dagger$. Varying $a$, one can represent $\phi$ anywhere in $\C$.
\end{remark}

\begin{proof}[of \textnormal{(b)} $\imp$ \textnormal{(a)} of \Cref{t:WN} and \Cref{t:invlapl}]
    To show that the function $\phi$ defined by \cref{e:phi} is entire, we examine its coefficients. Let $p\in\bar\C\setminus K$, and let $R_n$ be as in \cref{e:remdef} with $a = 0$. By \cref{e:rem},
    \begin{equation}\label{e:remopti}
        |R_n(p)| \le \frac{B n!}{(n/\varrho)!|p|^{n+1}}
    \end{equation}
    for every $\varrho > 1$ and all sufficiently large $|p|$, with $B > 0$ independent of $n$ and $p$. We have
    \begin{align*}
        |\phi_n|
        &= \frac{|p|^{n+1}}{n!}\abs{R_n(p) - R_{n+1}(p)} \le \frac{|p|^{n+1}}{n!}\pa{\abs{R_n(p)} + \abs{R_{n+1}(p)}} \\
        &\le \frac{B}{(n/\varrho)!} + \frac{B(n+1)}{((n+1)/\varrho)! |p|} \to_{|p|\to\infty} \frac{B}{(n/\varrho)!}.
    \end{align*}
    Since $K$ is a proper convex subset of $\bar\C$ and $\bar\C\setminus K$ is connected, there exists an infinity in $\bar\C\setminus K$ and a path in $\bar\C\setminus K$ to it. Hence we may let $|p|\to\infty$ in the preceding estimate. Therefore, by \cref{e:radconv}, the power series $\phi$ in \cref{e:phi} has infinite radius of convergence. By \cref{e:ordercoeff}, its order is at most $\varrho$:
    \[
    \limsup_{n\to\infty} \frac{n\log n}{\log(1/|\phi_n|)} \le \varrho.
    \]
    Since this holds for every $\varrho>1$, $\phi$ is entire and of order $\le 1$. It remains to show that $\dom h_\phi$ is an interval. To this end, we prove that \cref{e:invlapl} holds for $s\in S_{\dom \k_{K^\dagger}}$ whenever $\Gamma$ tightly encircles $K$.

    \medskip

    We first determine the region of convergence. If $K$ is bounded, that is, if $\dom \k_K = \Theta$, then the integral
    \begin{equation}\label{e:laplinvcv}
        \inv{2\pi i}\oInt_\Gamma \Phi(p) e^{ps} \d p,
    \end{equation}
    converges for all $s\in\C = S_{\dom \k_{K^\dagger}}$. Otherwise, let $I := \cond{\arg p}{p\in K\cap C_\infty}$. Since $\bar\C\setminus K$ is connected and nonempty, \Cref{p:clcvxinf} implies that $I$ is an interval of $\Theta$ and that $\bar I = [\alpha, \beta]$ for some $\alpha,\beta\in\Theta$ satisfying $\beta - \alpha \le \pi$. It follows that $\dom \k_{K^\dagger}$ is an interval of $\Theta$, and
    \[
    \oline{\dom \k_{K^\dagger}} = \bra{-\alpha + \frac\pi2, -\beta - \frac\pi2}.
    \]
    Since
    \[
    \abs{\oInt_\Gamma \Phi(p) e^{ps} \d p} \le  \oInt_\Gamma \abs{\Phi(p)} e^{\Re(sp)} \abs{\d p}, 
    \]
    where $\Re(sp) = \cos(\arg(sp))|sp|$, the integral in \cref{e:laplinvcv} converges whenever $\cos(\arg(sp)) < 0$ for large $p$. Along either side of the tight contour $\Gamma$, $\arg p$ converges to either $\alpha$ or $\beta$, so for
    \[
    \arg s \in \pa{-\alpha + \frac\pi2, -\beta - \frac\pi2} \subset \dom \k_{K^\dagger},
    \]
    $\arg(sp) = \arg s + \arg p$ eventually belongs to $\pa{\frac\pi2, - \frac\pi2}$ as $p$ grows along $\Gamma$. In other words, $\cos(\arg(sp))$ is eventually negative, which proves the convergence of \cref{e:laplinvcv} in the interior of $S_{\dom \k_{K^\dagger}}$.

    We next treat the nonzero boundary points of $S_{\dom \k_{K^\dagger}}$, when they exist. We have $\k_{K^\dagger}\pa{-\beta - \frac\pi2} = b$ for some $b\in\bar\R$. If $b<\infty$, since $\Phi$ is holomorphic, the contour $\Gamma$ can be deformed so that the branch asymptotic to the direction $\beta$ is eventually a straight line $\ell$ running from some $\xi\in\C\setminus K$ to $(\infty + iy)e^{i\beta}$, where $y := \Im(\xi e^{-i\beta}) > b$. By \cref{e:remopti} with $n=0$, it follows that $\Phi(p) \to 0$ along $\ell$. Therefore, integrating by parts gives
    \begin{equation}\label{e:PhiIBP}
        \Int_\ell \Phi(p)e^{ps} \d p = - \Int_\ell \Phi'(p) \pa{\int_\xi^p e^{sq} \d q} \d p.
    \end{equation}
    Applying Cauchy's integral formula together with \cref{e:remopti}, we obtain, for sufficiently small $r > 0$, the asymptotic bound
    \begin{align}\label{e:Phi'}
        |\Phi'(p)|
        &\le \abs{\Phi'(p) + \frac{\phi_0}{p^2}} + \abs{\frac{\phi_0}{p^2}} \le \inv{2\pi} \oInt_{C_r} \frac{\abs{\Phi(p+\zeta) - \frac{\phi_0}{p+\zeta}}}{|\zeta|^2} \abs{\d \zeta} + \abs{\frac{\phi_0}{p^2}} \nonumber\\
        &\le \oInt_{C_r} \frac{B}{|p+\zeta|^2 r^2} \abs{\d\zeta} + \abs{\frac{\phi_0}{p^2}} \le \frac{B}{(|p| - r)^2 r^2} + \abs{\frac{\phi_0}{p^2}} = O\pa{\inv{p^2}}.
    \end{align}
    For nonzero $s$ satisfying $\arg s = -\beta - \frac\pi2$, we use the bound
    \[
    \abs{\int_\xi^p e^{sq} \d q} = \abs{\frac{e^{s\xi} - e^{ps}}{s}} \le \frac{\abs{e^{s\xi}} + \abs{e^{ps}}}{|s|} = \frac{2e^{|s|y}}{|s|},
    \]
    and \cref{e:PhiIBP,e:Phi'} to obtain
    \[
    \abs{\Int_\ell \Phi(p)e^{ps} \d p} \le \Int_\ell \abs{\Phi'(p)} \frac{2e^{|s|y}}{|s|} \abs{\d p} < \infty.
    \]
    A similar argument applies when $\arg s = -\alpha + \frac\pi2$ and $\k_{K^\dagger}\pa{-\alpha + \frac\pi2} < \infty$. This shows that $S_{\dom \k_{K^\dagger}}$ is the region of convergence of \cref{e:laplinvcv}.

    \medskip

    Finally, we show that \cref{e:laplinvcv} recovers $\phi$. When $K$ is bounded, we use \cref{e:Phisum} and exchange the integral and summation.

    Otherwise, first, a standard Cauchy integral formula argument shows that
    \begin{equation}\label{e:hankelint}
        \inv{2\pi i}\oInt_\Gamma \frac{e^{ps}}{p^{n+1}} \d p = \frac{s^n}{n!},
    \end{equation}
    for the same contour $\Gamma$. For the same reasons as before, this integral also converges in $S_{\dom \k_{K^\dagger}}$.
    
    Let $U$ be a uniform neighborhood of $K$ such that $\Gamma = \partial U \cap \C$. Choose $R$ large enough that $D_R$ intersects $U$, and deform $\Gamma$ to $\Gamma_R := \partial (U \cup D_R)$. By \cref{e:hankelint},
    \begin{equation}\label{e:partialphi}
        \inv{2\pi i}\oInt_\Gamma \Phi(p) e^{ps} \d p = \inv{2\pi i}\oInt_{\Gamma_R} \pa{\Phi(p) - \sum_{n=0}^{N-1} \frac{\phi^{(n)}(0)}{p^{n+1}}} e^{ps} \d p + \sum_{n=0}^{N-1} \frac{\phi^{(n)}(0)}{n!} s^n,
    \end{equation}
    for $s\in S_{\dom \k_{K^\dagger}}$. We split the integral on the right-hand side of \cref{e:partialphi} into the contributions over $C_R\setminus U$ and $\Gamma_R\setminus D_R$. The latter vanishes as $R\to\infty$, so it remains only to bound the former. Using \cref{e:remopti}, for $R$ sufficiently large, there is a constant $B > 0$ such that
    \begin{align*}
        &\abs{\hspace{0.9em}\Int_{C_R\setminus U} \pa{\Phi(p) - \sum_{n=0}^{N-1} \frac{\phi^{(n)}(0)}{p^{n+1}}} e^{ps} \d p}
        \le \Int_{C_R\setminus U} \hspace{0.4em} \abs{\Phi(p) - \sum_{n=0}^{N-1} \frac{\phi^{(n)}(0)}{p^{n+1}}} e^{\Re(sp)} \abs{\d p} \\
        \le{}& \Int_{C_R\setminus U} \frac{B N!}{(N/\varrho)!|p|^{N+1}} e^{\Re(sp)} \abs{\d p} \le \Int_{C_R} \frac{B N! e^{|sp|}}{(N/\varrho)! |p|^{N+1}}\abs{\d p} = \frac{2\pi B N! e^{|s|R}}{(N/\varrho)! R^N}.
    \end{align*}
    Setting $R = N/|s|$, we arrive at
    \[
    \abs{\inv{2\pi i}\Int_{C_R\setminus U} \pa{\Phi(p) - \sum_{n=0}^{N-1} \frac{\phi^{(n)}(0)}{p^{n+1}}} e^{ps} \d p} \le \frac{B N! e^N }{(N/\varrho)! N^N} |s|^N \to_{N\to\infty} 0,
    \]
    by Stirling's approximation. Letting $N\to\infty$ in \cref{e:partialphi}, we obtain
    \[
    \inv{2\pi i}\oInt_\Gamma \Phi(p) e^{ps} \d p = \sum_{n=0}^\infty\frac{\phi^{(n)}(0)}{n!} s^n = \phi(s),
    \]
    for $s\in S_{\dom \k_{K^\dagger}}$.
\end{proof}

\begin{example}[Mellin's inversion formula]
    When $\dom h_\phi = \{0\}$, the indicator diagram is the half-plane $\cond{p\in\bar\C}{\Re p \le h_\phi(0)}$. In this case, the contour in the inverse Laplace transform formula may be chosen to be a vertical line:
    \begin{equation}\label{e:mellininv}
        \phi(s) = \inv{2\pi i} \int_{c-i\infty}^{c+i\infty} \Phi(p) e^{ps} \d p,
    \end{equation}
    where $c > h_\phi(0)$. If we only have $0\in\dom h_\phi$, the contour of \cref{e:invlapl} can be deformed into that of \cref{e:mellininv}. This is \emph{Mellin's inversion formula} \cites{mellin1896}[Chap.~4 \S 5]{doetsch1950}[\S 2.4(ii)]{NIST}, which converges for $s\in S_{\{0\}} = \R_+^*$.
\end{example}

\begin{example}[Pincherle's formula]
    The inversion formula \eqref{e:invlapl} when $\K_\phi$ is bounded predates Mellin and is due to Pincherle \cite{pincherle1926}.    
\end{example}

\subsection{The generalized Pólya theorem}\label{s:polya}

\begin{definition}
    The conjugate diagram of $\phi\in\E_3$ is the smallest closed convex subset $K$ of $\bar\C$ for which $\L\phi$ admits an analytic continuation to $\bar\C\setminus K$.
\end{definition}

Equivalently, if $\L\phi$ analytically continues to $\bar\C\setminus K'$ for some closed convex set $K'\subset\bar\C$, then the conjugate diagram is contained in $K'$. Thus the conjugate diagram is the intersection of all such sets $K'$.

\begin{theorem}[Generalized Pólya theorem]\label{t:polya}
    For $\phi\in\E_3$, the conjugate diagram of $\phi$ is equal to $\K_\phi^\dagger$, the complex conjugate of the indicator diagram.
\end{theorem}

\begin{remark}\label{r:cvxhullsing}
    If $\Phi$ were regarded as a function on all of $\bar\C$, with possible singularities, then \Cref{t:polya} would say that $\K_\phi^\dagger$ is the convex hull of the singularities of $\Phi$.
\end{remark}

\Cref{t:polya} can be reformulated as the classical identity between the indicator function and the support function of the conjugate diagram. Indeed, if $K$ denotes the conjugate diagram of $\phi$, then \Cref{p:h=k} shows that \Cref{t:polya} is equivalent to
\begin{equation}\label{e:polya}
    \k_K(\t) = h_\phi(-\t), \qquad \t\in\Theta.
\end{equation}
\Cref{e:polya} for functions of $\E_2$ is the classical P\'olya theorem \cite{polya1929}. The special case of \Cref{t:polya} where $\K_\phi$ is a singleton can be attributed to Servant \cite{servant1899a,servant1899,borel1928}.

\begin{proof}[of \Cref{t:polya}]
    Let $\Phi := \L\phi$, and let $K$ be the conjugate diagram of $\phi$.
    \begin{itemize}
        \item[( $\subset$ )] By \Cref{t:WN}, $\Phi$ is analytically continuable to $\K_\phi^\dagger$, so by minimality of the conjugate diagram,
        \begin{equation}\label{e:Ksubset}
            K \subset \K_\phi^\dagger.
        \end{equation}
        \item[( $\supset$ )] Since $\Phi'$ is the Laplace transform of $s\mapsto -s\phi(s)$, it follows from \Cref{t:invlapl} that
        \begin{equation}\label{e:invlapl'}
            -s\phi(s) = \inv{2\pi i}\oInt_\Gamma \Phi'(p) e^{ps} \d p,
        \end{equation}
        for $s\in S_{\dom h_\phi}$, where $\Gamma$ is any contour encircling $\K_\phi^\dagger$. Let $-\t\in\dom h_\phi$, let $r \ge 1$, and let $U$ be a uniform neighborhood of $K$. Set $\Gamma := \partial U \cap \C$, so that $\Gamma$ encircles $K$. By \cref{e:Ksubset}, replacing the contour by $\Gamma$ does not shrink the region of convergence in \cref{e:invlapl'}. Moreover, \cref{e:Phi'} gives $\Phi'(p) = O(1/p^2)$, so \cref{e:invlapl'} is absolutely convergent. Hence
        \[
            |\phi(re^{-i\t})| \le \inv{2\pi r} \oInt_\Gamma |\Phi'(p) e^{pre^{-i\t}} \d p| \le \inv{2\pi} \oInt_\Gamma |\Phi'(p)| \exp\pa{\sup_{p\in U}\Re(pe^{-i\t})r} \abs{\d p}.
        \]
        Therefore, for all $r \ge 0$,
        \begin{equation}\label{e:unifbproof}
            |\phi(re^{-i\t})| \le B e^{\k_U(\t)r}, \quad B := \max\pa{\sup_{s\in D_1} |\phi(s)| e^{\max\pa[big]{0,-\min\limits_{\t\in\Theta} \kappa_U(\t)}}, \inv{2\pi} \oInt_\Gamma \abs{\Phi'(p) \d p}}.
        \end{equation}
        By \Cref{p:h}, this implies $h_\phi(-\t) \le \k_U(\t)$. Since this holds for every uniform neighborhood $U$ of $K$, we obtain $h_\phi(-\t) \le \k_K(\t)$. Combining this with \Cref{p:h=k} and the converse of \Cref{l:kless}, we deduce that $\K_\phi^\dagger \subset K$, which completes the proof.
    \end{itemize}
\end{proof}

The proof of the reverse inclusion gives a stronger statement: since the constant $B$ in \cref{e:unifbproof} is uniform in $\t$, it also provides a proof of the following property.

\begin{proposition}\label{p:unifb}
    Let $\phi\in\E_3$ and $k$ be a trigonometric-convex function satisfying $k(\t) > h_\phi(\t)$ for all $\t\in\dom h_\phi$. Then there exists a constant $B > 0$ such that
    \begin{equation}\label{e:unifb}
        |\phi(s)| \le B e^{k(\arg s) |s|},
    \end{equation} 
    for all $s\in S_{\dom h_\phi}$.
\end{proposition}

\Cref{p:unifb} strictly improves on \Cref{p:h} by ensuring the constant $B$ can be chosen independently of $\t = \arg s$. This uniform bound has an application to the following somewhat intuitive, but nontrivial fact.

\begin{corollary}\label{c:maxh}
    Let $\sigma$ be the $1$-type of $\phi\in\E_3$, then $\sigma = \max\limits_{\t\in\Theta} h_\phi(\t)$.
\end{corollary}

\begin{proof}
    It is a general fact of trigonometric-convex functions that they admit a maximum (see \Cref{s:kcont}).  For all $\t\in\Theta$ and $r > 0$, $|\phi(r e^{i\t})| \le M_\phi(r)$ which implies $h_\phi(\t) \le \sigma$, by \Cref{p:type}, and the first inequality.

    By \Cref{p:unifb},
    \[
    |\phi(s)| \le B e^{k(\arg s)|s|} \le B \exp\pa{\sup_{\t\in\Theta} k(\t) |s|}.
    \]
    Thus by definition of the type, $\sigma \le \sup_{\t\in\Theta} k(\t)$. Taking the infimum over the functions $k$ yields the second inequality.
\end{proof}

By \Cref{t:WN}, $\Phi$ can be defined outside of the conjugate diagram $\K_\phi^\dagger$. On the other hand, $\Phi$ can be represented outside $\bar D_\sigma$ by \cref{e:Phisum}. Thus, \Cref{c:maxh} proves that the disk of divergence of \cref{e:Phisum} is precisely the smallest disk centered at the origin and containing $\K_\phi^\dagger$.

\section{Application to analytic continuability}\label{s:AC}

In this section, we prove a theorem characterizing when a complex function can be analytically continued beyond its disk of convergence. \Cref{t:AC} belongs to a broad family of analytic-continuation theorems, which characterize the continuability of $f$ in terms of the interpolability of its coefficients by an analytic function $\phi$ of exponential type in a suitable region. This line of results was initiated by Leau in 1899 \cite{leau1899}, and was subsequently developed in various directions, notably by Le Roy \cite{leroy1900}, Wigert \cite{wigert1900}, Lindelöf \cite{lindelof1905}, Carlson \cite{carlson1921}, and Arakelian \cite{arakelian1992}. Dufresnoy and Pisot \cite{dufresnoy1951} showed that the indicator diagram framework provides a particularly precise formulation; here we generalize it to the extended framework.

Many elements of the theory presented here, even in the infinite-type setting, already appear in a lesser-known 1907 memoir of Braitsev \cite{braitsev1907}, building on the initial works of Borel \cite{borel1899} and Servant \cite{servant1899,servant1899a}.

\subsection{Extension of the Carlson--Dufresnoy--Pisot theorem}

\begin{theorem}\label{t:AC}
    Let $(f_n)_{n\in\N}\subset\C$, and let $L$ be a closed subset of $\bar\C$ such that $\k_L(\t) < \infty$ for $\t\in\bra{-\frac{\pi}{2}, \frac\pi2}$ and such that $L$ is vertically bounded in the sense that
    \begin{equation}\label{e:vertical}
        \k_L\pa{\frac\pi{2}} + \k_L\pa{-\frac\pi{2}} = \sup_{p\in L} \Im p - \inf_{p\in L} \Im p < 2\pi.
    \end{equation}
    Then the following are equivalent:
    \begin{alphabetize}
        \item There exists a function $f$ analytic in $\tilde\C\setminus\exp(-L)$ represented in a neighborhood of the origin by the Taylor series
        \begin{equation}\label{e:ffn}
            f(z) = \sum_{n=0}^\infty f_n z^n,
        \end{equation}
        and a sequence of coefficients $(a_n)_{n\in\N}\subset\C$ such that for all $\varrho > 1$,
        \begin{equation}\label{e:fasymp}
            f(z) - \sum_{n=1}^N \frac{a_n}{z^n} = O\pa{\frac{e^{N^\varrho}}{z^{N+1}}}, \quad z \to \tilde C_\infty \setminus \exp(-L),
        \end{equation}
        uniformly in $N\in\N$.
        \item There exists $\phi\in\E_3$ interpolating $(f_n)_{n\in\N}$ such that $\L\phi$ is analytically continuable to $\bar\C\setminus L$.
    \end{alphabetize}
    Furthermore, $a_n = -\phi(-n)$.
\end{theorem}

\begin{remark}
    We record four comments on the statement.
    \begin{enumerate}
        \item The transform $\L\phi$ is defined on a subset of $\bar\C$, whereas $f$ is defined on a subset of $\tilde\C$. This difference arises from the natural definitions $\exp(\infty + iy) := \tilde\infty e^{iy}$.
        \item The asymptotic condition \eqref{e:fasymp} implies $f(z) = O(1/z)$ as $z\to\tilde C_\infty\setminus\exp(-L)$. Thus $f$ vanishes on $\tilde C_\infty\setminus\exp(-L)$. In particular, if $\exp(-L)\subset\tilde C_\infty$, meaning that all singularities of $f$ lie at infinity, then $f$ is entire.
        \item In the special case $L = \K_\phi^\dagger$, the theorem can be formulated without explicitly mentioning the Laplace transform, since \Cref{t:WN} already gives the analytic continuation of $\L\phi$ to $\bar\C\setminus\K_\phi^\dagger$.
        \item When $L \subset \K_\phi^\dagger$, the assumptions on $\k_L$ can be expressed in terms of $h_\phi$ through \Cref{p:h=k}: they become $h_\phi(\t) < \infty$ for $\t\in\bra{-\frac{\pi}{2}, \frac\pi2}$ and
        \begin{equation}\label{e:hphi2pi}
            h_\phi\pa{\frac\pi{2}} + h_\phi\pa{-\frac\pi{2}} < 2\pi.
        \end{equation}
    \end{enumerate}
\end{remark}

The next proposition records the representations in the extended analytic domains that are obtained in the proof of \Cref{t:AC}. 

\begin{proposition}\label{p:ACformulas}
    For $a\in\C$, let $\Phi_a := \L E^a \phi$ and $\Phi := \Phi_0$. Under the assumptions of \Cref{t:AC},
    \begin{equation}\label{e:phif}
        \phi(s) = \inv{2\pi i}\oInt_\Gamma \pa{f(e^{-u}) + \sum_{\Re s \le n < 0} \phi(n) e^{-nu}} e^{su} \d u,
    \end{equation}
    for all $s\in\C$.
    \begin{equation}\label{e:Phif}
        \Phi(p) = \inv{2\pi i} \oInt_\Gamma \frac{f(e^{-u})}{p-u} \odif u,
    \end{equation}
    for all $p\notin L$.
    \begin{equation}\label{e:fPhi}
        f(z) = \frac{1}{2\pi i}\oInt_\Gamma \frac{\Phi_{-\xi}(p) e^{\xi p}}{1-ze^p} \d p,
    \end{equation}
    for all $z\notin S$ and $\xi\in S_{\dom h_\phi}$, where $\Gamma$ encircles $L$ tightly.
\end{proposition}

\begin{proof}[of \Cref{t:AC}, \Cref{p:braitsev,p:ACformulas}]
The set $S := \exp(-L)$ does not contain $0$. In addition, \cref{e:vertical} implies that $L$ is contained in a horizontal strip of width less than $2\pi$, on which the map $p \mapsto e^{-p}$ is injective; hence it is bijective from $L$ to $S$.
\begin{itemize}
    \item[( $\imp$ )]
    Suppose that $f$ has the asymptotic expansion \eqref{e:fasymp}. Define $\phi_N$ by
    \begin{equation}\label{e:phia_n}
         \phi_N(s) := \inv{2\pi i}\oInt_\Gamma \pa{f(e^{-p}) - \sum_{n=1}^N a_n e^{np}} e^{ps} \d p,
    \end{equation}
    where $\Gamma$ encircles $L$ tightly. Since $L$ is vertically bounded, $\Im p$ is uniformly bounded by some constant $d > 0$ for $p\in\Gamma$. In addition, $\Re p \to -\infty$ as $|p| \to \infty$ along $\Gamma$, so we can apply \cref{e:fasymp}:
    \begin{align}\label{e:phifbound}
        |\phi_N(s)| \le \inv{2\pi}\oInt_\Gamma \abs{\pa{f(e^{-p}) - \sum_{n=1}^N a_n e^{np}} e^{ps} \d p} \le \frac{B}{2\pi}\oInt_\Gamma e^{N^\varrho + \Re((N+1+s)p)} \abs{\d p},
    \end{align}
    where $B > 0$ is independent of $N$. Given that
    \begin{equation}\label{e:phifcv}
        \Re((N+1+s)p) = (N+1 + \Re s)\Re p - \Im s\Im p \le (N+1+\Re s)\Re p + d\abs{\Im s},
    \end{equation}
    the integral absolutely converges for $\Re s > -N-1$.
    
    The sum of exponentials is entire and is present only to improve convergence. If $N\ge M$ and we restrict $\phi_N$ and $\phi_M$ to the common half-plane $\Re s > -M-1$, then the remaining exponential terms may be integrated separately and their integrals vanish. Hence $\phi_N=\phi_M$ on this common domain. In particular, $\phi_N$ analytically continues $\phi_0$, initially defined for $\Re s>-1$, to the half-plane $\Re s > -N-1$. These compatible continuations define an entire function $\phi$.
    
    Thus $\phi(s)$ can be expressed as $\phi_N(s)$ with $N = \max(0, \floor{-\Re s})$, and the representation \eqref{e:phif} is valid for all $s\in\C$. To obtain the ensuing bound, we instead choose $N = \max(0, -\floor{\Re s})$. According to \cref{e:phifbound,e:phifcv},
    \begin{align*}
        |\phi(s)|
        &\le \frac{B}{2\pi}\oInt_\Gamma e^{\max(0,-\floor{\Re s})^\varrho + \min(0,\Re p) + d \abs{\Im s}} \abs{\d p} \\
        &\le B' e^{\abs{\Re s + 1}^\varrho + d \abs{\Im s}}, \quad B' :=  \frac{B}{2\pi}\oInt_\Gamma e^{\min(0,\Re p)} \abs{\d p}.
    \end{align*}
    It follows that $\phi\in\E_3$. Let $\exp_u(s) = e^{us}$.  By Fubini's theorem, for $p\in \Pi_{\k_L(-\t),\t}$,
    \begin{align*}
        \L_\t\phi(p)
        &= \inv{2\pi i}\oInt_\Gamma \pa{f(e^{-u}) + \sum_{n=1}^N \phi(-n) e^{-nu}} \L_\t \exp_u(p) \d u \\
        &= \inv{2\pi i}\oInt_\Gamma \pa{f(e^{-u}) + \sum_{n=1}^N \phi(-n) e^{-nu}} \inv{p-u} \d u \\
        &= \inv{2\pi i}\oInt_\Gamma \frac{f(e^{-u})}{p-u} \d u,
    \end{align*}
    so \cref{e:Phif} follows. Moreover, the contour $\Gamma$ can be taken arbitrarily close to $L$, so $\Phi$ is defined for all $p\in\C\setminus L$.

    Now we show that $\phi$ interpolates $(f_n)_{n\in\N}$. Let $U$ be such that $\partial U = \exp(-\Gamma)$ and $r > 0$. For every $n\in\N$,
    \begin{equation}\label{e:phi(n)}
        \oInt_{\partial(U \cup D_r)} \frac{f(z)}{z^{n+1}} \d z = \oInt_{\partial U \setminus \bar D_r} \frac{f(z)}{z^{n+1}} \d z + \oInt_{C_r \setminus U} \frac{f(z)}{z^{n+1}} \d z.
    \end{equation}
    The first integral on the right-hand side of \cref{e:phi(n)} tends to $0$ as $r\to\infty$. For the second integral, \cref{e:fasymp} gives
    \[
    \abs{\,\oint\limits_{C_r \setminus U} \frac{f(z)}{z^{n+1}} \d z} \le \oInt_{C_r \setminus U} \frac{O(1/r)}{r^{n+1}} \abs{\d z} = O\pa{\inv{r^{n+1}}} \to_{r\to\infty} 0.
    \]
    Since no additional singularities are encircled as $r$ increases, the contour integral in \cref{e:phi(n)} does not depend on $r$, so it vanishes identically. On the other hand, since $0\notin S$ and $S$ is closed, we can choose $\epsilon > 0$ sufficiently small so that $D_\epsilon \cap U = \emptyset$, and then
    \[
    0 = \oInt_{\partial(U \cup D_r)} \frac{f(z)}{z^{n+1}} = \oInt_{C_\epsilon} \frac{f(z)}{z^{n+1}} + \oInt_{\partial U} \frac{f(z)}{z^{n+1}}.
    \]
    Applying Cauchy's integral formula to the first integral and making the bijective change of variables $z = e^{-p}$ in the second, which maps $\partial U$ to $\Gamma$, gives
    \[
    0 = 2\pi i f_n - \oInt_\Gamma f(e^{-p}) e^{np} \d p = 2\pi i(f_n - \phi(n)),
    \]
    i.e., $\phi(n) = f_n$.
    
    \item[( $\isimp$ )] Let $f$ be defined by \cref{e:fPhi}. The integral converges for $\xi\in S_{\dom h_\phi}$ by the same argument used to determine the convergence region of the inverse Laplace transform formula in \Cref{t:invlapl}. Moreover, $f$ is holomorphic for all $z$ such that $\Gamma$ does not encircle a pole of the integrand. If $e^{-u}\notin S$, we can choose a contour $\Gamma$ encircling $L$ but not $u$. Hence $f$ is holomorphic in $\C\setminus S$. Its derivatives are given by
    \begin{equation}
        f^{(n)}(z) = \frac{n!}{2\pi i}\oInt_\Gamma \frac{\Phi_{-\xi}(p) e^{(n+\xi)p}}{1-ze^p} \d p,
    \end{equation}
    Therefore, by \Cref{t:WN}, $f^{(n)}(0)/n! = \phi(n)$ for all $n\in\N$. Hence, by Taylor's theorem, \cref{e:ffn} holds for $|z| < R_f$, where $R_f$ is the radius of convergence of $f$. By \cref{e:radconv},
    \begin{equation}\label{e:hphi(0)}
        \log(1/R_f) = \limsup_{n\to\infty} \frac{\log|\phi(n)|}{n} \le \limsup_{r\to\infty} \frac{\log|\phi(r)|}{r} = h_\phi(0),
    \end{equation}
    where $n$ ranges over the integers and $r$ ranges over the reals. It follows that $R_f \ge e^{-h_\phi(0)} > 0$.
    
    For $N\in\N$ and $\xi\in N + S_{\dom h_\phi}$,
    \begin{align*}
        \sum_{n=1}^N \phi(-n) z^{-n}
        &= \sum_{n=1}^N \pa{\inv{2\pi i}\oInt_\Gamma \Phi_{-\xi}(p) e^{(-n+\xi)p} \d p} z^{-n} \\
        &= \inv{2\pi i}\oInt_\Gamma \Phi_{-\xi}(p) e^{\xi p} \frac{(ze^p)^{-N}-1}{1-ze^p} \d p,
    \end{align*}
    where the interchange is justified by dominated convergence. Thus
    \begin{align*}
        f(z) + \sum_{n=1}^N \phi(-n) z^{-n} = \frac{z^{-N}}{2\pi i}\oInt_\Gamma \Phi_{-\xi}(p) \frac{e^{(\xi-N)p}}{1-ze^p} \d p.
    \end{align*}
    Let $\epsilon > 0$, $\varrho > 1$, and set $\xi = N+\epsilon$. For $\varrho_0\in(1, \varrho)$, \cref{e:remshift} gives $|\Phi_{-N-\epsilon}(p)| = O(e^{(N+\epsilon)^{\varrho_0}}/p) = O(e^{N^\varrho}/p)$. Therefore,
    \begin{align*}
        \abs{f(z) + \sum_{n=1}^N \phi(-n) z^{-n}}
        &\le \frac{|z|^{-N}}{2\pi}\oInt_\Gamma \abs{\Phi_{-N-\epsilon}(p)} \abs{\frac{e^{\epsilon p}}{1-ze^p}} \abs{\d p} \\
        &\le \frac{|z|^{-N-1}}{2\pi}\oInt_\Gamma B \frac{e^{N^\varrho}}{|p|} \abs{\frac{e^{\epsilon p}}{z^{-1}-e^p}} \abs{\d p} \\
        &= O\pa{\frac{e^{N^\varrho}}{z^{N+1}}}.
    \end{align*}
    
    Hence $f$ admits the asymptotic expansion of \cref{e:fasymp} with $a_n = -\phi(-n)$.
\end{itemize}
\end{proof}

\begin{remark}
    While \cref{e:rem} gives a precise form of Gevrey asymptotics, \cref{e:fasymp} is a uniform \emph{Poincaré asymptotic expansion} \cite{loday-richaud2016,NIST}.
\end{remark}

\begin{theorem}[Carlson]\label{t:carlson}
    Let $\phi(s)$ be analytic and of exponential type for $\Re s \ge 0$. If \cref{e:hphi2pi} is satisfied and $\phi(n) = 0$ for all $n\in\N$, then $\phi \equiv 0$.
\end{theorem}

This is Carlson's theorem \cite{carlson1914}; see also \cite{boas1954,levin1996}. Our theory yields only a weaker version, in which $\phi\in\E_3$. A simple proof of the general statement follows from \Cref{l:riesz} \cite{riesz1920}.

\begin{proof}
    We use \Cref{t:AC} with $L = \K_\phi^\dagger$ and, here, $f(z) = 0$ for $z\notin S$. So by \cref{e:phif}, $\phi\equiv0$.
\end{proof}

\begin{proposition}[Braitsev]\label{p:sing}
    Under the assumptions of \Cref{t:AC}, the function $p \mapsto f(e^{-p}) - \Phi(p)$ admits analytic continuation across $L$.
\end{proposition}

Put differently, $f(e^{-p})$ and $\Phi(p)$ have exactly the same singularities.

\begin{proof}
    Applying the Laplace transform to \cref{e:invlapl} gives
    \begin{equation}\label{e:PhiPhi}
        \Phi(p) = \inv{2\pi i} \oInt_\Gamma \frac{\Phi(u)}{p-u} \d u,
    \end{equation}
    where $\Gamma$ encircles $L$ tightly. Subtracting \cref{e:PhiPhi} from the representation of $f(e^{-p})$ in \cref{e:fPhi} gives
    \[
    f(e^{-p}) - \Phi(p) = \inv{2\pi i}\oInt_\Gamma \Phi(u) \pa{\inv{1-e^{u-p}} - \frac{1}{p-u}} \d u,
    \]
    which extends holomorphically across $L$, since $\inv{1-e^{u-p}} - \frac{1}{p-u}$ is holomorphic there.
\end{proof}

\begin{proposition}[Braitsev]\label{p:braitsev}
    Under the assumptions of \Cref{t:AC}, the radius of convergence of \cref{e:ffn} is $e^{-h_\phi(0)}$. Equivalently, for $\phi\in\E_3$ satisfying \cref{e:hphi2pi},
    \begin{equation}\label{e:braitsev}
        \limsup_{n\to\infty} \frac{\log |\phi(n)|}{n} = \limsup_{r\to\infty} \frac{\log |\phi(r)|}{r},
    \end{equation}
    where $n$ ranges over the integers and $r$ ranges over the reals.
\end{proposition}

\begin{proof}
    The function $f(e^{-p})$ is analytic for $\Re p > \log(1/R_f)$. By \Cref{p:sing}, $f(e^{-p})$ and $\Phi(p)$ have the same singularities, so $\Phi$ is analytic in the same half-plane. On the other hand, by \Cref{t:polya}, the largest half-plane of convergence of $\Phi$ is $\Re p > h_\phi(0)$. Hence $\log(1/R_f) \ge h_\phi(0)$. Together with the reverse inequality in \cref{e:hphi(0)}, this proves \cref{e:braitsev}.
\end{proof}

\subsection{The special case of finite type}

We now state the special case in which $\phi$ is of exponential type. In this setting, the Poincaré asymptotic expansion is no longer needed.

\begin{theorem}[Dufresnoy--Pisot]\label{t:ACexp}
    Let $L\subset\C$ satisfy \cref{e:vertical}, and let $(f_n)_{n\in\N}\subset\C$. Then the following are equivalent:
    \begin{alphabetize}
        \item There exists a function $f$ analytic in $\hat\C\setminus \exp(-L)$, represented near the origin by the Taylor series of \cref{e:ffn} with nonzero radius of convergence $R_f$, and satisfying $f(\hat\infty) = 0$.
        \item There exists $\phi\in\E_2$ interpolating $(f_n)_{n\in\N}$ such that $\L\phi$ is analytically continuable to $\C\setminus L$.
    \end{alphabetize}
\end{theorem}

\begin{proof}
    Since $S$ is bounded, \cref{e:fasymp} gives $f(\tilde\infty e^{i\t}) = 0$ for all $\t\in\Theta$. In other words, $f(\hat\infty) = 0$, so $\tilde\C$ can be replaced by the Riemann sphere $\hat\C$.
    
    Conversely, there exists $r > 0$ such that $f$ is analytic in $\C\setminus\bar D_r$. Together with the condition $f(\hat\infty) = 0$, this implies that $g(z) := f(1/z)$ satisfies $g(0) = 0$ and admits a power series with radius of convergence $R_g \ge 1/r$. Hence,
    \begin{equation}\label{e:fa_n}
        f(z) = \sum_{n=1}^\infty a_n z^{-n},
    \end{equation}
    for $|z| > r$, where $a_n := g^{(n)}(0)/n!$. Thus \cref{e:fa_n} is the simplified form of \cref{e:fasymp} in the present case.
\end{proof}

This is the special case of \Cref{t:AC} due to Dufresnoy and Pisot \cite{dufresnoy1951}. The case $L = \K_\phi^\dagger$ was found by Carlson in 1914 \cite{carlson1914,carlson1921}; see also \cite[\S 1.3]{bieberbach1955}.

\begin{proposition}
    Under the assumptions of \Cref{t:ACexp}, let $\Phi := \L\phi$. Then
    \begin{equation}
        f(z) = \sum_{n=0}^\infty \phi(n) z^n, \quad |z| < R_f = e^{-h_\phi(0)},
    \end{equation}
    \begin{equation}\label{e:fphi(-n)}
        f(z) = -\sum_{n=1}^\infty \phi(-n) z^{-n}, \quad |z| > r_f = e^{h_\phi(\pi)},
    \end{equation}
    \begin{equation}\label{e:phifexp}
        \phi(s) = \inv{2\pi i}\oInt_\Gamma f(e^{-u}) e^{su} \d u,
    \end{equation}
    for all $s\in\C$,
    \begin{equation}\label{e:fPhiexp}
        f(z) = \frac{1}{2\pi i}\oInt_\Gamma \frac{\Phi(p)}{1-ze^p} \d p,
    \end{equation}
    for all $z\notin S$ and $\Gamma$ encircling $L$ sufficiently tightly.
\end{proposition}

\begin{proof}
    These formulas follow from \Cref{p:ACformulas}: \cref{e:phifexp} is the simplification of \cref{e:phif} in which the asymptotic terms are no longer needed for convergence, and \cref{e:fPhiexp} is \cref{e:fPhi} with $\xi = 0$. Moreover, \cref{e:fphi(-n)} is obtained by applying \Cref{t:AC} to $s\mapsto\phi(-s)$.
\end{proof}

In the setting of \Cref{t:AC}, the Laurent series in \cref{e:fphi(-n)} diverges when $h_\phi(\pi) = \infty$, that is, when the indicator diagram is unbounded to the left. This is why it must be replaced by the Poincaré asymptotic condition \eqref{e:fasymp}.

\Cref{e:fPhiexp} is due to Le Roy \cite[\S 28]{leroy1900}, who briefly mentioned the formula, which was later studied by Braitsev \cite{braitsev1907}. 

\subsection{Ramanujan's master theorem}\label{s:ram1}

For $z\in\C$ and $\phi\in\E_3$, let $(S_N(z))_{N\in\Z}$ be the sequence of functions satisfying $S_0 = 0$ and $S_{N+1}(z) = S_N(z) + \phi(N) z^N$ for all $N\in\Z$. Thus
\begin{equation}
    S_N(z) = \begin{dcases}
        \sum_{n=0}^{N-1} \phi(n) z^n &\com{if} N \ge 0\\
        -\sum_{n=1}^{-N} \phi(-n) z^{-n} &\com{if} N \le 0.
    \end{dcases}
\end{equation}
If $h_\phi(\pm\frac\pi2) < \pi$, then by \Cref{t:AC} the power series $f(z) := \lim\limits_{N\to\infty} S_N(z)$ analytically continues to $\tilde\C\setminus S$, where $S := \exp(-\K_\phi^\dagger)$, and
\begin{equation}\label{e:fSNO}
    f(z) = S_N(z) + O(z^{N-1}), \quad z\to \tilde C_\infty \setminus S
\end{equation}
for all $N\in\Z$. In addition, we have the following formula for $\phi$ in terms of $f_N := f - S_N$:
\begin{equation}\label{e:phiSN1}
    \phi(s) = \inv{2\pi i}\oInt_\Gamma f_N(e^{-u}) e^{su} \d u,
\end{equation}
where $\Gamma$ encircles $\K_\phi^\dagger$. Let $y_\pm \in (h_\phi(\pm\frac\pi2), \pi]$ and deform $\Gamma$ so that $\abs{\Im u} < \pi$ for $u\in\Gamma$, while the branches of $\Gamma$ with positive and negative imaginary parts are asymptotic to $-\infty + iy_+$ and $-\infty - iy_-$, respectively. In that case, \cref{e:phiSN1} converges for $\Re s > N-1$. Making the change of variables $z = e^{-u}$ in \cref{e:phiSN1} gives
\begin{equation}\label{e:hankel}
    \phi(s) = \inv{2\pi i}\oInt_H \frac{f_N(z)}{z^{s+1}} \d z,
\end{equation}
where the principal branch of $z^{s+1}$ is used and $H$ is a contour from $\tilde\infty e^{-iy_+}$ to $\tilde\infty e^{iy_-}$. Moreover, if $H$ crosses the real axis, it can only do so at a positive value. In other words, $H$ encircles the negative real axis, though not necessarily tightly. Such a contour is known as a \emph{Hankel contour}, especially when $y_+ = y_- = \pi$ where $H$ encircles the negative real axis tightly. Ramanujan's master theorem is obtained by letting $H$ tend to the negative real axis. Since $f(z) - S_N(z) = O(z^N)$ as $z\to 0$ and \cref{e:fSNO} controls the behavior at infinity, for $N-1 < \Re s < N$ we may set $x = -z$ and obtain
\begin{align}\label{e:RMT1}
   \phi(s)
   &= -\inv{2\pi i}\oInt_{-H} \frac{f_N(-x)}{(-x)^{s+1}} \d x \nonumber\\
   &= \inv{2\pi i} \pa{\int_0^\infty \frac{f_N(-x)}{(e^{-i\pi} x)^{s+1}} \d x - \int_0^\infty \frac{f_N(-x)}{(e^{i\pi} x)^{s+1}} \d x} \nonumber\\
   &= \inv{2\pi i}\int_0^\infty \frac{f_N(-x)}{x^{s+1}} (e^{i\pi(s+1)} - e^{-i\pi(s+1)}) \d x \nonumber\\
   &= -\frac{\sin(\pi s)}{\pi}\int_0^\infty \frac{f_N(-x)}{x^{s+1}}  \d x.
\end{align}
\Cref{e:RMT1} is often written as
\begin{equation}\label{e:RMT2}
     \int_0^\infty f_N(-x) x^{s-1}  \d x = \frac\pi{\sin(\pi s)} \phi(-s)
\end{equation}
for $-N < \Re s < -N+1$. In the case $N=0$, this is Ramanujan's master theorem \cite{amdeberhan2012}. The extension \eqref{e:RMT2} is known for functions $\phi$ defined in a half-plane and for positive values of $N$ \cite[Thm.~7.1]{amdeberhan2012}.

\section{Examples}\label{s:ex}

\subsection{Arccotangent and the cardinal sine function}

For $t > 0$, let $\sinc_t(s) := \frac{\sin(t s)}{s}$. Its singularity at the origin is removable, and setting $\sinc_t(0) := t$ makes $\sinc_t$ an entire function of exponential type $t$. Its indicator function is
\begin{equation}
    h_{\sinc_t}(\t) = t \abs{\sin\t},
\end{equation}
and the associated convex body, i.e., the indicator diagram of $\sinc_t$, is the segment
\begin{equation}
    \K_{\sinc_t} = [-it, it].
\end{equation}

To compute the Laplace transform of $\sinc_t$, first we find its derivative:
\[
-(\L_0\sinc_t)'(p) = \int_0^\infty \sin(ts) e^{-ps} \d s = \inv{2i}\pa{\inv{p - it} - \inv{p + it}} = \frac{t}{p^2 + t^2},
\]
for $\Re p\ge 0$. Hence, by analytic continuation, $(\L\sinc_t)' =-\frac{t}{p^2 + t^2}$ is valid for all $p\notin\K_{\sinc_t}$. An antiderivative of $(\L\sinc_t)'$ is
\[
\inv{2i} \log\pa{\frac{p+it}{p-it}} = \cot^{-1}\pa{\frac{p}{t}}.
\]
By \Cref{r:Phi(infty)}, we must have $\L\sinc_t(\hat\infty) = 0$. Therefore
\begin{equation}
    \L\sinc_t(p) = \cot^{-1}\pa{\frac{p}{t}}, \quad p\notin\K_{\sinc_t}.
\end{equation}
According to \Cref{r:cvxhullsing}, $\K_{\sinc_t}$ is the convex hull of the singularities of $\L\sinc_t$. Indeed, $\L\sinc_t$ has two branch points at $\pm it$, and $\K_{\sinc_t}$ is precisely the corresponding branch cut \cite[Fig.~4.23.1]{NIST}. Thus, as predicted by \Cref{t:polya}, $\L\sinc_t$ cannot be further analytically continued. Its Laurent expansion is
\begin{equation}
    \L\sinc_t(p) = \sum_{n=0}^\infty \frac{(-1)^n}{2n+1} \pa{\frac{t}{p}}^{2n+1}.
\end{equation}
It converges when $|p| > t$. Thus $\bar D_t$ is the disk of divergence of $\L\sinc_t$, and it is the smallest closed disk centered at $0$ that contains $\K_{\sinc_t}^\dagger$.

\medskip

We now apply the results of \Cref{s:AC}. Let $f$ be defined by the power series
\begin{equation}
    f(z) := \sum_{n=0}^\infty \sinc_t(n) z^n,
\end{equation}
for $|z| < 1$. The vertical boundedness condition on $\K_{\sinc_t}$ forces us to restrict to $t\in[0,\pi)$ and direct calculations give
\begin{equation}\label{e:fsinc}
    f(z) = t + \inv{2i} \log\pa{\frac{1-ze^{-it}}{1-ze^{it}}} = \cot^{-1}\pa{\frac{\cos t - z}{\sin t}} = \inv{2}\int_{-t}^t \frac{\d\t}{1 - ze^{i\t}}.
\end{equation}
By \Cref{t:ACexp}, $f$ can be analytically continued outside $\exp(-\K_{\sinc_t}) = \cond{e^{i\t}}{\t \in \bra{-t, t}}$, which is its branch cut. On this domain,
\begin{equation}
    f(z) = \inv{2}\int_{-t}^t \frac{\d\t}{1 - ze^{i\t}},
\end{equation}
whereas the other expressions in \cref{e:fsinc}, taken with the principal branches of $\log$ and $\cot^{-1}$, do not have the correct branch cut for every $t\in(0,\pi)$. The representation \eqref{e:fphi(-n)} gives another form of $f$ outside its disk of convergence:
\begin{equation}
    f(z) = -\sum_{n=1}^\infty \sinc_t(-n) z^{-n} = t - f(1/z),
\end{equation}
for $|z| > 1$.

\subsection{The exponential function and the reciprocal Gamma function}\label{s:RGamma}

We next consider $f := \exp$:
\[
\exp(z) = \sum_{n=0}^\infty \inv{n!} z^n.
\]
The function $\exp$ is entire, but it has singularities at infinity in the sens of the radial compactification. More precisely, its singular set in $\tilde\C$ is
\[
S := \cond{\tilde\infty e^{i\t}}{\t\in \bra{-\frac\pi2, \frac\pi2}}.
\]
Moreover, $\exp(z) = o(z^{-N-1})$ as $z\to\tilde\infty e^{i\t}$ with $\t\in\pa{\frac\pi2, -\frac\pi2}$, for every $N\in\N$. Thus \cref{e:fasymp} holds with $a_n = 0$ for all $n\in\N$. Applying \Cref{t:AC}, and using the Hankel-contour form discussed in \Cref{s:ram1}, gives
\[
\phi(s) = \inv{2\pi i} \oInt_H \frac{e^z}{z^{s+1}} \d z,
\]
where $H$ is a Hankel contour around the negative real axis. By Hankel's reciprocal Gamma formula \cites{hankel1864}[eq.~5.9.2]{NIST}, this is $\phi(s) = \frac{1}{\Gamma(s+1)}$. In particular, $\phi$ is entire, $\phi(n) = 1/n!$ for nonnegative integers $n$, and $\phi(-n) = 0$ for $n\ge 1$, as required.

The second consequence of \Cref{t:AC} is that the Laplace transform of the reciprocal Gamma function has singularities contained in $L = -\infty + i[-\frac\pi2,\frac\pi2]$. The same set can also be recovered from the indicator function of $\phi$. By Stirling's approximation over $\C$ \cite[eq.~5.11.7]{NIST},
\[
\inv{\Gamma(s+1)} \sim \frac{e^s}{\sqrt{2\pi}} s^{\inv2 - s} = \sqrt{\frac{s}{2\pi}} e^{s - s\log s},
\]
so $\phi$ is of order $1$ and
\[
\abs{\inv{\Gamma(re^{i\t}+1)}} \sim \sqrt{\frac{r}{2\pi}} e^{\Re (re^{i\t} - re^{i\t}\log(re^{i\t}))} =\sqrt{\frac{r}{2\pi}} e^{r\cos\t-r\log r\cos\t + r \t \sin\t}.
\]
It follows that
\begin{equation}\label{e:hRGamma}
    h_\phi(\t) =
    \begin{cases}
        -\infty &\com{if} \t \in \pa{-\frac\pi{2},\frac\pi{2}}, \\
        \frac\pi2 &\com{if} \t = \pm \frac\pi{2}, \\
        \infty &\com{if} \t \in \bra{\frac\pi{2},-\frac\pi{2}},
    \end{cases}
\end{equation}
and hence, by \Cref{t:polya}, the indicator diagram is $\K_\phi = -\infty + i[-\frac\pi2,\frac\pi2]$. In particular, since $h_\phi(0) = -\infty$, the Laplace transform of the reciprocal Gamma function is given explicitly for all $p\in\C$ by
\begin{equation}
    \L\phi(p) = \int_0^\infty \frac{e^{-pt}}{\Gamma(t+1)} \d t.
\end{equation}
This transform does not seem to have been particularly studied in the literature.

\subsection{The Bernoulli generating function and the Riemann zeta function}\label{s:zeta1}

We now consider the closely related function
\[
f(z) := \inv{1-e^{-z}} - \inv{z} = \sum_{n=0}^\infty\frac{B_{n+1}}{(n+1)!} z^n,
\]
where $(B_n)_{n\in\N}$ are the \emph{Bernoulli numbers} \cite[eq.~24.2.1]{NIST} under the convention $B_1 = 1/2$. The function $f$ has the same singularities at infinity as $\exp$, because it has the same exponential growth in the relevant directions. Unlike $\exp$, however, its Taylor series has finite radius of convergence. Indeed, $f$ has poles at $2\pi i\Z^*$, while the pole at the origin is removable. The closest poles to the origin are at $\pm 2\pi i$, so the largest disk of holomorphy of $f$ has radius $2\pi$. Thus
\[
S := 2\pi i\Z^* \cup \cond{\tilde\infty e^{i\t}}{\t\in \bra{-\frac\pi2,\frac\pi2}}.
\]

In this case the asymptotic coefficients are $a_1 = -1$ and $a_n = 0$ for $n > 1$. Using \Cref{t:AC} and the same reasoning as in \Cref{s:RGamma}, we are led to
\begin{equation}\label{e:zeta/Gamma1}
    \phi(s) = \inv{2\pi i} \oInt_H \frac{f(z) + \inv z}{z^{s+1}} \d z = \inv{2\pi i} \oInt_H \inv{(1 - e^{-z})z^{s+1}} \d z
\end{equation}
where $H$ is a Hankel contour that encircles the negative real axis, but not the singularities $S$ of $f$. According to \cite[eq.~25.5.20]{NIST}, $\phi(s) = -\zeta(-s)/\Gamma(s+1)$. This is essentially the formula Riemann used to analytically continue the zeta function in his landmark paper \cite{riemann1859}, using a Hankel contour before Hankel's own work.

The function $\phi$ interpolates the coefficients of $f$. This lets us recover the value of $\zeta$ at negative integers \cite[eq.~25.6.3]{NIST}:
\begin{equation}\label{e:zeta(-n)}
    \zeta(-n) = -\frac{B_{n+1}}{n+1}.
\end{equation}
By \cref{e:hRGamma} and Riemann's functional equation \cite{riemann1859,titchmarsh1988,NIST}
\begin{equation}\label{e:zeta/Gamma2}
    \phi(s) = 2(2\pi)^{-s-1} \sin\pa{\frac\pi{2} s} \zeta(1+s).
\end{equation}
it follows that $\phi$ is of order $1$ (see \cite{titchmarsh1988}) and that
\begin{equation}
    h_\phi(\t) = \begin{dcases}
    \frac\pi{2} \abs{\sin\t} - \log(2\pi) \cos\t &\com{if} \t\in \bra{-\frac\pi{2}, \frac\pi{2}} \\
    \infty &\com{if} \t\in \pa{\frac\pi{2}, -\frac\pi{2}}.
    \end{dcases}
\end{equation}
For the cases $\t = \pm\frac\pi2$, some classical analytic number theory bounds for $\zeta$ are needed \cite{titchmarsh1988}. Here, $h_\phi$ is the support function of the strip
\[
\K_\phi = \K_\phi^\dagger = \cond{p\in\bar\C}{\abs{\Im p} \le \frac\pi{2}, \Re p \le -\log(2\pi)},
\]
which is the convex hull $\conv L$ of
\begin{equation}\label{e:Lzeta/Gamma}
    L := \pa{-\log(2\pi \N^*) \pm i\frac\pi{2}} \cup \pa{-\infty + i \bra{-\frac\pi2,\frac\pi2}},
\end{equation}
that satisfies $\exp(-L) = S$.

From \cref{e:zeta/Gamma2}, we compute the Laplace transform of $\phi$, for $\Re p > -\log(2\pi)$:
\begin{align*}
    \L_0\phi(p)
    &= \int_0^\infty 2(2\pi)^{-t-1} \sin\pa{\frac\pi{2} t} \sum_{n=1}^\infty \inv{n^{1+t}} e^{-pt} \d t \\
    &= \sum_{n=1}^\infty \inv{2\pi in} \int_0^\infty  \pa{e^{-\pa{\log(2\pi n) - i\frac\pi{2} + p}t} - e^{-\pa{\log(2\pi n) + i\frac\pi{2} + p}t}} \d t \\
    &= \sum_{n=1}^\infty \inv{2\pi in} \pa{\inv{\log(2\pi n) - i\frac\pi{2} + p} - \inv{\log(2\pi n) + i\frac\pi{2} + p}} \\
    &= \sum_{n=1}^\infty\inv{2n \pa{(p + \log(2\pi n))^2 + (\frac\pi2)^2}}.
\end{align*}
By the usual convergence tests, the latter series converges absolutely for all $p\in\C$ away from its poles. Therefore, by analytic continuation,
\begin{equation}
    \L\phi(p) = \sum_{n=1}^\infty\inv{2n \pa{(p + \log(2\pi n))^2 + (\frac\pi2)^2}},
\end{equation}
and the poles of $\L\phi$ are precisely described by the set $L$ given by \cref{e:Lzeta/Gamma}.

\subsection{Trigamma function and the Riemann zeta function}\label{s:zeta2}

Let $\phi(s) = s\zeta(1+s)$. The pole of $\phi$ at $s = 0$ is removable, and setting $\phi(0) = 1$ defines a function of $\E_3$ by the arguments of \Cref{s:zeta1}. For its Laplace transform, we first compute, for $\Re p > 0$,
\begin{align}\label{e:Phi3}
    &\L_0\phi(p)
    = \int_0^\infty t\zeta(1+t) e^{-pt} \d t = \sum_{n=1}^\infty \inv{n} \int_0^\infty \frac{t}{n^t} e^{-pt} \d t \nonumber\\
    ={}& \sum_{n=1}^\infty \inv{n(p+\log n)}\int_0^\infty e^{-(p+\log n)t} \d t = \sum_{n=1}^\infty \inv{n(p+\log n)^2}.
\end{align}
Since \cref{e:Phi3} converges absolutely whenever $p\notin -\log \N^*$, analytic continuation gives, for such $p$,
\begin{equation} 
    \L\phi(p) = \sum_{n=1}^\infty \inv{n(p+\log n)^2}.
\end{equation}
The set of singularities of $\L\phi$ must be closed, so we let $L = -\log\N^*\cup\{-\infty\}$. According to \Cref{t:AC}, the power series
\begin{equation}
    f(z) := \sum_{n=0}^\infty \phi(n) z^n
\end{equation}
analytically continues outside $S := \exp(-L) = \N^* \cup \{\tilde\infty\}$. In fact, $f(z) = -z\psi'(1-z)$, where $\psi = \Gamma'/\Gamma$ is the \emph{digamma function} \cite[eq.~5.7.4]{NIST} and $\psi'$ is the \emph{trigamma function} \cite{NIST}. Thus $S$ corresponds precisely to the singularities of $f$ \cite[eq.~5.2.2]{NIST}.

The second consequence of \Cref{t:AC} is the Poincaré asymptotic expansion for $f$. According to \cref{e:zeta(-n)}, $a_n = -\phi(-n) = -n\zeta(1-n) = B_n$, so for $\varrho > 1$
\begin{equation}\label{e:digasymp}
    -z\psi'(1-z) = \sum_{n=1}^N \frac{B_n}{z^n} + O\pa{\frac{e^{N^\varrho}}{z^{N+1}}}, \quad z\to \tilde C_\infty \setminus S,
\end{equation}
uniformly in $N\in\N$. This expansion is consistent with standard results \cite[eq.~5.11.8, 5.15.8]{NIST}, typically derived via the Euler--Maclaurin summation formula \cite[Sec.~2.10(i)]{NIST}.

\subsection{Further examples}

We conclude with several further examples, stated without proof, to illustrate the scope of the theory. As in \Cref{s:RGamma}, these examples come from hypergeometric series. Their corresponding functions $\phi$ do not seem to admit closed-form Laplace transforms, in contrast with the zeta function examples which are not hypergeometric coefficients.

\subsubsection{Incomplete Gamma function and the falling factorial}

For $a\in\C$, let
\[
\gamma(a, z) := \int_0^z t^{a-1} e^{-t} \d t
\]
be the \emph{incomplete Gamma function} \cite{NIST} and
\[
(a)_s := \frac{a!}{(a-s)!}
\]
for complex $s$; this is the \emph{falling factorial} function. It is known that \cite[eq.~8.7.1]{NIST}
\[
e^z z^{-a-1} \gamma(a+1, z) = \sum_{n=0}^\infty (a)_{-n-1} z^n
\]
for all $z\in\C$, where $(a)_{-n} = \inv{(a+1) \cdots (a+n)}$. Since the function $\phi(s) := (a)_{-s-1}$ is in $\E_3$, \Cref{t:AC} yields
\[
e^z z^{-a-1} \gamma(a+1, z) = -\sum_{n=1}^N \frac{(a)_{n-1}}{z^n} + O\pa{\frac{e^{N^\varrho}}{z^{N+1}}} \quad z\to \tilde C_\infty \setminus S_{\bra{-\frac\pi{2}, \frac\pi{2}}},
\]
for all $\varrho > 1$ and $N\in\N$, where $(a)_n = a(a-1) \cdots (a-n+1)$. This is consistent with the asymptotic expansion of the complementary incomplete Gamma function $\Gamma(a+1,z) := \Gamma(a+1) - \gamma(a+1, z)$; see \cite[eq.~8.11.2]{NIST}.

\subsubsection{The Scorer functions}

The \emph{Scorer function} $\Hi(z)$, which is closely related to the \emph{Airy functions}, is a particular solution of the differential equation $f''(z) - z f(z) = \inv{\pi}$ \cite{NIST}. Its power series is given by \cite[eq.~9.12.17]{NIST}
\begin{equation}
    \Hi(z) = \inv\pi \sum_{n=0}^\infty \Gamma\pa{\frac{n+1}{3}} 3^{\frac{n+1}{3}} \frac{z^n}{n!}.
\end{equation}
A function of $\E_3$ interpolating the coefficients of $\Hi$ is given by
\begin{equation}\label{e:phiHi}
    \phi(s) := \frac{\Gamma\pa{\frac{s+1}{3}} 3^{\frac{s+1}{3}}}{\pi \Gamma(s+1)}.
\end{equation}
From \cref{e:hRGamma}, it follows that
\begin{equation}\label{hphiHi}
    h_\phi(\t) =
    \begin{cases}
        -\infty &\com{if} \t \in \pa{-\frac\pi{2},\frac\pi{2}}, \\
        \frac\pi3 &\com{if} \t = \pm \frac\pi{2}, \\
        \infty &\com{if} \t \in \bra{\frac\pi{2},-\frac\pi{2}}.
    \end{cases}
\end{equation}
Therefore, by Carlson's theorem \Cref{t:carlson}, the function in \cref{e:phiHi} is the unique interpolant of the coefficients of $\Hi$ satisfying $h_\phi(\frac{\pi}{2}) + h_\phi(-\frac{\pi}{2}) < 2\pi$. The associated indicator diagram is $\K_\phi = -\infty + i \bra{-\frac{\pi}{3}, \frac{\pi}{3}}$, and the corresponding (purely infinite) set of singularities of $\Hi$ is
\[
S := \exp(-\K_\phi^\dagger) = \cond{\tilde\infty e^{i\t}}{\t\in\bra{-\frac\pi{3}, \frac{\pi}{3}}}.
\]
It then follows from \Cref{t:AC} that
\[
\Hi(z) = -\inv{\pi z} \sum_{k=0}^N \frac{(3k)!}{k! (3z)^{3k}} + O\pa{\frac{e^{N^\varrho}}{z^{3N+1}}}, \quad z\to \tilde C_\infty \setminus S_{\bra{-\frac\pi{3}, \frac\pi{3}}},
\]
for all $\varrho > 1$ and $N\in\N$, since, on the real axis, $\phi$ vanishes only at the negative integers congruent to $0$ or $1$ modulo $3$. This agrees with \cite[eq.~9.12.27]{NIST}. A similar analysis can be carried out for the Scorer function $\Gi$ \cite{NIST}, but it fails for the Airy functions $\Ai$ and $\Bi$ for the same reason as in the example discussed in the next section.

\subsubsection{Bessel function of the first kind}
Let $z,\alpha\in\C$, and
\[
 z^{-\alpha} J_\alpha(z) := \sum_{m=0}^\infty \frac{(-1)^m z^{2m+\alpha}}{m!(m+\alpha)!2^{2m}}
\]
be the Bessel function of the first kind \cite{NIST}. A function of $\E_3$ interpolating the coefficients of $z^{-\alpha} J_\alpha(z)$ is
\begin{equation}
    \phi(s) = \frac{\cos\pa{\frac\pi{2} s}}{\pa{\frac{s}{2}}!\pa{\frac{s}{2}+\alpha}!2^{s+\alpha}},
\end{equation}
The function $z^{-\alpha} J_\alpha(z)$ provides a non-trivial example of the failure of \Cref{t:AC} when \cref{e:vertical} is not satisfied. From \cref{e:hRGamma}, it follows that
\begin{equation}
    h_\phi(\t) =
    \begin{cases}
        -\infty &\com{if} \t \in \pa{-\frac\pi{2},\frac\pi{2}}, \\
        \pi &\com{if} \t = \pm \frac\pi{2}, \\
        \infty &\com{otherwise}.
    \end{cases}
\end{equation}
Hence Carlson's theorem (\Cref{t:carlson}) does not apply to the uniqueness of the interpolation $\phi$ of the coefficients, and we do not obtain a general asymptotic expansion of the form \eqref{e:fasymp}. Instead, an asymptotic expansion in terms of square roots, sines, and cosines is given in \cite[eq.~10.17.2]{NIST}.

\section{Conclusion and perspectives}

In this paper, we extended indicator diagram theory beyond the classical finite-type setting. The main step was to replace compact convex subsets of $\C$ by closed convex subsets of the extended plane $\bar\C$, allowing unbounded diagrams to be treated as closed objects. This led to an extension of Pólya's theorem and to a characterization of analytic continuability for power series.

The theory developed here still leaves several classes of examples outside its scope. For instance, stability under antidifferentiation in \Cref{t:AC} would require an indicator-diagram theory for meromorphic functions; in that direction, \emph{Nevanlinna theory} \cite{rubel1996,goldberg1997} provides the natural notion of order. Multivalued functions form another extension, where the interaction with \emph{resurgence theory} \cite{loday-richaud2016} should be investigated.

The analytic continuation theorem also places Ramanujan's master theorem within the same framework. A further problem is to determine which other classical integral formulas in complex analysis admit a similar interpretation \cite{NIST,lindelof1905}.

\subsection*{Acknowledgements}

I would like to thank my advisor, Olivier Rioul, for his guidance and support in the early stages of this work.

\appendix
\appendixpage

\section{Proof of the classification of geodesic segments}\label{s:cvxinf}

The proof of \Cref{t:geodesic} is mostly a bookkeeping argument. A segment in $\bar\C$ is obtained as a limit of ordinary Euclidean segments, so the problem is to decide which limits can occur when one or both endpoints escape to infinity.

\begin{proof}[of \Cref{t:geodesic}]
    Let $u_n \to u\in\bar\C$, $v_n \to v\in\bar\C$, and let $t_n\in[0,1]$ satisfy $t_n\to t$. Set
    \[
    w_n := t_n u_n + (1-t_n)v_n \to w\in\bar\C.
    \]
    In each case, we first determine the possible limits $w$, and then show that every such limit is attained for suitable choices of $u_n$, $v_n$, and $t_n$.
    \begin{alphabetize}
        \item In this case, $w = tu + (1-t)v$, so $[u_n,v_n]\to \bra{u,v}$.
        \item Without loss of generality, assume $y_1 \le y_2$ and $\t=0$.
        \begin{itemize}
            \item If $t<1$, then $\Re w=\infty$ and $\Im w = ty_1 + (1-t)y_2 \in (y_1,y_2]$. Given $y_0\in (y_1,y_2]$, choose
            \[
            t = \frac{y_2-y_0}{y_2-y_1}\in[0,1)
            \]
            to obtain $w=\infty+iy_0$.
            \item If $t=1$, then $\Im w = y_1$ and $\Re w \in [x,\infty]$. For $x_0\in[x,\infty)$, choose
            \[
            t_n = \frac{\Re v_n - x_0}{\Re v_n - x},
            \]
            while for $x_0=\infty$, choose a sequence $\epsilon_n\to0$ such that $\epsilon_n\Re v_n\to\infty$ and set $t_n=1-\epsilon_n$. In either case, $w=x_0+iy_1$.
        \end{itemize}
        \item[(c.1)] Again assume $y_1 \le y_2$. Then $\Re(we^{-i\alpha})=\infty$ and
        \[
        \Im(we^{-i\alpha}) = ty_1 + (1-t)y_2 \in [y_1,y_2].
        \]
        Given $y_0\in[y_1,y_2]$, choose
        \[
        t = \frac{y_2-y_0}{y_2-y_1}\in[0,1]
        \]
        when $y_1<y_2$; this yields $w=(\infty+iy_0)e^{-i\alpha}$.
        \item[(c.2)] Since $\beta-\alpha<\pi$ and $\arg w_n$ is a weighted angular mean of $\arg u_n$ and $\arg v_n$, we have $\arg w_n\in[\arg u_n,\arg v_n]$ for all sufficiently large $n$. Passing to the limit gives $\arg w\in[\alpha,\beta]$. Moreover, since $0<\beta-\alpha<\pi$,
        \[
        \Im(v e^{-i\alpha}) = \infty \sin(\beta - \alpha) - y_2 \cos(\beta - \alpha) = \infty.
        \]
        Since $\Im(u e^{-i\alpha}) = y_1 < \infty$, it follows that for all sufficiently large $n$,
        \begin{align*}
            \Im(w_n e^{-i\alpha})
            &= t_n\Im(u_n e^{-i\alpha}) + (1-t_n)\Im(v_n e^{-i\alpha}) \\
            &\ge t_n\Im(u_n e^{-i\alpha}) + (1-t_n)\Im(u_n e^{-i\alpha}) = \Im(u_n e^{-i\alpha}).
        \end{align*}
        Therefore, if $\arg w = \alpha$, then passing to the limit gives $\Im(w e^{-i\alpha}) \ge \Im(u e^{-i\alpha}) = y_1$. Similarly, if $\arg w = \beta$, then $\Im(w e^{-i\beta}) \le y_2$.
        
        We set $u_n = (n + iy_1)e^{i\alpha}$, $v_n = (n + iy_2)e^{i\beta}$.
        \begin{itemize}
        \item Let $\t_0\in(\alpha,\beta)$ and $y_0\in\R$. Set
        \[
        t := \frac{\sin(\beta - \t_0)}{\sin(\beta - \t_0) + \sin(\t_0 - \alpha)}\in(0, 1), \quad 1-t = \frac{\sin(\t_0 - \alpha)}{\sin(\beta - \t_0) + \sin(\t_0 - \alpha)},
        \]
        and define $t_n := t - \frac{c}{n}$, for some $c\in\R$. We claim there exists $c\in\R$ such that $w = (\infty + iy_0) e^{i\t_0}$. For all sufficiently large $n$, $|c|/n \le \min(t, 1-t)$, hence $t_n\in[0,1]$. We have
        \[
        w_n = \pa{t - \frac{c}{n}}(n + iy_1) e^{i\alpha} + \pa{1-t + \frac{c}{n}}(n + iy_2) e^{i\beta},
        \]
        and therefore
        \[
        w_n \sim n(te^{i\alpha} + (1-t)e^{i\beta}) = n\frac{\sin(\beta - \alpha) e^{i\t_0}}{\sin(\beta - \t_0) + \sin(\t_0 - \alpha)},
        \]
        by \Cref{p:sincos}. Thus $|w|=\infty$ and $\arg w=\t_0$. Moreover,
        \begin{align*}
            \Im(w_n e^{-i\t_0})
            &= c(\sin(\t_0 - \alpha) + \sin(\beta - \t_0)) + ty_1 \cos(\alpha - \t_0) + (1-t)y_2\cos(\beta - \t_0) \\
            &- \frac{c}{n} (y_1 \cos(\alpha - \t_0) - y_2\cos(\beta - \t_0)),
        \end{align*}
        so if we set
        \[
        c = \frac{y_0 - ty_1 \cos(\alpha - \t_0) - (1-t)y_2\cos(\beta - \t_0)}{\sin(\t_0 - \alpha) + \sin(\beta - \t_0)}.
        \]
        $\Im(we^{-i\t_0}) = y_0$.
        \item For $\t_0 = \alpha$, we have
        \[
        \Im(w_n e^{-i\alpha}) = t_n y_1 + (1-t_n)(n \sin(\beta - \alpha) + y_2 \cos(\beta - \alpha)).
        \]
        Given $y_0 \ge y_1$, choose
        \[
        t_n =1 - \frac{y_0 - y_1}{n \sin(\beta - \alpha)},
        \]
        which eventually lies in $[0,1]$. Then $\Im(we^{-i\alpha}) = y_0$.
        \item The case $\t_0 = \beta$ with $y_0 \le y_2$ is analogous.
        \end{itemize}
        \item[(c.3)] Without loss of generality, assume $y_1 \le y_2$ and $\alpha = 0$. Clearly, $ y_1 \le \Im w \le y_2$. Fix sequences $(u_n)_{n\in\N}$ and $(v_n)_{n\in\N}$, and let $(\tau_n)_{n\in\N}$ be a sequence in $[0,1]$ converging to $\tau$ such that
        \[
        \tau_n u_n + (1-\tau_n)v_n \to x+iy,
        \]
        for some $x\in\R$ and $y\in[y_1,y_2]$. Then necessarily $\tau\in(0,1)$, and
        \begin{equation}\label{e:wntau}
            w_n = \tau_n u_n + (1 - \tau_n) v_n + (t_n - \tau_n)(u_n - v_n).
        \end{equation}
        We show that the existence of such a sequence $(\tau_n)_{n\in\N}$ restricts the possible limits $w$ of $(w_n)_{n\in\N}$.
        \begin{itemize}
            \item If $\abs{\Re w} < \infty$, then \cref{e:wntau} implies that the sequence $(t_n - \tau_n)(u_n - v_n)$ converges. Since $\Re(u_n - v_n)\to\infty$, we obtain $t_n - \tau_n \to 0$, hence also $(t_n - \tau_n)\Im(u_n - v_n) \to 0$. Therefore $\Im w = y$.
            \item If $\Re w = \pm\infty$, then \cref{e:wntau} gives $(t_n - \tau_n)\Re(u_n - v_n) \to \pm\infty$. Since $\Re(u_n - v_n)\to\infty$, this implies $\pm(t-\tau)\ge0$. Consequently,
            \[
            \pm\Im w = \pm y \pm (t - \tau)(y_1 - y_2) \le \pm y.
            \]
        \end{itemize}
        Next, we show that all these possible limits are attained:
        \begin{itemize}
            \item Let $x_0\in\R$ and set
            \[
            t_n = \tau_n + \frac{x_0 - x}{\Re(u_n - v_n)},
            \]
            which eventually belongs to $[0,1]$. Then \cref{e:wntau} gives $w = x_0 + iy$, so all possible finite limits can be attained.
            \item The case $y_1 = y_2$ is immediate.
            \item If $y_1 < y_2$, let $y_0\in[y_1,y_2]$, define
            \[
            \Delta := \frac{y - y_0}{y_2 - y_1}, \qquad \Delta_n := \Delta\pa{1 - \frac{1}{\sqrt n}}\in(-1,1),
            \]
            and set $u_n = (\Delta_n + 1) n + iy_1$ and $v_n = (\Delta_n - 1) n + iy_2$. We have $\Re u_n\to\infty$ and $\Re v_n\to-\infty$, even when $\Delta = \pm1$. Since $\Re(\tau_n u_n + (1-\tau_n)v_n) \to x$,
            \[
            \tau_n = \frac{x/n - (\Delta_n-1)}{(\Delta_n + 1) - (\Delta_n-1)} + o\pa{\inv n} = \frac{1 - \Delta_n}{2} + O\pa{\inv n}.
            \]
            Now set
            \[
            t_n := \tau_n + \Delta_n = \frac{1+\Delta_n}{2}+ O\pa{\inv n},
            \]
            which eventually lies in $[0,1]$. Then, when $\pm\Delta > 0$, we obtain $w = \pm\infty + iy_0$.
        \end{itemize}
    \end{alphabetize}
\end{proof}

\section{Properties of trigonometric-convex functions}\label{s:tcvxprop}

This appendix collects the elementary properties of trigonometric-convex functions that are used in the paper. The point of including the proofs is not novelty, but compatibility: the functions considered here may take the values $\pm\infty$, and the domain may be a proper subset of the circle.

Trigonometric convexity first appeared in the work of Phragmén and Lindelöf \cite{phragmen1908}, as a structural property of the indicator function they introduced. The same inequality already had a natural geometric meaning for support functions; in that setting it was studied by Blaschke \cite{blaschke1914}. The name itself was coined by Gelfond \cite{gelfond1938} and later became standard through Levin's book \cite{levin1964}.

Several results below are often associated with Levin \cite{levin1964}. Historically, however, many of them entered the literature earlier through convex geometry, especially in the work of Meissner \cite{meissner1909}, Blaschke \cite{blaschke1914}, and Bonnesen \cite{bonnesen1920} on support functions and curvature.

The defining inequality \eqref{e:tcvx} is the trigonometric analog of the increasing-chords property of an ordinary convex function $g$. In one of its symmetric forms, this property says that, for $x_1 < x_2 < x_3$,
\[
g(x_1) (x_2 - x_3) + g(x_2)(x_3 - x_1) + g(x_3)(x_1 - x_2) \leq 0.
\]
Because of this parallel, ordinary convex functions and trigonometric-convex functions share many formal properties. We point out the analogies when they help explain the proof.

\subsection{Basic properties}

\begin{proposition}\label{p:tcvxgeneral}
	If $k\not\equiv-\infty$ is trigonometric-convex, \cref{e:tcvximp} holds more generally for
	\begin{equation}\label{e:thetageneral}
		\t_1, \t_3\in \dom k, \enspace\text{such that}\quad \t_2 - \t_1\le\pi,\enspace \t_3 - \t_2\le\pi, \enspace\text{and}\quad \t_2 \in [\t_1, \t_3]
	\end{equation}
\end{proposition}

This extension of the basic three-point inequality is due to Phragmén and Lindelöf \cite{phragmen1908}. 

\begin{proof}
    The case $\t_3 - \t_1 \le \pi$ follows by definition. When $\t_3 - \t_1 > \pi$, $\t_1 - \t_3 < \pi$ and $\t_2 + \pi \in [\t_3, \t_1]$. Using \cref{e:k>-k} and \cref{e:tcvximp} with $\t_3, \t_2+\pi, \t_1$, we find
    \begin{align*}
        k(\t_2) \sin(\t_3 - \t_1)
        &\le k(\t_2 + \pi) \sin(\t_1 - \t_3) \\
        &\le k(\t_3) \sin(\t_3 - \t_2 - \pi) + k(\t_1) \sin(\t_2 + \pi - \t_3) \\
        &\quad= k(\t_1) \sin(\t_3 - \t_2) + k(\t_3) \sin(\t_2 - \t_3).
    \end{align*}
\end{proof}

\begin{proposition}\label{p:tcvxneg}
	For a trigonometric-convex function $k\not\equiv-\infty$, $\cond{\t\in\Theta}{k(\t) < 0}$ is an interval of length at most $\pi$.
\end{proposition}

\begin{proof}
    If $0 \leq \t_3 - \t_1 < \pi$ and $k(\t_1), k(\t_3) < 0$ then by \cref{e:tcvx}, for all $\t_2 \in [\t_1, \t_3]$
    \[
    k(\t_2) \leq \frac{k(\t_1) \sin(\t_3 - \t_2) + k(\t_3) \sin(\t_2 - \t_1)}{\sin(\t_3 - \t_1)},
    \]
    so $k(\t_2) < 0$. Thus $N := \cond{\t}{k(\t) < 0}$ is an interval, and $\bar N = [\alpha, \beta]$ for some $\alpha, \beta\in\Theta$. If $\beta - \alpha > \pi$, there exists $\t_1\in N$ such that $\t_3 := \t_1+\pi\in N$. With $\t_2 := \t_1 + \pi/2$, \cref{e:tcvximp} leads to a contradiction. Hence, $\beta - \alpha \leq \pi$.
\end{proof}

\begin{proof}[of \Cref{p:domk}]
    Let $\alpha, \beta\in \dom k$.
    \begin{itemize}
        \item If $\beta - \alpha < \pi$, it follows from \cref{e:tcvximp} that $[\alpha, \beta]\subset \dom k$. Therefore, $\dom k$ is an interval. If $\dom k\neq\Theta$, there exists $\alpha, \beta\in\Theta$ so that $\oline{\dom k} = [\alpha, \beta]$. If $\beta - \alpha > \pi$, there is $\alpha', \beta' \in \dom k$ such that $\beta' - \alpha' > \pi$, i.e., $\alpha' - \beta' < \pi$. Hence, $[\beta', \alpha'] \subset \dom k$, which contradicts $\dom k\neq\Theta$, and therefore $\beta - \alpha \le \pi$.
        \item Otherwise, for all $\alpha, \beta\in \dom k$, $\beta - \alpha = \pi$, that is, $\dom k = \{\alpha, \alpha+\pi\}$.
    \end{itemize}
\end{proof}

\begin{proposition}
	Let $k_1, k_2$ be trigonometric-convex functions.
	\begin{enumerate}
		\item $\max(k_1, k_2)$ is trigonometric-convex.
		\item If $k_1$ or $k_2$ is real-valued, $\lambda k_1 + \mu k_2$ is trigonometric-convex for any $\lambda, \mu > 0$.
	\end{enumerate}
\end{proposition}

These properties follow either from the definition or from \Cref{p:kunionsum}.

\subsection{Left and right derivatives}

The following result is the trigonometric analogue of the monotonicity of the slopes of a convex function.

\begin{proposition}\label{p:divdiff}
	A function $k:\Theta\to\R\cup\{\infty\}$ is trigonometric-convex if and only if, for all $\omega\in\dom k$, the function
	\[
		\t\mapsto\frac{k(\t) - k(\omega)\cos(\t - \omega)}{\sin(\t - \omega)}
	\]
	is a nondecreasing function over $\dom k \cap(\omega-\pi, \omega+\pi)\setminus\{\omega\}$.
\end{proposition}

Blaschke first proved this statement for support functions using geometric arguments \cite[\S 9]{blaschke1914}.

\begin{proof}
	Let $\t_1, \t_2\in \dom k \cap(\omega-\pi, \omega+\pi)\setminus\{\omega\}$. If $\t_1 \leq \t_2 < \omega$ or $\omega< \t_1 \leq \t_2$,
	\begin{align}\label{a:ineqq1}
		       & k(\t_1)\sin(\t_2 - \omega) + k(\t_2)\sin(\omega - \t_1) \nonumber                                       \\
		\leq{} & k(\omega)\sin(\t_2 - \t_1) = k(\omega) (\sin(\t_2 - \omega)\cos(\omega - \t_1) + \cos(\t_2 - \omega)\sin(\omega - \t_1)).
	\end{align}
	In other words
	\begin{equation}\label{e:ineqq2}
		(k(\t_1) - k(\omega) \cos(\t_1 - \omega))\sin(\t_2 - \omega) \leq (k(\t_2) - k(\omega) \cos(\t_2 - \omega))\sin(\t_1 - \omega),
	\end{equation}
	hence
	\begin{equation}\label{e:ineqq3}
		\frac{k(\t_1) - k(\omega) \cos(\t_1 - \omega)}{\sin(\t_1 - \omega)} \leq\frac{k(\t_2) - k(\omega) \cos(\t_2 - \omega)}{\sin(\t_2 - \omega)}.
	\end{equation}
	For $\t_1 < \omega < \t_2$, \cref{a:ineqq1,e:ineqq2} are valid with the inequality flipped, from which \cref{e:ineqq3} follows. Thus \cref{e:ineqq3} holds whenever $\t_1 \leq \t_2$, independently of the position of $\omega$, proving monotonicity. Since every step can be reversed, the equivalence follows.
\end{proof}

The following theorem is the trigonometric counterpart of the supporting-line property of convex functions.
\begin{theorem}\label{t:derivative}
	A function $k:\Theta\to\R\cup\{\infty\}$ is trigonometric-convex if and only if, for all $\omega\in\dom k$, there exists a real number $k'(\omega)$ such that for all $\t$
	\begin{equation}\label{e:tangent}
		k(\t) \geq k'(\omega)\sin(\t - \omega) + k(\omega)\cos(\t - \omega).
	\end{equation}
	Furthermore, $k$ has finite right and left-hand side derivatives $k'_\pm$ satisfying
	\begin{equation}\label{e:k'+-}
		k'_-(\t) \leq k'_+(\t)
	\end{equation}
	for all $\t\in\dom k$, and $k'(\omega)$ can be any value of $[k'_-(\omega), k'_+(\omega)]$.
\end{theorem}

\begin{proof}[of \Cref{t:derivative}]
	By \Cref{p:divdiff}, as $\t_1 \to \omega^-$, the left-hand side of \cref{e:ineqq3} increases, while as $\t_2 \to \omega^+$, the right-hand side decreases. Both of these limits thus exist, are finite and with the former being less than or equal to the latter. As $\t\to\omega^\pm$
	\begin{align*}
		\frac{k(\t) - k(\omega) \cos(\t - \omega)}{\sin(\t - \omega)}
		 &= \frac{k(\t) - k(\omega) (1+(\t-\omega)^2/2 + o((\t-\omega)^2))}{\t - \omega + o(\t - \omega)}  \\
		 &= \frac{k(\t) - k(\omega)}{\t - \omega} (1+o(1)) + k(\omega) \frac{\t-\omega}{2} + o(\t-\omega) \\
		 &\sim\frac{k(\t) - k(\omega)}{\t - \omega},
	\end{align*}
	therefore these limits are precisely the left and right derivatives of $k$. Thus, for $\t_1<\omega<\t_2$,
	\begin{equation}\label{e:ineqtangent}
		\frac{k(\t_1) - k(\omega) \cos(\t_1 - \omega)}{\sin(\t_1 - \omega)} \leq k'_-(\omega) \leq k'_+(\omega) \leq\frac{k(\t_2) - k(\omega) \cos(\t_2 - \omega)}{\sin(\t_2 - \omega)}.
	\end{equation}
	\Cref{e:ineqtangent} implies \cref{e:tangent} for any $\t\in\Theta$, $\omega\in\dom k$ and $k'(\omega)\in [k'_-(\omega), k'_+(\omega)]$.

	For the converse, let $\t_1, \t_2, \t_3\in\dom k$ such that \cref{e:thetaimp} holds. According to \cref{e:tangent},
	\begin{align}
		k'(\t_2)\sin(\t_1 - \t_2) +  k(\t_2) \cos(\t_1 - \t_2) & \leq k(\t_1) \label{a:tangent1}  \\
		k'(\t_2)\sin(\t_3 - \t_2) +  k(\t_2) \cos(\t_3 - \t_2) & \leq k(\t_3). \label{a:tangent2}
	\end{align}
	By multiplying \cref{a:tangent1} by $\sin(\t_3 - \t_2)$ and adding it to \cref{a:tangent2} times $\sin(\t_2 - \t_1)$, we arrive at \cref{e:tcvximp}.
\end{proof}

Minkowski proved the continuity of support functions in 1903 \cite{minkowski1903}, and Borel had already anticipated the analogous regularity for indicator functions in 1901 \cite{borel1928}. Phragmén and Lindelöf later proved the stronger local boundedness of $(k(\t+\delta)-k(\t))/\delta$ as $\delta\to0$ \cite[\S 13]{phragmen1908}. This estimate comes close to semi-differentiability, but does not by itself prove it.\footnote{Chebotarev and Meiman \cite[\S 12]{chebotarev1949} claimed that it did imply semi-differentiability, a mistake later repeated by Boas \cite[\S 5.1]{boas1954}.} Levin gave a correct proof of semi-differentiability \cite[Sec.~16]{levin1964}, apparently without using the divided-difference formulation of \Cref{p:divdiff} or the supporting inequality \eqref{e:tangent}. In the support-function setting, \cref{e:tangent} had already been proved by Blaschke \cite{blaschke1914}. The following special case of \Cref{t:derivative} was known to Phragmén and Lindelöf \cite{phragmen1908}.

\begin{corollary}[Phragmén--Lindelöf]\label{c:k'=0}
	If $k:\Theta\to\R\cup\{\infty\}$ is a trigonometric-convex function and $\t^*\in\dom k$ is a local maximum, $k$ is differentiable at $\t^*$ and $k'(\t^*) = 0$.
\end{corollary}

\begin{proof}
    $\t^*$ satisfies $k'_-(\t^*) \geq 0$ and $k'_+(\t^*) \leq 0$, but since $k'_-(\t^*) \leq k'_+(\t^*)$, according to \Cref{t:derivative}, $k'(\t^*)$ exists and equals $0$.
\end{proof}

\subsection{Continuity and maximum}\label{s:kcont}

\begin{proposition}\label{p:tcvxcont}
	If a trigonometric-convex function $k:\Theta\to\bar\R$ is finite over an interval, then it is continuous there.
\end{proposition}

\begin{proof}
    By \Cref{t:derivative}, $k$ has finite one-sided derivatives at every point of the interval. Hence the corresponding one-sided difference quotients have finite limits, which implies one-sided continuity at each point. The two one-sided limits agree at interior points, so $k$ is continuous on the interval.
\end{proof}

\begin{proposition}\label{p:tcvxmax}
	If $k:\Theta\to\bar\R$ is trigonometric-convex and $I$ is a closed interval of $\Theta$, then $k$ admits a maximum over $I$.
\end{proposition}

\begin{proof}
    If there is $\t\in I$ such that $k(\t) = \infty$, then $\infty$ is the maximum of $k$. If there is $\t\in I$ such that $k(\t) = -\infty$, then it follows from \Cref{t:tcvxK} and \Cref{t:cvxinf} that $k$ admits a maximum. Otherwise, by \Cref{p:tcvxcont}, $k$ is continuous over $I$ and by the extreme value theorem, $k$ admits a maximum.
\end{proof}

The following is a counterexample to the claim of \Cref{p:tcvxmax} with $I = \dom k$.

\begin{example}
	If $S = \cond{z}{(\Im z)^2 \le -\Re z}$, then via considering the tangents of the parabola $S$, one computes
	\begin{equation}
		\k_S(\t) =
		\begin{dcases}
			\inv{4}\sin(\t)\tan(\t) &\com{if} \t \in \pa{-\frac\pi2,\frac\pi2} \\
			\infty &\com{otherwise.}
		\end{dcases}
	\end{equation}
    Here $\dom \k_S = (-\frac\pi2, \frac\pi2)$ and $\sup_{\t\in\dom \k_S} \k_S(\t) = \infty$.
\end{example}

\subsection{Weak second derivative}

\begin{theorem}\label{t:k'+intk}
	Let $k:\Theta\to\R\cup\{\infty\}$. $k$ is trigonometric-convex if and only if $k$ admits finite left and right-hand side derivatives in $\dom k$ and for (real) $\t_1 < \t_2$
	\begin{equation}\label{e:k'+intk}
		k'_+(\t_1) \leq k'_-(\t_2) + \int_{\t_1}^{\t_2} k(\omega) \d\omega.
	\end{equation}	
\end{theorem}

\begin{corollary}\label{c:k'inc}
	Let $k:\Theta\to\R\cup\{\infty\}$, $\t_0\in\R$ and $t\in [0, 1]$. $k$ is trigonometric-convex if and only if $k$ admits finite left and right-hand side derivatives in $\dom k$ and the following function is nondecreasing:
	\begin{equation}\label{e:k'inc}
		\left\{
		\begin{array}{@{}r@{}l}
		\R &{}\to\bar\R \\
		\t &{}\mapsto k'_t(\t) + \displaystyle\int_{\t_0}^\t k(\omega) \d\omega,
		\end{array}
		\right.
	\end{equation}
	where $k'_t(\t) := tk'_-(\t) + (1-t)k'_+(\t)$.
\end{corollary}

This property was first proved by Blaschke for support functions \cite[\S 9]{blaschke1914} and later rediscovered by Levin \cite[Sec.~16]{levin1964}. It is the analogue of the monotonicity of the derivative of an ordinary convex function.

\begin{proof}[of \Cref{t:k'+intk} and \Cref{c:k'inc}]
	First we show that trigonometric-convexity implies \cref{e:k'+intk}. The integral appears by passing from a discrete estimate to its Riemann-sum limit. Let $\t_1 < \t_2$, $n\in\N$ and $\omega_p := \t_1+ (\t_2 - \t_1)p/n$, for $p\in\{0, \dots, n\}$. By \cref{e:tangent}
	\[
		\begin{cases}
			k(\omega_p) \geq k'_-(\omega_{p+1})\sin(\omega_p - \omega_{p+1}) + k(\omega_{p+1})\cos(\omega_p - \omega_{p+1}) \\
			k(\omega_{p+1}) \geq k'_+(\omega_p)\sin(\omega_{p+1} - \omega_p) + k(\omega_p)\cos(\omega_{p+1} - \omega_p).
		\end{cases}
	\]
	Adding these two inequalities leads to
	\[
		(k'_-(\omega_{p+1})-k'_+(\omega_p))\sin(\omega_{p+1} - \omega_p) + (k(\omega_p) + k(\omega_{p+1}))(1 -\cos(\omega_{p+1} - \omega_p)) \geq 0,
	\]
	and using the identities $(1-\cos\t)/\sin\t = \tan(\t/2)$ and $\omega_{p+1} - \omega_p = (\t_2 - \t_1)/n$, we arrive at
	\begin{equation}\label{e:summand}
		k'_-(\omega_{p+1})-k'_+(\omega_p) + (k(\omega_p) + k(\omega_{p+1})) \tan\pa{\frac{\t_2 - \t_1}{2n}} \geq 0.
	\end{equation}
	Now we can sum \cref{e:summand} over $p$ and obtain
	\begin{equation}\label{e:riemannsum}
		k'_-(\t_2) - k'_+(\t_1) + \tan\pa{\frac{\t_2 - \t_1}{2n}}\sum_{p=0}^{n-1} (k(\omega_p) + k(\omega_{p+1})) \geq 0.
	\end{equation}
	As $n\to\infty$
	\begin{align*}
		\tan\pa{\frac{\t_2 - \t_1}{2n}}\sum_{p=0}^{n-1} (k(\omega_p) + k(\omega_{p+1}))
		 & = \tan\pa{\frac{\t_2 - \t_1}{2n}}\pa{2\sum_{p=1}^{n-1} k(\omega_p) + k(\t_1) + k(\t_2)} \\
		 & = \frac{\t_2 - \t_1}{n}(1 + o(1)) \sum_{p=1}^{n-1} k(\omega_p) + o(1)                   \\
		 & \to_{n\to\infty} \int_{\t_1}^{\t_2} k(\omega) \d\omega,
	\end{align*}
	by Riemann sums. So as $n\to\infty$, \cref{e:riemannsum} reads
	\begin{equation}\label{e:ineqk'}
		k'_-(\t_2) - k'_+(\t_1) + \int_{\t_1}^{\t_2} k(\omega) \d\omega\geq 0.
	\end{equation}
	To prove the direct implication of \Cref{c:k'inc}, we use \cref{e:ineqk',e:k'+-} to find that, for $\t_0\in\dom k$,
	\[
	 k'_\pm(\t_1) + \int_{\t_0}^{\t_1} k(\omega) \d\omega\leq k'_\pm(\t_2) + \int_{\t_0}^{\t_2} k(\omega) \d\omega,
	\]
	which implies that for $k'_t(\t) := tk'_-(\t) + (1-t) k'_+(\t)$, $t\in [0, 1]$, the function
	\[
	L(\t) := k'_t(\t) + \int_{\t_0}^\t k(\omega) \d\omega
	\]
	is nondecreasing.
	
	Now we prove the converses. Let $\t_1, \t_2\in\dom k$. Two integrations by parts give the Stieltjes integral identity
	\[
	\int_{\t_1}^{\t_2} \sin(\t_1 - \omega) \d k'(\omega) = \sin(\t_1 - \t_2) k'(\t_2) + \cos(\t_1 - \t_2)k(\t_2) - k(\t_1) - \int_{\t_1}^{\t_2} \sin(\t_1 - \omega) k(\omega) \d\omega,
	\]
	that is
	\begin{equation}\label{e:stiel1}
		k'_t(\t_2)\sin(\t_1 - \t_2) + k(\t_2)\cos(\t_1 - \t_2) = k(\t_1) + \int_{\t_1}^{\t_2} \sin(\t_1 - \omega) \d L(\omega) \leq k(\t_1),
	\end{equation}
	since $L$ is nondecreasing and the integrand is nonpositive if $\t_2 \in [\t_1, \t_1 + \pi]$ and nonnegative if $\t_2 \in [\t_1 - \pi, \t_1]$. It thus follows from \Cref{t:derivative} that $k$ is trigonometric-convex.
\end{proof}

\begin{theorem}\label{t:k'cont}
	If $k$ is trigonometric-convex, then $k'_\pm$ admit left and right limits in $\dom k$, and their discontinuities are at most countably many jump discontinuities. Additionally, for $\t\in\dom k$,
	\begin{equation}
		k'_+(\t^\pm) = k'_-(\t^\pm) = k'_\pm(\t).
	\end{equation}
\end{theorem}

\begin{proof}
	Let $\t_1, \t_2\in\dom k$ such that $0 < \t_2 - \t_1 < \pi$. By \Cref{c:k'inc},
	\[
	\t\mapsto k'_\pm(\t) + \displaystyle\int_{\theta_1}^\t k(\omega) \d\omega
	\]
	is nondecreasing over $[\t_1, \t_2]$. By standard theorems in real analysis \cite[Ths.~4.29-4.30]{rudin2008}, it has left and right limits at all points and at most countably many discontinuities, properties that transfer to $k'_\pm$ by continuity of the integral.

	Let $\t_1, \t_3\in\dom k$ such that $0 < \t_3 - \t_1 < \pi$, and let $\t_2\in (\t_1, \t_3)$. By \Cref{t:derivative} and \Cref{t:k'+intk},
	\begin{equation}\label{e:k'-1}
		\frac{k(\t_1) - k(\t_2) \cos(\t_1 - \t_2)}{\sin(\t_1 - \t_2)} \leq k'_\pm(\t_2) + \int_ {\t_2}^{\t_2} k(\omega) \d\omega\leq k'_-(\t_3) + \int_ {\t_2}^{\t_3} k(\omega) \d\omega,
	\end{equation}
	and
	\begin{equation}\label{e:k'+1}
		k'_+(\t_1) + \int_ {\t_2}^{\t_1} k(\omega) \d\omega\leq k'_\pm(\t_2) + \int_ {\t_2}^{\t_2} k(\omega) \d\omega\leq\frac{k(\t_3) - k(\t_2) \cos(\t_3 - \t_2)}{\sin(\t_3 - \t_2)}.
	\end{equation}
	As $\t_2\to\t_3^-$ in \cref{e:k'-1}, the middle term increases, so the left limit exists and as $\t_2\to\t_1^+$ in \cref{e:k'+1}, the middle term decreases, so the right limit exists. After applying the respective limits, we obtain, by continuity of $k$
	\begin{equation}\label{e:k'-2}
		\frac{k(\t_1) - k(\t_3) \cos(\t_1 - \t_3)}{\sin(\t_1 - \t_3)} \leq k'_\pm(\t_3^-) \leq k'_-(\t_3),
	\end{equation}
	and
	\begin{equation}\label{e:k'+2}
		k'_+(\t_1) \leq k'_\pm(\t_1^+) \leq\frac{k(\t_3) - k(\t_1) \cos(\t_3 - \t_1)}{\sin(\t_3 - \t_1)}.
	\end{equation}
	Now we respectively let $\t_1\to\t_3^-$ and $\t_3\to\t_1^+$ in \cref{e:k'-2,e:k'+2} to find
	\[
	k'_-(\t_3)\leq k'_\pm(\t_3^-) \leq k'_-(\t_3) \quad\text{and}\quad k'_+(\t_1) \leq k'_\pm(\t_1^+) \leq k'_+(\t_1),
	\]
	hence the equalities.
\end{proof}

\Cref{c:k'inc} suggests the following second-order characterization, which is most naturally stated in the language of distributions.

\begin{theorem}\label{t:k''}
	A trigonometric-convex function $k:\Theta\to\R\cup\{\infty\}$ satisfies
	\begin{equation}\label{e:k''}
		k''(\t) + k(\t) \geq 0,
	\end{equation}
	on the interior of $\dom k$ and in the distributional sense.
\end{theorem}

This result is suggested by the work of Meissner \cite{meissner1909} and Bonnesen \cite{bonnesen1920} on support functions. In that setting, the quantity $k''+k$ represents a \emph{radius of curvature}, which is necessarily nonnegative along the boundary of a convex body. \Cref{e:k''} was also observed by Levin \cite{levin1964}, although he did not use the formalism of distributions.

Apparently independently of the Phragmén--Lindelöf indicator and of the later terminology of trigonometric convexity, Carlson introduced what he called $l$-functions in order to control indicators in the study of entire functions \cite{carlson1914,carlson1921}. His axioms may look somewhat artificial from the analytic side: the functions were required to be twice differentiable and to satisfy the strict form of \eqref{e:k''}. From the present viewpoint, this condition is exactly the curvature inequality behind trigonometric convexity. Carlson was therefore very close to the same concept, reached through a different route.

\begin{proof}
    By \Cref{t:derivative}, $k$ admits finite left and right derivatives on the interior of $\dom k$, and these derivatives are locally integrable there. Hence $k'_\pm$ have a distributional derivative $k''$. By \Cref{c:k'inc}, the function defined in \eqref{e:k'inc} is nondecreasing. Therefore, its distributional derivative, which is $k''+k$, is nonnegative on the interior of $\dom k$.
\end{proof}

A converse follows under additional conditions on the regularity of $k$ and \Cref{c:k'inc}.

\subsection{Correspondence between convex and trigonometric convex functions}\label{s:cvxtcvx}

Following Rockafellar \cite[Sec.~4]{rockafellar1970}, we define convex functions $g$ by means of their \emph{epigraph}
\[
\epi g := \cond{(z, x)\in\C\times\R}{g(z) \le x}.
\]

\begin{definition}[Convex function]\label{d:cvxf}
	$g:\C\to\bar\R$ is said to be \emph{convex} if $\epi g$ is a convex set.
\end{definition}

The familiar characterization of convex functions with values in $\bar\R$ follows.

\begin{proposition}\label{p:cvxf}
	$g:\C\to\bar\R$ is convex if and only if for all $z_1, z_2\in\dom g$ and $t\in [0, 1]$,
	\begin{equation}\label{e:cvxf}
		g(tz_1 + (1-t)z_2) \le t g(z_1) + (1-t)g(z_2).
	\end{equation}
\end{proposition}

\begin{proof}
    The convexity of the epigraph of $g$ means that, for all $t\in[0, 1]$, $g(tz_1 + (1-t)z_2) \le t x_1 + (1-t) x_2$ whenever $g(z_1) \le x_1\in\R$ and $g(z_2) \le x_2\in\R$. For $z_1,z_2\in\dom g$, one may take $x_j\downarrow g(z_j)$ for $j=1,2$; this yields \cref{e:cvxf}. Conversely, if \cref{e:cvxf} holds then the epigraph is closed under convex combinations, and with the convention $0\cdot\infty = 0$ the set $\epi g$ is convex.
\end{proof}

\begin{definition}[Positive homogeneity]\label{d:ph}
	$g:\C\to\bar\R$ is said to be \emph{positive homogeneous} if for all $\lambda>0$ and $z\in\C$, $g(\lambda z) = \lambda g(z)$.
\end{definition}

The combination of convexity and positive homogeneity gives a different natural formulation of trigonometric convexity.

\begin{theorem}\label{t:cvxtcvx}
	The function $-\infty\not\equiv k : \Theta \to\bar\R$ is trigonometric-convex if and only if there exists a (unique) positive homogeneous convex function $g:\C\to\bar\R$ such that $g(0) = 0$ and $k(\t) = g(e^{i\t})$ for all $\t\in\Theta$.
\end{theorem}

\begin{proof}[of \Cref{t:cvxtcvx}]
	\begin{itemize}
		\item[( $\isimp$ )] \makeatletter\@itemdepth=0\makeatother
        Let $g$ be positive homogeneous and convex such that $g(0) = 0$ and $k(\t) = g(e^{i\t})$. We prove that $k$ is trigonometric-convex through the characterization of \Cref{p:diameter}. Let $\t_1, \t_2, \t_3 \in \Theta$ satisfy \cref{e:thetaimp2}, and set
        \begin{equation}\label{e:t}
            t := \frac{\sin(\t_3 - \t_2)}{\sin(\t_3 - \t_2) + \sin(\t_2 - \t_1)} \in [0, 1], \quad 1 - t = \frac{\sin(\t_2 - \t_1)}{\sin(\t_3 - \t_2) + \sin(\t_2 - \t_1)}
        \end{equation}
        and, using \Cref{p:sincos} with $s(\t) = e^{i\t}$, we compute
        \begin{align}
            & e^{i\t_2} \sin(\t_3 - \t_1) = e^{i\t_1}\sin(\t_3 - \t_2) + e^{i\t_3}\sin(\t_2 - \t_1) \nonumber \\
            ={}& (te^{i\t_1} + (1-t) e^{i\t_3})(\sin(\t_3 - \t_2) + \sin(\t_2 - \t_1)). \label{a:t}
        \end{align}
        Therefore, by positive homogeneity and convexity in the sense of \Cref{p:cvxf},
        \begin{align*}
        & k(\t_2)\sin(\t_3 - \t_1) = g(e^{i\t_2} \sin(\t_3 - \t_1)) \\
        ={}& g((te^{i\t_1} + (1-t) e^{i\t_3})(\sin(\t_3 - \t_2) + \sin(\t_2 - \t_1))) \\
        ={}& g(te^{i\t_1} + (1-t) e^{i\t_3})(\sin(\t_3 - \t_2) + \sin(\t_2 - \t_1)) \\
        \le{}& (t g(e^{i\t_1}) + (1-t) g(e^{i\t_3}))(\sin(\t_3 - \t_2) + \sin(\t_2 - \t_1)) \\
        &= k(\t_1)\sin(\t_3 - \t_2) + k(\t_3) \sin(\t_2 - \t_1), \tag{$*$}
        \end{align*}
        so the functional inequality \eqref{e:tcvximp} is satisfied. If $\t,\t+\pi\in\dom k$, then by \cref{e:cvxf}
        \[
        0 \le g(0) = g\pa{\inv2 e^{i\t} + \inv2 e^{i(\t+\pi)}} \le \inv2 k(\t) + \inv2 k(\t+\pi),
        \]
        which implies \cref{e:diameter}.
		\item[( $\imp$ )] \makeatletter\@itemdepth=0\makeatother
        Define $g(z) := |z| k(\arg z)$ for $z\neq0$ and $g(0) := 0$. Then $g$ is positive homogeneous and satisfies $k(\t) = g(e^{i\t})$. It remains to prove that $g$ is convex.

        Let $z_1,z_3\in\dom g$ and $t\in[0,1]$. If $z_1=0$ or $z_3=0$, the convexity inequality follows immediately from positive homogeneity and $g(0)=0$. We may therefore assume that $z_1,z_3\neq0$.

        Let $z_1 := |z_1| e^{i\t_1}, z_3 := |z_3| e^{i\t_3}$ be noncolinear, meaning that $\t_3 - \t_1 \notin \{0, \pi\}$. Without loss of generality, assume $0 < \t_3 - \t_1 < \pi$. Let $z_2 = |z_2| e^{i\t_2} := t z_1 + (1-t) z_3$.
        We compute
        \begin{align*}
            |z_2| \sin(\t_3 - \t_2)
            & = |z_2| (\sin(\t_3)\cos(\t_2) - \cos(\t_3)\sin(\t_2)) \\
            & = \sin(\t_3)\Re(z_2) - \cos(\t_3)\Im(z_2) \\
            & = \sin(\t_3)t\Re(z_1) - \cos(\t_3)t\Im(z_1) \\
            & = t|z_1| \sin(\t_3 - \t_1)
        \end{align*}
        and, similarly
        \[
        |z_2| \sin(\t_2 - \t_1) = (1-t)|z_3| \sin(\t_3 - \t_1).
        \]
        Since $\t_2\in[\t_1, \t_3]$, by \cref{e:tcvximp},
        \begin{align*}
        g(z_2)
        &= \frac{|z_2|}{\sin(\t_3 - \t_1)}k(\t_2)\sin(\t_3 - \t_1) \\
        &\le \frac{k(\t_1) |z_2|\sin(\t_3 - \t_2) + k(\t_3) |z_2|\sin(\t_2 - \t_1)}{\sin(\t_3 - \t_1)} \\
        &\quad= \frac{k(\t_1)t|z_1|\sin(\t_3 - \t_1) + k(\t_3)(1-t)|z_3| \sin(\t_3 - \t_1)}{\sin(\t_3 - \t_1)} \\
        &\quad= t g(z_1) + (1-t) g(z_3).
        \end{align*}
        It remains to verify the cases where $z_1$ and $z_3$ are colinear.
        
        Let $\lambda\in\R$ such that $z_3 = \lambda z_1$. Then $z_2 = (t + \lambda(1-t)) z_1$.
        \begin{itemize}
            \item If $\lambda > 0$, $g(z_2) = (t + \lambda (1-t))g(z_1) = t g(z_1) + (1-t)g(z_3)$.
            \item In the case $\lambda < 0$, set $\mu := t + \lambda(1-t)$.
            \begin{itemize}
                \item If $\mu = 0$, by \cref{e:diameter}
                \[
                g(z_2) = 0 \le t|z_1| (k(\t_1) + k(\t_1+\pi)) = t g(z_1) + (1-t)g(z_3).
                \]
                \item If $\mu > 0$, by \cref{e:k>-k}
                \begin{align*}
                    g(z_2)
                    &= \mu |z_1| k(\t_1) = t g(z_1) + \lambda(1-t)|z_1| k(\t_1) \\
                    &\le t g(z_1) - \lambda(1-t)|z_1| k(\t_1+\pi) = tg(z_1) + (1-t)g(z_3).
                \end{align*}
                \item If $\mu < 0$, by \cref{e:k>-k}
                \begin{align*}
                    g(z_2)
                    &= -\mu |z_1| k(\t_1+\pi) = -t |z_1| k(\t_1+\pi) - \lambda(1-t)|z_1| k(\t_1+\pi) \\
                    &\le t |z_1| k(\t_1) + (1-t)g(\lambda z_1) = tg(z_1) + (1-t)g(z_3).
                \end{align*}
            \end{itemize}
        \end{itemize}
        Finally, uniqueness follows from positive homogeneity for $z\neq0$ and from the normalization $g(0)=0$.
	\end{itemize}
\end{proof}

\begin{example}
    We now describe the natural applications of \Cref{t:cvxtcvx} that arise in this paper.
    \begin{enumerate}
        \item By \Cref{p:supptcvx}, the support function $\k_U$ of a convex subset $U$ of $\bar\C$ is trigonometric-convex. In view of \Cref{t:cvxtcvx}, this is equivalent to the convexity of the function $\Kappa_U$ defined in \cref{e:Kappa}, a classical fact \cite{rockafellar1970}.
        \item Let $\phi\in\E_3$ be an entire function. Among the positive-homogeneous convex functions associated with $h_\phi$ via \Cref{t:cvxtcvx}, the most natural one is
        \begin{equation}
            \H_\phi(s) := \limsup_{r\to\infty} \frac{\log\abs{\phi(rs)}}{r}.
        \end{equation}
        This is the so-called \emph{radial indicator} \cite{lelong1968}. It satisfies
        \begin{equation}
            \H_\phi(s) =
            \begin{cases}
                h_\phi(\arg s) \abs{s} &\com{if} s \neq 0 \\
                0 &\com{if} s = 0 \com{and} \phi(0) \neq 0 \\
                -\infty &\com{if} s = 0 \com{and} \phi(0) = 0,
            \end{cases}
        \end{equation}
        in analogy with \cref{e:Kk}. The function $\H_\phi$ is the form of the indicator used in several complex variables \cite{lelong1968}.
    \end{enumerate}
\end{example}

\section{Properties of support functions}\label{s:propsupport}

This appendix contains the support-function facts used in the main text. They are standard in finite-dimensional convex analysis, but $\bar\C$ introduces purely infinite convex sets and requires a few endpoint cases. The proofs below keep those cases explicit.

\begin{proof}[of \Cref{l:kless}]
    The first assertion is clear from \eqref{e:kdef}. For the converse, the cases $V = \bar\C$ and $V = \emptyset$ are immediate. Otherwise, we prove the contrapositive: if $U \not\subset V$, then there is a $\t$ for which $\k_U(\t) > \k_V(\t)$.
    \begin{itemize}
        \item If $U \subset C_\infty$, let $(\infty + iy_0)e^{i\t_0} \in U \setminus V$. By \Cref{p:clcvxinf}, the set $\cond{\arg p}{p\in C_\infty\setminus V}$ contains an interval $\pa{\t - \frac{\pi}{2}, \t + \frac\pi2}$, for some $\t\in\Theta$, that contains $\t_0$. Since $\t \in \pa{\t_0 - \frac{\pi}{2}, \t_0 + \frac\pi2}$, we have $\infty = \k_U(\t) > \k_V(\t)$.
        \item Otherwise, by \Cref{t:cvxinf}, $\k_U$ never assumes the value $-\infty$, and we may choose $z_0\in \C\cap(U\setminus V)$.
        \begin{itemize}
            \item If $V\subset C_\infty$, then, since $V$ is closed and convex, $V$ and $\k_V$ have the forms described in \cref{e:cvxinf,e:cvxinfk}, with $\beta - \alpha < \pi$. It follows that $\k_U(\t) > \k_V(\t) = -\infty$ for $\t\in\pa{\beta + \frac\pi2, \alpha - \frac\pi2}$.
            \item Assume now that $V\not\subset C_\infty$. Then the projection $z_p\in\C\cap V$ of $z_0$ onto $V$ satisfies the separation inequality \cite[Thm.~3.1.1]{hiriart-urruty2001}
            \begin{equation}\label{e:separation}
                \Re((z-z_p)\overline{(z_0-z_p)}) = \scal{z-z_p, z_0-z_p} \le 0
            \end{equation}
            for all $z\in V$. Let $\t := \arg(z_0 - z_p)$. Dividing \cref{e:separation} by $|z_0 - z_p| > 0$, we obtain $\Re(z_p e^{-i\t}) \ge \Re(z e^{-i\t})$. Taking the supremum over $z\in V$ gives $\Re(z_p e^{-i\t}) \ge \k_V(\t)$. Finally,
            \[
            \k_U(\t) \ge \Re(z_0 e^{-i\t}) = |z_0 - z_p| + \Re(z_p e^{-i\t}) > \Re(z_p e^{-i\t}) \ge \k_V(\t).
            \]
        \end{itemize}
    \end{itemize}
\end{proof}

\begin{proof}[of \Cref{p:supptcvx}]
	If $U = \emptyset$, $\k_U \equiv -\infty$ is trigonometric-convex by \Cref{d:tcvx}. Else, let $\t_1, \t_2, \t_3\in\dom \k_U$ such that \cref{e:thetaimp} is satisfied and $z\in U$. For the function $s(\t) := e^{-i\t} = \cos\t - i \sin\t$, we apply \Cref{p:sincos} in the following: 
		\begin{align*}
			& \Re(z e^{-i\t_2}) \sin(\t_3 - \t_1) = \Re(z e^{-i\t_2} \sin(\t_3 - \t_1)) \\
			={} & \Re(z (e^{-i\t_1}\sin(\t_3 - \t_2) + e^{-i\t_3}\sin(\t_2 - \t_1)))        \\
			={} & \Re(z e^{-i\t_1})\sin(\t_3 - \t_2) + \Re(z e^{-i\t_3})\sin(\t_2 - \t_1)   \\
			\le{} & \k_U(\t_1)\sin(\t_3 - \t_2) + \k_U(\t_3)\sin(\t_2 - \t_1).
		\end{align*}
		Taking the supremum over $z\in U$ yields \cref{e:tcvximp} with $k = \k_U$.
\end{proof}

\begin{proof}[of \Cref{t:tcvxK}]
	The uniqueness of $K$ follows from \Cref{l:kless}.
	\begin{itemize}
		\item[( $\isimp$ )] It follows from \Cref{p:supptcvx} with $U = K$.
		\item[( $\imp$ )] \makeatletter\@itemdepth=0\makeatother
        Let
		\begin{equation}\label{e:Kproof}
            K := \bigcap_{\t\in\Theta} \Pi_{k(\t), \t} = \cond{z\in\bar\C}{\forall\t\in\Theta,\; \Re(ze^{-i\t}) \leq k(\t)}
        \end{equation}
		which is closed and convex. We prove that $\k_K = k$. If $k\equiv\pm\infty$, then $K = \bar\C$ or $K = \emptyset$, and $\k_K = k$. Otherwise, $K$ could still be empty, but we show that this is not the case. Since $\dom k\neq\emptyset$, \Cref{t:derivative} gives, for all $\t\in\Theta$ and $\omega\in\dom k$,
		\[
		k(\t) \geq k'(\omega)\sin(\t - \omega) + k(\omega)\cos(\t - \omega) = \Re((k(\omega) + i k'(\omega))e^{i\omega}e^{-i\t}),
		\]
		where $k'(\omega)\in [k'_-(\omega), k'_+(\omega)]$. By \cref{e:Kproof}, this means that $z(\omega) := (k(\omega) + ik'(\omega))e^{i\omega}\in K$. Thus $K\neq\emptyset$. Moreover, for all $z\in K$, we have $\Re(ze^{-i\t}) \leq k(\t)$ for all $\t$, so $\k_K(\t) \le k(\t)$. Conversely, for all $\t\in\dom k$, $k(\t) = \Re(z(\t)e^{-i\t}) \leq \k_K(\t)$. If $k$ is unbounded, it remains to show that $\k_K(\t) = \infty$ for $\t\notin\dom k$. This follows from the following cases:
		\begin{itemize}
			\item If $\dom k = \{\alpha, \alpha + \pi\}$, $(\pm i\infty - y)e^{i\alpha}\in K$ for $y\in [-k(\alpha), k(\alpha+\pi)]$.
			\item If $\oline{\dom k} = [\alpha, \beta]$ with $\beta - \alpha \leq\pi$, let $y\in\R$.
			\begin{itemize}
				\item If $\beta - \alpha = \pi$, $(-i\infty + y)e^{i\alpha}\in K$ when $-k(\alpha) \leq y \leq k(\beta)$.
				\item If $\beta - \alpha < \pi$
				\begin{itemize}
					\item $(+i\infty - y)e^{i\beta}\in K$ when $y \leq k(\beta)$.
					\item $(\infty + iy)e^{i\t}\in K$ for $\t\in\pa{\beta + \frac\pi2, \alpha - \frac\pi2}$.
					\item $(-i\infty + y)e^{i\alpha}\in K$ when $-k(\alpha) \leq y$.
				\end{itemize}
			\end{itemize}
		\end{itemize}
		According to \Cref{p:domk}, these are all the possible cases for $\dom k$. We prove only the first case, since the others are similar. We need to show that
		\begin{equation}\label{e:alpha}
			\Re((\pm i\infty - y)e^{i\alpha}e^{-i\t}) = \pm\infty\sin(\t - \alpha) - y\cos(\t - \alpha) \leq k(\t)
		\end{equation}
		for all $\t$, where $y\in [-k(\alpha), k(\alpha+\pi)]$. There are two cases:
		\begin{itemize}
			\item $\t = \alpha$: \cref{e:alpha} becomes $-y \leq k(\alpha)$, which holds thanks to \cref{e:diameter}.
			\item $\t = \alpha+\pi$: \cref{e:alpha} becomes $y \leq k(\alpha+\pi)$, which is true.
			\item $\t\notin\dom k$: \cref{e:alpha} is trivial.
		\end{itemize}
		Hence, $(\pm i\infty - y)e^{i\alpha}\in K$ and $\Re((\pm i\infty - y)e^{i\alpha}e^{-i\t}) \leq \k_K(\t)$, which shows that $k(\t) = \k_K(\t) = \infty$ for $\t\notin\dom k$.
	\end{itemize}
\end{proof}

For a compact convex set $K$ with finite support function, the preceding proof recovers the usual normal parametrization of the boundary. When $\kappa_K$ is differentiable at $\t$, the boundary point whose outward normal direction is $e^{i\t}$ is
\[
(\kappa_K(\t) + i\kappa_K'(\t))e^{i\t},
\]
as in Blaschke's treatment of support functions \cite[\S 9]{blaschke1914}. If $\kappa_K$ is not differentiable at $\t$, then the admissible values of $\kappa_K'(\t)$ fill the interval $[(\kappa_K)_-'(\t),(\kappa_K)_+'(\t)]$. The corresponding points
\[
(\kappa_K(\t) + iy)e^{i\t}, \quad y\in[(\kappa_K)_-'(\t),(\kappa_K)_+'(\t)],
\]
form the boundary segment of $K$ with outward normal direction $e^{i\t}$.

The quantity
\[
\rho_K(\t) := \kappa_K''(\t) + \kappa_K(\t)
\]
is the radius of curvature in the smooth case. More generally, $\kappa_K''+\kappa_K$ is a nonnegative measure; this is another expression of the convexity of $K$ \cite{meissner1909,bonnesen1920}. In the smooth case,
\[
\kappa'_K(\t) - \kappa'_K(\t_0) + \int_{\t_0}^\t\kappa_K(\omega) \d\omega
\]
is the arc length between the boundary points with normal directions $e^{i\t_0}$ and $e^{i\t}$, which explains the monotonicity appearing in \Cref{c:k'inc}. In particular, the perimeter of $K$ is
\begin{equation}
	P(K) = \int_0^{2\pi} \kappa_K(\t) \d\t.
\end{equation}
The area of $K$ can also be expressed as follows
\begin{equation}
	A(K) = \inv{2} \int_0^{2\pi} \kappa_K(\t) \rho_K(\t) \d\t = \inv{2} \int_0^{2\pi} (\kappa_K(\t)^2 - \kappa'_K(\t)^2) \d\t,
\end{equation}
see \cite{blaschke1914,bonnesen1920}.

\printbibliography[notkeyword={cvxtopo}, title={References}]
\defbibnote{otherRefs}{See also \cite{valiron1932,phragmen1908,carlson1914,carlson1921,boller1932,gelfond1938,chebotarev1949,boas1954,cartwright1956,levin1964,levin1996,maergoiz2003}.}
\printbibliography[keyword=cvxtopo, title={References on convex analysis and topology}, postnote=otherRefs]

\end{document}